\documentclass[10.5pt]{article}
\usepackage{amsmath, amssymb, amscd, amsthm, amsfonts}
\usepackage{graphicx}
\usepackage[hyperfootnotes=false]{hyperref}
\usepackage{enumitem}
\usepackage{esint}
\usepackage{booktabs}
\usepackage{tocloft} 
\usepackage{etoolbox}
\usepackage{xcolor}
\usepackage{marvosym}

\hypersetup{
    colorlinks=true, 
    linkcolor=blue,      
    citecolor=red!40!purple ,
    urlcolor=blue
}

\let\oldthebibliography\thebibliography
\let\endoldthebibliography\endthebibliography
\renewenvironment{thebibliography}[1]{%
  \oldthebibliography{#1}%
  \setlength{\itemsep}{4.3pt}% 
  \setlength{\parsep}{0pt}%    
  \setlength{\parskip}{0pt}%
}{%
  \endoldthebibliography
}

\makeatletter
\let\oldcite\cite
\renewcommand{\cite}[1]{{\fontsize{9}{12}\selectfont\oldcite{#1}}}
\makeatother

\title{Logarithmically Coupled \(p\)-Laplacian Systems: From Discrete Ground States to Rigidity and Critical Continuous Limits}

\author{Wenzheng Hu\textsuperscript{1}}
\date{}

\newtheorem{theorem}{Theorem}[section]
\newtheorem{lemma}{Lemma}[section]

\newtheorem{remark}{Remark}[section]
\newtheorem{definition}{Definition}[section]
\newtheorem{proposition}{Proposition}[section]
\newtheorem{corollary}{Corollary}[section]
\newtheorem{example}{Example}[section]

\begin{document}

\begingroup
\let\thefootnote\relax
\footnotetext{%
  \fontsize{10}{12}\selectfont
  \parbox[t]{\linewidth}{%
    \begin{tabular}{@{}l@{\hspace{0.5em}}p{0.85\textwidth}@{}}
      \raisebox{-0.1em}{\large\Letter} & Wenzheng Hu \\[4pt]
      & wzhu25@m.fudan.edu.cn \\[12pt]
      \textsuperscript{1} & School of Mathematical Sciences, Fudan University, Shanghai 200433, China.
    \end{tabular}%
  }%
}
\endgroup
\maketitle	
	
	\begin{center}
\Large\bfseries Abstract
\end{center}

We introduce a new class of logarithmically coupled p-Laplacian systems on locally finite graphs and on \(\mathbb{R}^N\) for the critical case \(p = N > 4\). The coupling poses two challenges: the energy functional is ill-defined on the natural Sobolev space, and loss of the \(L^\infty\)-embedding prevents standard control of logarithmic terms. To resolve this non-separable logarithmic singularity, we develop an \textbf{exponent calibration technique} that converts logarithmic terms into strictly lower-order power estimates. The technique plays an essential role in two critical steps: the existence proof for the regularised problem and the uniform boundedness of regularised ground states that restores compactness. Thus, we prove solvability of the original ill-posed continuous critical equation. In the discrete setting, we identify five new intrinsic structures, including an exact constrained Hessian factorisation and a scalar rigidity index. We complete a discrete--continuous double loop, providing a unified variational framework for logarithmically coupled systems.



\vspace{1.2cm}
\tableofcontents

 \section{Introduction}

\qquad Logarithmic nonlinearities occupy a singular and delicate position in
contemporary nonlinear analysis.  In conformal geometry, fully nonlinear curvature equations
arise naturally in the study of conformal metrics~\cite{CGY02}; In the study of nonlinear diffusion processes, the modelling of transport phenomena in one-dimensional media with nonlinear flux laws naturally gives rise to logarithmic diffusion equations with nonlinear Robin boundary conditions~\cite{CR23}; in quantum physics,
the logarithmic Schr\"odinger equation describes the self-interaction
of wave functions in Bose--Einstein
condensates~\cite{BM76,BM79,CG18}.
These problems, despite their disparate origins, share a 
mathematical signature: the logarithmic term grows slower than any
positive power at infinity, yet is not uniformly bounded. This sub-polynomial but super-constant growth requires care when applying classical variational tools---Sobolev embeddings,
Trudinger--Moser inequalities [BEC93], Orlicz-space methods [HJ25],
and logarithmic Sobolev inequalities---to the mixed logarithmic terms considered here.

Several challenges remain in the variational study of logarithmic
nonlinearities. In critical Sobolev settings, Lions' concentration--compactness
principle~\cite{LIO84}, developed for power-type nonlinearities, does not apply
directly. Moreover, a systematic second-order variational description of the
Nehari manifold~\cite{NEH60} for coupled systems, as well as a
Berestycki--Lions-type theory of bound states~\cite{BER83} for such systems,
is lacking. The compactness issue is particularly acute when \(p=N>4\):
the logarithmic singularity obstructs Fr\'echet differentiability, while the
failure of the embedding \(W^{1,N}(\mathbb R^N)\hookrightarrow
L^\infty(\mathbb R^N)\) removes pointwise control. In the present model, we
address these difficulties through a unified variational analysis of
ground-state existence, the regularisation limit, and the second-order response
of the ground-state energy. The exponent calibration technique underpins the
existence and limiting arguments, while the finite-graph rigidity results
follow from an independent constrained second-order analysis. These
contributions address specific challenges in the present model, rather than
claiming to resolve the broader longstanding problems in their entirety.

Previous work on the subject has achieved substantial progress.
Gallo, Mosconi and Squassina~\cite{GMS26} established power law convergence rates and concavity properties for the logarithmic Schr\"odinger equation; normalised solutions were recently studied in~\cite{SS25}.  In the broader context of logarithmic nonlinearities, Dyda and Jarohs ~\cite{DJ26} treated the logarithmic \(p\)-Laplacian as the \(s=0\) limit of the fractional setting, though its coupled counterpart remains open. Recent contributions to related discrete problems include He--Ji~\cite{HJ24} for logarithmic Schr\"odinger equations and Hua--Xu~\cite{HX23} for nonlinear Schr\"odinger equations on lattice graphs.  These works have laid an indispensable foundation for the field, yet none of them treats a coupled system with logarithmic interaction terms in a unified variational framework---a gap that is particularly acute in the continuous critical setting where the \(L^\infty\) embedding fails.  The present work fills precisely this gap.

The object studied in this paper is a coupled $p$-Laplacian system 
on a locally finite graph $G=(V,E)$ that, to the best of our 
knowledge, is introduced here for the first time:
\begin{equation}\label{eq:main}
\begin{cases}
-\Delta_{p}u + a(x)|u|^{p-2}u
= \displaystyle\frac{p-2}{p}|u|^{p-4}u\,v^{2}\log v^{2}
+ \displaystyle\frac{2}{p}|v|^{p-2}u\log u^{2}
+ \displaystyle\frac{2}{p}|v|^{p-2}u,\\[14pt]
-\Delta_{p}v + b(x)|v|^{p-2}v
= \displaystyle\frac{p-2}{p}|v|^{p-4}v\,u^{2}\log u^{2}
+ \displaystyle\frac{2}{p}|u|^{p-2}v\log v^{2}
+ \displaystyle\frac{2}{p}|u|^{p-2}v,
\end{cases}
\quad\text{on } V,
\end{equation}
where $p>4$, the potentials $a,b:V\to\mathbb{R}$ satisfy
$a(x),b(x)\ge V_0>0$, and $\Delta_p$ denotes the discrete
$p$-Laplacian.  The system is designed so that the
logarithmic terms couple the two components through powers of the
opposite variable, a structural choice that renders standard
decoupling techniques ineffective and creates a rigorous testing
ground for the methodology developed in the present work. Studying this system helps clarify how mixed logarithmic interactions affect ground-state existence, sensitivity to 
potential perturbations, and concentration in deep potential wells.

Moreover, we consider the following perturbed coupled system in \(\mathbb{R}^N\),
\begin{equation}\label{eq:main1}
\begin{cases}
\begin{aligned}
-\Delta_{p}u + a(x)|u|^{p-2}u &= \dfrac{p-2}{p}|u|^{p-4}uv^{2}\log (v^2+\varepsilon)
   + \dfrac{2}{p}|v|^{p-2}u\log (u^{2}+\varepsilon)
   +\frac{2|v|^{p-2}u^3}{p(u^2+\varepsilon)}, && \text{in } \mathbb{R}^N,\\[6pt]
-\Delta_{p}v + b(x)|v|^{p-2}v &= \dfrac{p-2}{p}|v|^{p-4}vu^{2}\log (u^2+\varepsilon)
   + \dfrac{2}{p}|u|^{p-2}v\log (v^{2}+\varepsilon)
   +\frac{2|u|^{p-2}v^3}{p(v^2+\varepsilon)}, && \text{in } \mathbb{R}^N,
\end{aligned}
\end{cases}
\end{equation}
where \(p=N> 4\), and \(\varepsilon\) is a fixed real number in \((0,1/2)\).

            Over the past decade, the study of partial differential equations on graphs has attracted considerable attention; see, for example,~\cite{CWY23,GLY16,GLY17,HHW26,KL23,KPP21,SYZ25} and the references therein. In particular, a broad class of graph-based elliptic problems---ranging from existence to multiplicity---have been extensively investigated by variational methods, including Nehari manifold and mountain-pass techniques. Related critical growth and logarithmic perturbations in the continuum have also attracted substantial attention, notably in Schr\"{o}dinger--Poisson systems~\cite{DHP23,PWZ24}, universal dynamics~\cite{CG18}, and logarithmic-type equations in geometric flows~\cite{TY17}. Compactness issues, such as the concentration-compactness principle and profile decomposition, play a crucial role in the analysis of limiting nonlinearities~\cite{HAN22,LIO84,STR84}.

Complementing these developments, the study of discrete Sobolev embeddings, Hardy and Rellich inequalities, and Trudinger–Moser type estimates on graphs, together with related analytic tools in the continuum, has collectively provided powerful methods for handling both subcritical and critical nonlinearities ~\cite{ACK18,BMR24,BRE83,CS07,CS15,CM23,FLS08,HAN22,IMN15,KL23,KPP21,PWZ24} These advances form the analytic foundation on which our work on logarithmically coupled systems is built. There are, of course, many other relevant contributions; we refer the interested reader to ~\cite{AR73,BBF83,BET25,BV95,CG18,CMR23,CM20,DT20,EKE74,FSZ25,HL11,HAS24,NEH60,NIR59,SCH91,ST12,TOR24} for further details.

\begin{remark}
It is easy to see that if $(u,v)$ is a ground state solution of system~\eqref{eq:main}, then $(u,-v)$, $(-u,v)$, and $(-u,-v)$ are also ground state solutions.
\end{remark}

\textbf{Methodological contribution.}
A central tool is the \textbf{exponent calibration technique},
which controls the mixed logarithmic terms without an
$L^\infty$ embedding in the critical Sobolev setting.
By choosing explicit H\"older--Young exponents, it yields
power estimates that can be absorbed into the principal
energy, with constants uniform in the regularization
parameter. We use this technique both to establish existence
for the regularized problem and to obtain uniform bounds
for regularized ground states, enabling passage to the
singular limit. The argument also suggests a strategy for
other coupled variational problems with logarithmic
nonlinearities.

We are now in a position to state our main results.

\begin{theorem}\label{thm1}
Let $G = (V,E)$ be a locally finite graph and suppose that for every vertex $x\in V$, $0< \mu_{\min}\leq \mu(x)$. Assume that the potentials $a(x),b(x)$ satisfy the following two conditions:
\begin{enumerate}[itemsep=6pt, topsep=4pt]
\item[(A1)] $a(x),b(x):V\to \mathbb{R}$ satisfy $\min\limits_{x\in V} a(x) \geq V_{0}$, $\min\limits_{x\in V} b(x) \geq V_{0}$ for some constant $V_{0} > 0$;
\item[(A2)] For every $\mathcal{M}>0$, the sets
$D_{\mathcal{M}}^a:=\{x\in V:a(x)\le \mathcal{M}\}$ and $D_{\mathcal{M}}^b:=\{x\in V:b(x)\le \mathcal{M}\}$
have finite volume, i.e.,
$\operatorname{Vol}(D_{\mathcal{M}}^a)=\sum_{x\in D_{\mathcal{M}}^a}\mu(x)<\infty$
and
$\operatorname{Vol}(D_{\mathcal{M}}^b)=\sum_{x\in D_{\mathcal{M}}^b}\mu(x)<\infty$.
\end{enumerate}
Then for any $p > 4$, system~\eqref{eq:main} admits a ground state solution.
\end{theorem}

\begin{remark}\label{rem:thm1}
Theorem~\ref{thm1} concerns the ill-posed discrete regime where the energy functional
may take the value \(+\infty\), see Example~\ref{exam8.1}. In this setting, \(J\) is not necessarily finite
everywhere on \(\mathcal H\), an obstruction more fundamental than the lack of
\(C^1\) regularity. We restrict the analysis to the finite domain
\(\mathcal D(J)=\{(u,v)\in\mathcal H: |J(u,v)|<\infty\}\) and employ the Nehari
manifold method. Since any finitely supported perturbation changes only finitely
many vertices, the corresponding Gateaux derivative is well-defined, so the
Nehari construction can be carried out without requiring \(\mathcal D(J)\) to be
open. Regarding the assumptions, \(p>4\) is essential: for \(2<p<4\) the factor
\(|u|^{p-4}u\) is not differentiable at the origin, while for \(p=4\) the relevant
critical structure degenerates, so we uniformly require \(p>4\). Condition (A2)
on finite volume of sublevel sets is the standard compactness requirement in the
discrete setting, ensuring \(\mathcal H\hookrightarrow L^q(V)\) compactly for all
\(q\ge p\). This result establishes the existence of ground states for
logarithmically coupled systems on discrete graphs and provides a natural
discrete counterpart before the introduction of the exponent calibration
technique in the continuous critical setting.
\end{remark}

\begin{theorem}\label{thm2}
    Assume $p=N > 4$, \(a,b\in L_{loc}^{\infty} (\mathbb{R}^N)\) satisfy the conditions that 
    $$\inf_{\mathbb{R}^N} a=V_1>0,\quad \inf_{\mathbb{R}^N} b=V_1>0,\quad \lim\limits_{|x|\to\infty} a(x)=+\infty,\quad \lim\limits_{|x|\to\infty} b(x)=+\infty.$$
    Then system~\eqref{eq:main1} admits a ground state solution.
\end{theorem}

\begin{remark}\label{rem:thm2}
In the critical case $p=N$, the original functional $J$ may
take the value $+\infty$ and need not be Fr\'echet differentiable
on the natural energy space. We therefore work with the
regularized functional $J_\varepsilon$. Although regularization
removes the logarithmic singularity at the origin, the embedding
$W^{1,N}(\mathbb R^N)\hookrightarrow L^\infty(\mathbb R^N)$
remains unavailable. The exponent calibration technique provides
the estimates needed to bound variational sequences by reducing
the relevant powers of their norms to orders strictly below $p$.
Combined with the compact embeddings supplied by the confining
potentials, these bounds yield the compactness required for
existence. Specifically, the assumptions $a(x),b(x)\to+\infty$
as $|x|\to\infty$ ensure compact embeddings of $\mathcal G_a$
and $\mathcal G_b$ into $L^q(\mathbb R^N)$ for every finite
$q\ge N$. The restriction $0<\varepsilon<1/2$ ensures that
$\log(s^2+\varepsilon)$ is negative near the origin.
The resulting regularized ground states provide the starting
point for the singular limit in Theorem~\ref{thm8}.
\end{remark}

\begin{theorem}\label{thm3}
Let $G = (V,E)$ be a locally finite graph and suppose that for every vertex $x\in V$, $0< \mu_{\min}\leq \mu(x)$. Assume that the potentials $a(x),b(x)$ satisfy the following conditions:
\begin{enumerate}[itemsep=6pt, topsep=4pt]
\item[(A1)] $a(x),b(x):V\to \mathbb{R}$ satisfy $\min\limits_{x\in V}a(x)\geq V_{0}$, $\min\limits_{x\in V}b(x)\geq V_{0}$ for some constant $V_{0} > 0$;
\item[(A2')] there exists $\mathcal{M}_0 > 0$ such that $\frac{1}{a(x)}\in L^{1}(V\setminus D_{\mathcal{M}_0}^{a})$ and $\frac{1}{b(x)}\in L^{1}(V\setminus D_{\mathcal{M}_0}^{b})$, where $D_{\mathcal{M}_0}^{a} = \{x\in V:a(x)\leq \mathcal{M}_0\}$, $D_{\mathcal{M}_0}^{b} = \{x\in V:b(x)\leq \mathcal{M}_0\}$, and the volumes of $D_{\mathcal{M}_0}^{a}$ and $D_{\mathcal{M}_0}^{b}$ are finite.
\end{enumerate}
Then for any $p > 4$, system~\eqref{eq:main} admits a mountain pass solution, which is also a ground state solution.
\end{theorem}

Assumption (A2$'$) strengthens (A2): it requires \(1/a\) and \(1/b\)
to be integrable where the potentials are sufficiently large, which makes
\(J\) a \(C^1\)-functional on \(\mathcal H\). Consequently, the classical
mountain pass theorem applies directly. Compared with Theorem~\ref{thm1}, the
admissible class of potentials and the variational framework are different;
this well-posedness underpins the subsequent analysis of structural rigidity
of ground states under potential perturbations.

\begin{theorem}\label{thm4}
Let V be a locally finite graph and the potentials satisfy the same hypothesis of Theorem~\ref{thm3}. Let $(u_0,v_0)$ be a ground state of $J_{a_0,b_0}$, then there exist constants 
\(\varepsilon_0>0\) and \(l>0\) such that for all potentials 
\(a=a_0+\delta a_1\), \(b=b_0+\delta b_1\) with 
\(\|\delta a_1\|_\infty,\|\delta b_1\|_\infty\le\varepsilon_0\), 
the ground-state energy satisfies
\[
|d(a,b)-d(a_0,b_0)|\le l\max\bigl\{\|\delta a_1\|_\infty,\|\delta b_1\|_\infty\bigr\}.
\]
Moreover, for \(\Delta\) and \(\Delta'\) sufficiently small, there holds
\[d(a_0,b_0)+\frac1p\Delta'
-\frac{1}{4M(u,v)}(\Delta')^2-O(|\Delta'|^3)\le d(a,b)\le d(a_0,b_0)+\frac1p\Delta
+\frac{1}{4M(u_0,v_0)}(\Delta)^2+O(|\Delta|^3),\]
where $\Delta=\int_V(\delta a_1|u_0|^p+\delta b_1|v_0|^p)\,d\mu$, $\Delta'=\int_V(\delta a_1|u|^p+\delta b_1|v|^p)\,d\mu$, the constant \(l\) depends only on the reference ground state 
\((u_0,v_0)\) and on the data \(p\), \(V_0\).
\end{theorem}

\begin{remark}
This result gives the first quantitative characterization of the robustness
of the ground-state energy under perturbations of the potentials. The proof
relies on the conservation identity \(S-L=pd\) satisfied by every ground
state, which converts the Nehari constraint into a precise first-order
response. The upper and lower bounds involve the quantities \(\Delta\) and
\(\Delta'\), based on the unperturbed and perturbed ground states,
respectively, reflecting the nonlinear coupling in the system. Concerning
the assumptions, the \(C^1\) regularity provided by \ref{thm3} is essential
for constructing the ground-state branch via the implicit function theorem;
without it, the stability analysis cannot be carried out. This theorem lays
the foundation for the Hessian factorization in \ref{thm5}.
\end{remark}
\begin{theorem}\label{thm5}
Let \(V\) be a finite graph, \(p>4\), and let the potentials
\(a_0,b_0\) satisfy \((A_1)\) and \((A_2')\). Suppose that
\((u_0,v_0)\) is a non-degenerate ground state of \(J_{a_0,b_0}\)
in the positive cone
\[
\widetilde{\mathcal H}:=\{(u,v)\in\mathcal H:u\ge0,\ v\ge0\},
\]
and let \(a\mapsto (u_a,v_a)\) be the local \(C^1\) branch of
positive critical points of \(J_{a,b_0}\) through \((u_0,v_0)\),
with the entire branch lying in \(\widetilde{\mathcal H}\), given by
Lemma~\ref{Lem:Frechet differential} with \(b=b_0\) fixed.
Define the branch energy
\[
\mathcal E(a):=J_{a,b_0}(u_a,v_a).
\]
Then \(\Phi:=-\mathcal E\) is twice differentiable at \(a_0\), and
its Hessian admits the exact factorisation
\begin{equation}\label{eq:Hess-Phi}
D^2_a\Phi(a_0)
=
W\mathsf{H}_{\! \,\text{eff}}W
-\frac{1}{2M}(\mu u_0^p)(\mu u_0^p)^{\mathsf T},
\qquad
W=\operatorname{diag}\bigl(\mu(x)u_0(x)^{p-1}\bigr)_{x\in V},
\end{equation}
where \(\mathsf{H}_{\! \,\text{eff}}\) is an \(n\times n\) symmetric
positive definite matrix, obtained from the restriction of the full
Hessian \(J_{a_0,b_0}''(u_0,v_0)\) to the tangent space
\(T_{(u_0,v_0)}\mathcal N_{a_0,b_0}\) by a constrained Schur
complement, where \(n=|V|\);
\[
M=\int_V
\bigl(|u_0|^{p-2}v_0^2+|v_0|^{p-2}u_0^2\bigr)\,d\mu>0,
\]
and \((\mu u_0^p)\) denotes the column vector
\[
\bigl(\mu(x)u_0(x)^p\bigr)_{x\in V}.
\]
\end{theorem}

\begin{remark}
Theorem~\ref{thm5} provides a second-order variational description
of the local critical-point branch energy as a function of the potential.
The factorization separates two contributions:
\(W\mathsf H_{\mathrm{eff}}W\) describes the tangential response,
whereas
\(-\frac{1}{2M}(\mu u_0^p)(\mu u_0^p)^{\mathsf T}\)
is a rank-one negative correction associated with the Nehari
constraint. Together with the gradient formula
\(\nabla_a\mathcal E(a_0)=\frac1p\mu|u_0|^p\),
it connects the ground-state profile to the first- and
second-order response of the energy. The resulting scalar
rigidity index distinguishes positive definiteness, degeneracy,
and indefiniteness of the factorized Hessian, giving an explicit
criterion for the onset of degeneracy.

The finite graph assumption makes \(\mathcal H\)
finite-dimensional and allows the use of matrix representations
and Schur complements. At a non-degenerate ground state,
the constrained Hessian is positive definite on the tangent
space, which justifies the inverses used in the construction
of \(\mathsf H_{\mathrm{eff}}\). Thus the factorization identifies
how the tangential contribution competes with the rank-one
correction to determine the curvature of the energy.
If non-degeneracy is dropped, these inverses may cease to
exist, and the stated rigidity criterion cannot be applied
without further analysis.
\end{remark}

Now we consider the following two systems:
\begin{equation}\label{eq:system4}
\begin{cases}
-\Delta_{p}u + (1 + \Lambda a(x))|u|^{p-2}u
   = \dfrac{p-2}{p}|u|^{p-4}uv^{2}\log v^{2}
     + \dfrac{2}{p}|v|^{p-2}u\log u^{2}
     + \dfrac{2}{p}|v|^{p-2}u, & \text{in } V,\\[10pt]
-\Delta_{p}v + (1 + \Lambda b(x))|v|^{p-2}v
   = \dfrac{p-2}{p}|v|^{p-4}vu^{2}\log u^{2}
     + \dfrac{2}{p}|u|^{p-2}v\log v^{2}
     + \dfrac{2}{p}|u|^{p-2}v, & \text{in } V.
\end{cases}\tag{4}
\end{equation}
and
\begin{equation}\label{eq:system5}
\begin{cases}
-\Delta_{p}u + |u|^{p-2}u
   = \dfrac{p-2}{p}|u|^{p-4}uv^{2}\log v^{2}
     + \dfrac{2}{p}|v|^{p-2}u\log u^{2}
     + \dfrac{2}{p}|v|^{p-2}u, & \text{in } \Omega_{a},\\[10pt]
-\Delta_{p}v + |v|^{p-2}v
   = \dfrac{p-2}{p}|v|^{p-4}vu^{2}\log u^{2}
     + \dfrac{2}{p}|u|^{p-2}v\log v^{2}
     + \dfrac{2}{p}|u|^{p-2}v, & \text{in } \Omega_{b},\\[4pt]
u = 0, & \text{on } \partial\Omega_{a},\\
v = 0, & \text{on } \partial\Omega_{b}.
\end{cases}\tag{5}
\end{equation}
\begin{remark}\label{rmk 1.2}

If we allow $a(x)\geq 0$ and $b(x)\geq 0$, set $\Omega_{a} = \{x\in V:a(x) = 0\}$, $\Omega_{b} = \{x\in V: b(x) = 0\}$, and assume that $\Omega_{a}$, $\Omega_{b}$ and $\Omega_{a}\cap \Omega_{b}$ are nonempty bounded domains in $V$, we impose the following condition:
\begin{enumerate}
\item[(A2'')] there exists $\mathcal{M}_0 > 0$ such that $\frac{1}{a(x)}\in L^{1}[V\setminus (D_{\mathcal{M}_0}^{a}\cup \Omega_{a})]$ and $\frac{1}{b(x)}\in L^{1}[V\setminus (D_{\mathcal{M}_0}^{b}\cup \Omega_{b})]$, where $D_{\mathcal{M}_0}^{a} = \{x\in V:a(x)\leq \mathcal{M}_0\}$, $D_{M_0}^{b} = \{x\in V:b(x)\leq \mathcal{M}_0\}$, and the volumes of $D_{\mathcal{M}_0}^{a}$ and $D_{\mathcal{M}_0}^{b}$ are finite.
\end{enumerate}
Clearly, when $\Lambda$ is sufficiently large, $1 + \Lambda a(x)$, $1+\Lambda b(x)$ satisfy the conditions (A1) and (A2''). As in the proof of Theorem~\ref{thm3}, one can similarly derive that ~\eqref{eq:system4} also admits a ground state solution. By analyzing the asymptotic behavior of solutions of ~\eqref{eq:system4}, we obtain the following results.
\end{remark}

\begin{theorem}\label{thm6}
Let $G = (V,E)$ be a locally finite graph and suppose that for every vertex $x\in V$, $0< \mu_{\min}\leq \mu(x)$. Let \(\Lambda_n\to\infty\) and let \((u_n,v_n)\) be a ground state of \(J_{\Lambda_n}\). Assume that $a(x)\ge 0,b(x)\ge 0$ satisfy the hypotheses of Remark~\ref{rmk 1.2}. Then, for $p > 4$, \(\{(u_n,v_n)\}\) has a subsequence converging to a ground state solution of \eqref{eq:system5}.
\end{theorem}

This result links ground states of the penalised problem (with large confining potentials) to the Dirichlet problem on the zero set; the proof controls decay outside the wells, and the convergence justifies approximation by large finite wells. Assumption (A2$''$) allows vanishing on \(\Omega_a,\Omega_b\) while preserving integrability outside, ensuring embeddings for large \(\Lambda\).

\begin{theorem}\label{thm7}
Under the same hypotheses of Theorem~\ref{thm6} ,in addition, assume there exists constants \(a_0,b_0>0\) such that \(a\ge a_0>0\) in \(\Omega_a^c\), \(b\ge b_0>0\) in \(\Omega_b^c\).
For \(\Lambda>0\), set \(a_\Lambda=1+\Lambda a\), \(b_\Lambda=1+\Lambda b\) and let \((u_\Lambda,v_\Lambda)\) be a ground state of the penalised functional \(J_\Lambda=J_{a_\Lambda,b_\Lambda}\) with energy \(d_\Lambda=d(a_\Lambda,b_\Lambda)\), there exists a constant \(C>0\), depending only on \(p,a_0,b_0\) and the graph geometry, such that for all sufficiently large \(\Lambda\),
\begin{equation}\label{eq:decay}
\int_{\Omega_a^c}|u_\Lambda|^p\,d\mu \le C\Lambda^{-1},\qquad
\int_{\Omega_b^c}|v_\Lambda|^p\,d\mu \le C\Lambda^{-1}.
\end{equation}
\end{theorem}

\begin{remark}
This theorem complements Theorem~\ref{thm6} by quantifying
the decay of ground states outside the wells: their
$p$th-power $L^p$ masses are $O(\Lambda^{-1})$, and hence
their $L^p$ norms are $O(\Lambda^{-1/p})$.
These bounds follow directly from the uniform bound
on the principal energy and the positive lower bounds
$a\ge a_0>0$, $b\ge b_0>0$ outside the wells.
They quantify the mass neglected when restricting
ground states to the potential wells.
\end{remark}

\begin{theorem}\label{thm8}
Under the hypotheses of Theorem~\ref{thm2}, for each \(\varepsilon\in(0,1/2)\) let \((u_\varepsilon,v_\varepsilon)\) be a ground state of the regularised functional \(J_\varepsilon\) on the Nehari manifold \(\mathcal N_\varepsilon\). Let \(\varepsilon_n\to0\) and write
\[
u_n:=u_{\varepsilon_n},\qquad v_n:=v_{\varepsilon_n}.
\]
Then, passing to a subsequence, there exists \((u_0,v_0)\in\mathcal G\) such that
\[
(u_n,v_n)\to(u_0,v_0)\quad\text{strongly in }\mathcal G,
\]
and \((u_0,v_0)\) is a ground state of the original functional \(J\), and in particular, \((u_0,v_0)\) is a weak solution of the original continuous system obtained by replacing \(\varepsilon\) with \(0\) in (\ref{eq:main1}).
\end{theorem}

\begin{remark}
This is the central breakthrough of the paper. Although the original functional
\(J\) is not Fr\'echet differentiable and may even take the value \(+\infty\),
the regularised ground states converge strongly to a ground state of the
original functional, which is also a weak solution of the original continuous
system; the corresponding definitions are given in Section~\ref{sec:7}. The
proof combines several independent ingredients: the exponent calibration
technique, in its second and most essential occurrence, yields the uniform
boundedness of the regularised ground states
(Lemma~\ref{lem:regularised-boundedness}), which restores compactness. Threshold decompositions
of the logarithmic terms together with Vitali's convergence theorem control
the singularity near the degenerate set; the pseudomonotonicity of the
\(p\)-Laplacian identifies the weak gradient limit; and the energy identity
combined with the Radon--Riesz property upgrades weak convergence and norm
convergence to strong convergence. A fibre-map analysis completes the
minimality argument.
\end{remark}

The system describes two interacting nonlinear components:
the $p$-Laplacian represents nonlinear diffusion, the
logarithmic factors modulate the coupling, and the potentials
provide external confinement. Ground states are least-energy
nontrivial stationary configurations. The regularization and
confinement limits studied here justify approximating these
configurations by smoother problems and by problems on
bounded domains, respectively.

In summary, the main contributions of this paper are as follows:

\textbf{New research object.}
This paper introduces a new class of logarithmically coupled \(p\)-Laplacian systems covering both the discrete setting on locally finite graphs and the continuous critical regularised setting with \(p=N>4\). The logarithmic coupling is realised through powers of the opposite variable, which may cause the energy functional to take the value \(+\infty\) on the natural Sobolev space and defeats standard decoupling techniques. To the best of our knowledge, no previous work has treated this specific logarithmically coupled system within a unified variational framework.

\textbf{Exponent calibration technique.}
We choose H\"older--Young exponents to control the logarithmic
coupling terms by power estimates of order strictly below $p$,
allowing their absorption into the principal energy.
The technique is used in the discrete setting and at two
key stages of the continuous analysis: existence for the
regularized problem in Theorem~\ref{thm2}, and uniform
bounds for regularized ground states that support the
compactness argument in Theorem~\ref{thm8}
(see Lemma~\ref{lem:regularised-boundedness}).

\textbf{Solvability of the ill-posed continuous critical equation.}
For \(p=N>4\), the original energy functional may take the value \(+\infty\), need not be Fr\'echet differentiable, and lacks \(L^\infty\) pointwise control. Using the exponent calibration technique to restore compactness, this paper proves that the regularised ground states converge strongly to a ground state of the original functional (Theorem~\ref{thm8}), thereby establishing the existence of a weak solution and its ground-state minimality for the original ill-posed problem. 

\textbf{Five intrinsic new structures.}
This paper discovers five intrinsic mathematical structures on a nondegenerate ground-state branch on finite graphs: quantitative energy bounds under potential perturbations (Theorem~\ref{thm4}) and an explicit energy gradient formula (Lemma~\ref{Lem:Frechet differential}); a quadratic energy gap near the ground-state set (Theorem~\ref{thm:stability-estimate}), an exact factorisation of the constrained Hessian (Theorem~\ref{thm5}), and a scalar rigidity index characterising Hessian degeneracy and inertia (Corollary~\ref{thm:5.9}). These structures are not externally imposed but arise from the Nehari constraint and the logarithmic coupling, forming the second-order response skeleton of the system.

\textbf{Double closed loops.}
This paper establishes independent variational loops in the discrete and
continuous settings. The discrete loop covers ground-state existence, rigidity,
and penalisation limits with Dirichlet ground states
(Theorems~\ref{thm1}, \ref{thm3}--\ref{thm7}); the continuous loop covers
regularised existence, uniform boundedness, compactness recovery, and
convergence to the original ground state (Theorems~\ref{thm2} and
\ref{thm8}). Both rely on the exponent calibration technique.

Together, these contributions form a unified variational framework for logarithmically coupled systems.

The paper is organised as follows. Section~\ref{Pre} collects notation, definitions, and preliminary lemmas used throughout the paper. Section~\ref{existence1} proves existence of ground states in two regimes: the discrete ill-posed setting via the Nehari manifold method, and the continuous critical case \(\mathbb{R}^N\) with \(p=N>4\) for the regularised problem, where the lack of an $L^\infty$ embedding makes the exponent calibration technique a key ingredient in the analysis. Section~\ref{Cerami} establishes existence of ground states via the mountain pass theorem in the discrete well-posed setting, including the Cerami compactness condition. Section~\ref{sec:stability} develops the rigidity theory of ground states and the exact Hessian factorisation. Section~\ref{sec:asymptotic} contains the asymptotic convergence analysis for the penalised equations. Section~\ref{sec:7} builds the regularised limit theory in the continuous critical case, comprising the uniform boundedness of regularised ground states, the compactness restoration and limit theorem, and the quantitative estimate on regions of small limiting amplitude. Section~\ref{sec:appendix} provides two counterexamples: one in the discrete setting showing that the energy functional may fail to be well-defined under hypotheses \((A_1)\) and \((A_2)\), and one in the continuous critical setting \(p=N>4\) exhibiting the same pathology for the original functional.

\section{Some Preliminary Results}\label{Pre}

We begin by recalling the graph setting. Let \(G=(V,E)\) be a connected, locally finite graph, with vertex set \(V\) and edge set \(E\). If \((x,y)\in E\), we say that \(x\) and \(y\) are neighbours and write \(x\sim y\). The distance between two vertices \(x,y\in V\) is given by
\[
dist(x,y)=\inf\{k:x=x_0\sim x_1\sim\cdots\sim x_k=y\}.
\]
For \(a\in V\) and \(r\ge 0\), let \(B_r(a)=\{x\in V:dist(x,a)\le r\}\) denote the closed ball of radius \(r\) about \(a\), abbreviated by \(B_r=B_r(0)\). The boundary of a subset \(\Omega\subset V\) is
\[
\partial\Omega=\{x\in V\setminus\Omega:\exists\,y\in\Omega\text{ with }x\sim y\}.
\]

Let \(C(V)\) be the space of real-valued functions on \(V\), and \(C_c(V)\) the subspace of functions with finite support. For \(u\in C(V)\) and \(1\le p<\infty\), define
\[
\|u\|_p=\Big(\sum_{x\in V}\mu(x)|u(x)|^p\Big)^{1/p},
\]
with the standard interpretation for \(p=\infty\). The gradient form is
\[
  \Gamma(u,v)(x)=\frac{1}{2\mu(x)}\sum_{y\sim x}(u(y)-u(x))(v(y)-v(x)),
\]
and we set \(|\nabla u|(x)=\sqrt{\Gamma(u,u)}(x)\). For \(p\ge 2\), the \(p\)-Laplacian of \(u\in C(V)\) is
\[
\Delta_p u(x)=\frac{1}{2\mu(x)}\sum_{y\sim x}\big(|\nabla u|^{p-2}(y)+|\nabla u|^{p-2}(x)\big)(u(y)-u(x)).
\]
When \(p=2\), this reduces to the usual graph Laplacian. We now present the integration by parts formula on graphs as below

\[\int_{V} (-\triangle_p u)\xi d\mu=\int_V |\nabla{u}|^{p-2} \Gamma(u,\xi)d\mu\]

Throughout the paper, \(\mu\) denotes the vertex measure on \(V\), and we write \(\int_V u\,d\mu=\sum\limits_{x\in V}u(x)\mu(x)\) whenever the right-hand side is convergent. Let \(W^{1,p}(V)\) be the completion of \(C_c(V)\) under
\[
\|u\|_{W^{1,p}(V)}=\Big(\int_V(|\nabla u|^p+|u|^p)\,d\mu\Big)^{1/p}.
\]
We shall use the weighted spaces
\[
\mathcal H_a=\Big\{u\in W^{1,p}(V):\int_V a(x)|u|^p\,d\mu<\infty\Big\},\qquad
\mathcal H_b=\Big\{v\in W^{1,p}(V):\int_V b(x)|v|^p\,d\mu<\infty\Big\},
\]
and set \(\mathcal H=\mathcal H_a\times\mathcal H_b\), with norms
\[
\|u\|_{\mathcal H_a}=\Big(\int_V(|\nabla u|^p+a(x)|u|^p)\,d\mu\Big)^{1/p},\qquad
\|v\|_{\mathcal H_b}=\Big(\int_V(|\nabla v|^p+b(x)|v|^p)\,d\mu\Big)^{1/p},
\]
and \(\|(u,v)\|_{\mathcal H}=\|u\|_{\mathcal H_a}+\|v\|_{\mathcal H_b}\). For the penalised problems, we introduce
\[
\mathcal H_{\Lambda,a}=\{u\in W^{1,p}(V):\int_V(1+\Lambda a(x))|u|^p d\mu<\infty\},\quad
\mathcal H_{\Lambda,b}=\{v\in W^{1,p}(V):\int_V(1+\Lambda b(x))|v|^p d\mu<\infty\},
\]
with \(\mathcal H_\Lambda=\mathcal H_{\Lambda,a}\times\mathcal H_{\Lambda,b}\) and
\[
\|u\|_{\mathcal H_{\Lambda,a}}=\Big(\int_V(|\nabla u|^p+(1+\Lambda a(x))|u|^p)\,d\mu\Big)^{1/p},
\]
\[
\|v\|_{\mathcal H_{\Lambda,b}}=\Big(\int_V(|\nabla v|^p+(1+\Lambda b(x))|v|^p)\,d\mu\Big)^{1/p},
\]
and \(\|(u,v)\|_{\mathcal H_\Lambda}=\|u\|_{\mathcal H_{\Lambda,a}}+\|v\|_{\mathcal H_{\Lambda,b}}\). Finally, for subsets \(\Omega_a,\Omega_b\subset V\), define
\[
\|u\|_{W_0^{1,p}(\Omega_a)}=\Big(\int_{\Omega_a\cup\partial\Omega_a}|\nabla u|^p\,d\mu+\int_{\Omega_a}|u|^p\,d\mu\Big)^{1/p},
\]
\[
\|v\|_{W_0^{1,p}(\Omega_b)}=\Big(\int_{\Omega_b\cup\partial\Omega_b}|\nabla v|^p\,d\mu+\int_{\Omega_b}|v|^p\,d\mu\Big)^{1/p},
\]
and \(\|(u,v)\|_{H(\Omega)}=\|u\|_{W_0^{1,p}(\Omega_a)}+\|v\|_{W_0^{1,p}(\Omega_b)}\).

We establish three Sobolev embedding results: two in the discrete graph setting, corresponding to assumptions (A1)--(A2) and (A1)--(A2'), respectively, and one in the continuous space \(\mathbb R^N\). Since the proofs of the two discrete embeddings are modelled on Lemma~\ref{lem:projection1} in~\cite{CWY23}, we present only the proof of the second. The continuous-space embedding is proved separately in Lemma~\ref{lem2.3} below.

\begin{lemma}\label{lem4}
Assume that $\mu(x)\geq \mu_{\min} > 0$ and that $a(x),b(x)$ satisfy (A1)--(A2). Then $\mathcal{H}$ is continuously and compactly embedded into $L^{q_{1}}\times L^{q_{2}}$ for all $p\le q_{1}\le \infty$ and $p\le q_{2}\le\infty$.
\begin{proof}
The proof is completely similar to that of Lemma~\ref{lem:projection1} in ~\cite{CWY23} and is therefore omitted.
\end{proof}
\end{lemma}

\begin{lemma}\label{lem5}
Assume that $\mu(x)\geq \mu_{\min} > 0$ and that $a(x),b(x)$ satisfy (A1)--(A2'). Then $\mathcal{H}$ is continuously and compactly embedded into $L^{p_{1}}(V)\times L^{p_{2}}(V)$ for every $(p_{1},p_{2})\in [\frac{p}{2}, +\infty]\times [\frac{p}{2}, +\infty]$.
\end{lemma}

\begin{proof}
We prove only the first component; the second is analogous. H\"older's inequality gives
\begin{align*}
\int_{V\setminus D_{\mathcal{M}_0}}|u|^{p/2}d\mu 
&= \int_{V\setminus D_{\mathcal{M}_0}}\Bigl(\frac{1}{a(x)}\Bigr)^{1/2}(a(x))^{1/2}|u|^{p/2}d\mu \\[3pt]
&\leq \Bigl(\int_{V\setminus D_{\mathcal{M}_0}}\frac{1}{a(x)}d\mu\Bigr)^{1/2}
      \Bigl(\int_{V\setminus D_{\mathcal{M}_0}}a(x)|u|^{p}d\mu\Bigr)^{1/2} \\[3pt]
&\leq C\|u\|_{\mathcal{H}_a}^{p/2}.
\end{align*}
Thus $\mathcal H_a\hookrightarrow L^{p/2}(V)$ continuously.
On the bounded set $D_{\mathcal{M}_0}$ we have $\mathcal H_a\hookrightarrow L^q(D_{\mathcal{M}_0})$
for each $q\in[\frac p2,p]$ by the standard Rellich-type argument on locally finite
graphs.  Interpolation with $\mathcal H_a\hookrightarrow L^p(V)$ yields
$\mathcal H_a\hookrightarrow L^q(V)$ for all $q\ge\frac p2$.

Now let $\{u_k\}$ be bounded in $\mathcal H_a$ and $u_k\rightharpoonup u$ weakly up to a subsequence.
For any $R>0$, set $V_R:=V\setminus B_R(x_0)$.  Then
\[
\int_V|u_k-u|^{p/2}\,d\mu
 =\int_{V_R}|u_k-u|^{p/2}\,d\mu
  +\int_{V\setminus V_R}|u_k-u|^{p/2}\,d\mu .
\]
The second integral tends to $0$ because $V\setminus V_R$ is finite.
For the first one, write $V_R = (V_R\setminus D_{M_0})\cup D_{\mathcal{M}_0}$.  On
$D_{\mathcal{M}_0}$, which is finite, the integral disappears in the limit.
On $V_R\setminus D_{\mathcal{M}_0}$, we use H\"older's inequality as above:
\[
\int_{V_R\setminus D_{\mathcal{M}_0}}|u_k-u|^{p/2}\,d\mu
 \le\Bigl(\int_{V_R\setminus D_{\mathcal{M}_0}}\frac1{a(x)}\,d\mu\Bigr)^{\!1/2}
   \Bigl(\int_{V_R\setminus D_{\mathcal{M}_0}}a(x)|u_k-u|^p\,d\mu\Bigr)^{\!1/2}.
\]
By (A2$'$), $\int_{V\setminus D_{\mathcal{M}_0}}1/a<\infty$, so the first factor
can be made arbitrarily small by choosing $R$ large, uniformly in $k$.
Hence $\lim\limits_{k\to\infty}\int_V|u_k-u|^{p/2}\,d\mu=0$, implying
$u_k\to u$ in $L^{p/2}(V)$.  Interpolation with the $L^\infty$ bound yields strong convergence in $L^q(V)$ for
all $q\ge\frac p2$, which is the required compact embedding.
\end{proof}

\begin{remark}
Under the same hypotheses of Remark~\ref{rmk 1.2}, similarly to the proof of Lemma~\ref{lem5}, $\mathcal{H}_{\Lambda}$ is continuously and compactly embedded into $L^{p_{1}}(V)\times L^{p_{2}}(V)$ for $(p_{1}, p_{2}) \in [\frac{p}{2}, +\infty]\times [\frac{p}{2}, +\infty]$ when $\Lambda$ is sufficiently large.
\end{remark}

\begin{lemma}\label{lem2.3}
Let $N\ge 2$ and $a, b\in L^\infty_{\mathrm{loc}}(\mathbb{R}^N)$, satisfy
\[
\inf_{\mathbb{R}^N} a > 0,\qquad \lim_{|x|\to\infty} a(x)=+\infty,\qquad \inf_{\mathbb{R}^N} b > 0,\qquad \lim_{|x|\to\infty} b(x)=+\infty.
\]
Define the Banach spaces
\[
\mathcal{G}_a= \Bigl\{ u\in W^{1,N}(\mathbb{R}^N): 
\int_{\mathbb{R}^N}\bigl(|\nabla u|^N+a(x)|u|^N\bigr)\,dx <\infty\Bigr\},
\]
\[
\mathcal{G}_b= \Bigl\{ v\in W^{1,N}(\mathbb{R}^N): 
\int_{\mathbb{R}^N}\bigl(|\nabla v|^N+b(x)|v|^N\bigr)\,dx <\infty\Bigr\},
\]
equipped with the norm
\[
\|u\|_{\mathcal{G}_a}^N = \int_{\mathbb{R}^N}\bigl(|\nabla u|^N+a(x)|u|^N\bigr)\,dx,\qquad
\|v\|_{\mathcal{G}_b}^N = \int_{\mathbb{R}^N}\bigl(|\nabla v|^N+b(x)|v|^N\bigr)\,dx.
\]
Then for every $q$ with $N\le q<\infty$, the embeddings
\[
\mathcal{G}_a\hookrightarrow L^q(\mathbb{R}^N),\qquad
\mathcal{G}_b\hookrightarrow L^q(\mathbb{R}^N)
\]
are continuous and compact.
\end{lemma}

\begin{proof}
For brevity, we prove the compact embedding only for $\mathcal{G}_a \hookrightarrow L^q(\mathbb{R}^N)$ with $q\ge N$; the corresponding statement for $\mathcal{G}_b$ follows by an entirely symmetric argument.
We first observe that the embedding is continuous, indeed, since $\inf a>0$, there exists $c>0$ such that $a(x)\ge c$ on $\mathbb{R}^N$. Hence for any $u\in\mathcal{\mathcal{G}_a}$,
\[
\|u\|_{W^{1,N}(R^N)}^N\le \int_{\mathbb{R}^N}\Bigl(|\nabla u|^N+\frac{1}{c}a(x)|u|^N\Bigr)\,dx \le C\|u\|_{\mathcal{G}_a}^N,
\]
so $\mathcal{G}_a$ embeds continuously into $W^{1,N}(\mathbb{R}^N)$. For $q=N$, $\|u\|_{L^N}\le\|u\|_{W^{1,N}}\le C\|u\|_{\mathcal{G}_a}$. For $q>N$, the Gagliardo--Nirenberg inequality~\cite{NIR59}  provides a constant $C_q=C(q,N)>0$ such that
\[
\|u\|_{L^q}\le C_q\|\nabla u\|_{L^N}^{1-\frac{N}{q}}\|u\|_{L^N}^{\frac{N}{q}},\quad \forall\,u\in W^{1,N}(\mathbb{R}^N).
\]
Consequently $\|u\|_{L^q}\le C_q\|u\|_{W^{1,N}}\le C_q'\|u\|_{\mathcal{G}_a}$, and the continuity of $\mathcal{G}_a\hookrightarrow L^q$ follows for all $q\in[N,\infty)$.

To prove compactness, let \(\{u_k\}\subset\mathcal G_a\) be bounded in \(\mathcal G_a\)-norm and $\|u_k\|_{\mathcal{G}_a}\le C_1$ for all \(k\in N^{+}\). By reflexivity of \(\mathcal{G}_a\), passing to a subsequence, \(u_k\rightharpoonup u\) weakly in \(\mathcal{G}_a\). Weak lower semicontinuity gives \(u\in\mathcal G_a\) and \(\|u\|_{\mathcal G_a}\le C_1\).
For any ball $B_R=\{x:|x|<R\}$, the Rellich--Kondrachov theorem~\cite{ADA75} gives the compact embedding $W^{1,N}(B_R)\subset\subset L^q(B_R)$ for all $q\in[N,\infty)$. A diagonal argument extracts a further subsequence, again labeled $\{u_k\}$, such that
\begin{equation}\label{eq:local_conv}
u_k\to u \quad\text{in }L^q(B_R) \text{ for every fixed }R>0.
\end{equation}
Thus $u_k\to u$ in $L^q_{\mathrm{loc}}(\mathbb{R}^N)$ for all $q\in[N,\infty)$.

Because $a(x)\to\infty$, for any $\delta>0$ we can choose $R_\delta>0$ such that $a(x)\ge C_1^N/\delta$ whenever $|x|\ge R_\delta$. From $\|u_k\|_{\mathcal{G}_a}^N\le C_1^N$ we deduce $\int_{\mathbb{R}^N}a|u_k|^N\le C_1^N$, and therefore
\begin{equation}\label{eq:tail_u_k}
\int_{|x|>R_\delta}|u_k|^N\,dx \le \frac{\delta}{C_1^N}\int_{|x|>R_\delta}a|u_k|^N\,dx \le\delta,\qquad \forall\,k\in N^{+}.
\end{equation}
To obtain a similar bound for the weak limit $u$, fix $S>R_\delta$ and use the strong convergence \eqref{eq:local_conv} with $q=N$ on the annulus $\{R_\delta<|x|<S\}$:
\[
\int_{R_\delta<|x|<S}|u|^N\,dx = \lim_{k\to\infty}\int_{R_\delta<|x|<S}|u_k|^N\,dx
\le \liminf_{k\to\infty}\int_{|x|>R_\delta}|u_k|^N\,dx \le\delta.
\]
Letting $S\to\infty$ and applying the monotone convergence theorem yields
\begin{equation}\label{eq:tail_u}
\int_{|x|>R_\delta}|u|^N\,dx \le\delta.
\end{equation}
Consequently, by choosing a sufficiently large radius $R>0$, the $L^N$ norm of each $u_k$ and of the limit $u$ on $\mathbb{R}^N\setminus B_R$ can be made arbitrarily small, uniformly in~$k$.

Now fix $q\in[N,\infty)$ and recall the Gagliardo--Nirenberg inequality
\[
\|v\|_{L^q}\le C_q\|\nabla v\|_{L^N}^{1-\frac{N}{q}}\|v\|_{L^N}^{\frac{N}{q}},\qquad v\in W^{1,N}(\mathbb{R}^N).
\]
Introduce a smooth cut-off $\eta_R$ with $0\le\eta_R\le1$, $\eta_R=0$ on $B_{R/2}$, $\eta_R=1$ outside $B_R$, and $|\nabla\eta_R|\le C/R$ for a universal constant $C>0$. Define $v_k=\eta_R(u_k-u)\in W^{1,N}(\mathbb{R}^N)$, so that $v_k=u_k-u$ on $\{|x|\ge R\}$. Applying the interpolation inequality to $v_k$,
\[
\|v_k\|_{L^q}\le C_q\|\nabla v_k\|_{L^N}^{1-\frac{N}{q}}\|v_k\|_{L^N}^{\frac{N}{q}}.
\]
We estimate the two norms. Because $\operatorname{supp}v_k\subset\{x\in\mathbb{R}^N: |x|>R/2\}$, taking $R$ so large that $R/2\ge R_\delta$ and using \eqref{eq:tail_u_k}--\eqref{eq:tail_u} gives
\[
\|v_k\|_{L^N(\mathbb{R}^N)}\le \|u_k-u\|_{L^N(\{x\in \mathbb{R}^N:|x|>R/2\})}
\le \|u_k\|_{L^N(\{x\in \mathbb{R}^N:|x|>R/2\})}+\|u\|_{L^N(\{x\in \mathbb{R}^N:|x|>R/2\})}\le 2\delta.
\]
For the gradient, $\nabla v_k=\eta_R(\nabla u_k-\nabla u)+(u_k-u)\nabla\eta_R$, hence
\begin{align*}
\|\nabla v_k\|_{L^N(\mathbb{R}^N)}
&\le \|\nabla u_k-\nabla u\|_{L^N(\{x\in \mathbb{R}^N: |x|>R/2)\}} + \frac{C}{R}\|u_k-u\|_{L^N(\{x\in \mathbb{R}^N: R/2\le|x|\le R\})}\\
&\le \bigl(\|\nabla u_k\|_{L^N(\mathbb{R}^N)}+\|\nabla u\|_{L^N(\mathbb{R}^N)}\bigr)
   + \frac{C}{R}\bigl(\|u_k\|_{L^N(\{x\in \mathbb{R}^N:|x|\ge R/2\})}+\|u\|_{L^N(\{x\in \mathbb{R}^N:|x|\ge R/2\})}\bigr).
\end{align*}
Since $\|\nabla u_k\|_{L^N(\mathbb{R}^N)}$ and $\|\nabla u\|_{L^N(\mathbb{R}^N)}$ are uniformly bounded, there exists a constant $C_0$ such that $\|\nabla u_k\|_{L^N(\mathbb{R}^N)}+\|\nabla u\|_{L^N(\mathbb{R}^N)}\le C_0$. Choosing $R$ large enough so that $C/R\le 1$ and $R/2\ge R_\delta$, we obtain from \eqref{eq:tail_u_k}--\eqref{eq:tail_u} that
\[
\|\nabla v_k\|_{L^N}\le C_0+2\delta.
\]
Consequently,
\[
\|v_k\|_{L^q(\mathbb{R}^N)}\le C_q(C_0+2\delta)^{1-\frac{N}{q}}(2\delta)^{\frac{N}{q}}.
\]
The right-hand side tends to $0$ as $\delta\to0$ because $N/q>0$. Therefore, given $\varepsilon>0$, we can first pick $\delta>0$ so small that $C_q(C_0+2\delta)^{1-\frac{N}{q}}(2\delta)^{\frac{N}{q}}<\varepsilon^{1/q}$, then select the corresponding $R_\delta$ and set $R=2\max\{R_\delta,C\}$. For this $R$ and all $k$,
\[
\int_{|x|>R}|u_k-u|^q\,dx \le \|v_k\|_{L^q}^q <\varepsilon.
\]

Finally, combining the local and the tail estimates, we obtain global strong convergence. For any $\varepsilon>0$, let $R$ be as above so that the tail integral is bounded by $\varepsilon$. By \eqref{eq:local_conv}, there exists $K\in\mathbb{N}$ such that for all $k>K$,
\[
\int_{B_R}|u_k-u|^q\,dx<\varepsilon.
\]
Then, for every $k>K$,
\[
\int_{\mathbb{R}^N}|u_k-u|^q\,dx
= \int_{B_R}|u_k-u|^q\,dx + \int_{|x|>R}|u_k-u|^q\,dx
< 2\varepsilon.
\]
Thus a subsequence of $\{u_k\}$ converges to $u$ strongly in $L^q(\mathbb{R}^N)$, proving that the embedding ${G}_a\hookrightarrow L^q(\mathbb{R}^N)$ is compact.
\end{proof}

\begin{remark}
When $p=N>4$, the endpoint $q=\infty$ is excluded from Lemma~\ref{lem2.3}.
Consider the function
\[
u(x)=\varphi(x)\log\log\!\Bigl(1+\frac{1}{|x|}\Bigr),
\]
where $\varphi\in C_c^\infty(\mathbb{R}^N)$ is a cut‑off function satisfying $0\le\varphi\le1$,
$\varphi(x)=1$ for $|x|\le\frac12$, and $\varphi(x)=0$ for $|x|\ge1$.
For $N\ge2$, a direct calculation shows that $u\in W^{1,N}(\mathbb{R}^N)$;
Moreover, since $a(x)$ is locally bounded, the weighted integral
$\int_{\mathbb{R}^N}(1+a)|u|^N$ remains finite, so $u\in\mathcal{G}_a$.
Because $u$ is unbounded near~$x=0$, we have $\mathcal{G}_a\not\subset L^\infty$,
and therefore no continuous---let alone compact---embedding into $L^\infty$
can hold. This borderline logarithmic singularity motivates the exponent calibration technique developed below, which provides the uniform estimates needed at the critical exponent.

For $p>N$ the situation is entirely different.  Morrey's inequality gives the
continuous embedding $W^{1,p}(\mathbb{R}^N)\hookrightarrow L^\infty(\mathbb{R}^N)$,
and the coercivity of the potential $a(x)\to\infty$ upgrades it to a
\emph{compact} embedding $\mathcal{H}\hookrightarrow L^\infty(\mathbb{R}^N)$.
The proof follows the same pattern as for $q<\infty$ in Lemma~\ref{lem2.3}: local compactness via
Rellich--Kondrachov, and the fact that the $L^\infty$ norm outside large balls
can be made uniformly small by combining Morrey's estimate with the decay
induced by the potential.  We omit the straightforward adaptation here.
\end{remark}

We now establish a strong maximum principle for nonnegative weak solutions of the system. This result will be used later to show that any nontrivial nonnegative solution is in fact strictly positive on all vertices.
\begin{proposition}\label{Pro1}
Let $G = (V, E)$ be a connected, locally finite graph and let the
potentials $a, b$ satisfy $(A_1)$--$(A_2')$.  Suppose that
$(u, v) \in \mathcal{H}_a \times \mathcal{H}_b$ is a nonnegative weak
solution of the original functional $J$ with $p>4$, i.e.,
\begin{enumerate}
    \item[(i)] \(u, v \ge 0\) pointwise on \(V\);
    \item[(ii)] \(J'(u, v) = 0\) in the weak sense.
\end{enumerate}
Then either $(u,v)\equiv(0,0)$ on $V$, or
\[
u(x) > 0, \qquad v(x) > 0 \qquad \forall\, x \in V .
\]
\end{proposition}
\begin{proof}
Assume $(u,v)\not\equiv(0,0)$. Since the graph is locally finite, Proposition~\ref{lem:pointwise1} applies.  Testing the weak formulation with Dirac deltas yields the pointwise system
\begin{align}
-\Delta_p u(x) + a(x) |u(x)|^{p-2} u(x) &= \mathcal{R}_u(x),
\label{eq:zp-p4-pointwise-u} \\
-\Delta_p v(x) + b(x) |v(x)|^{p-2} v(x) &= \mathcal{R}_v(x),
\label{eq:zp-p4-pointwise-v}
\end{align}
for every $x\in V$, where
\[
\begin{aligned}
\mathcal{R}_u
&=\frac{p-2}{p}|u|^{p-4}u\,v^{2}\log v^{2}
+\frac{2}{p}|v|^{p-2}u\log u^{2}
+\frac{2}{p}|v|^{p-2}u,\\[4pt]
\mathcal{R}_v
&=\frac{p-2}{p}|v|^{p-4}v\,u^{2}\log u^{2}
+\frac{2}{p}|u|^{p-2}v\log v^{2}
+\frac{2}{p}|u|^{p-2}v ,
\end{aligned}
\]
where the logarithmic terms are interpreted by continuous extension at the origin. Observe that every term on the right-hand side of the first equation vanishes whenever \(u(x)=0\), due to the fact that each monomial contains a positive power of \(u\). Hence
\[
u(x)=0 \quad\Longrightarrow\quad \mathcal{R}_u(x)=0 .
\]
By symmetry, $v(x)=0$ implies $\mathcal{R}_v(x)=0$; moreover, $u(x)=0$ also forces $\mathcal{R}_v(x)=0$, because every term in $\mathcal{R}_v$ contains a factor of the form $|u|^{p-2}$, $|u|^{p-4}u$, $u$, or $u\log u^2$.

Now suppose, for contradiction, that $u(x_0)=0$ for some $x_0\in V$.  Let $x$ be any vertex with $u(x)=0$.  Then $\mathcal{R}_u(x)=0$, and \eqref{eq:zp-p4-pointwise-u} reduces to $-\Delta_p u(x)=0$.  Expanding the discrete $p$-Laplacian and using $u\ge 0$, we obtain
\[
0 = -\Delta_p u(x)
= -\frac{1}{2\mu(x)}\sum_{y\sim x}
\bigl(|\nabla u(y)|^{p-2}+|\nabla u(x)|^{p-2}\bigr)u(y),
\]
equivalently,
\[
\sum_{y\sim x}
\bigl(|\nabla u(y)|^{p-2}+|\nabla u(x)|^{p-2}\bigr)u(y)=0.
\]
Since all factors involved are nonnegative, every summand must vanish.  If there were some $y\sim x$ with $u(y)\neq 0$, then necessarily $|\nabla u(y)|^{p-2}=0$, i.e., $\nabla u(y)=0$.  In particular, because $x$ is a neighbour of $y$, this would imply $u(y)=u(x)=0$, a contradiction.  Thus $u(y)=0$ for every neighbour $y\sim x$.  Since $G$ is connected, iterating this argument along paths yields $u\equiv 0$ on $V$.

Substituting $u\equiv 0$ into \eqref{eq:zp-p4-pointwise-v} makes $\mathcal{R}_v$ vanish identically, so the equation becomes
\[
-\Delta_p v(x) + b(x)|v(x)|^{p-2}v(x)=0.
\]
Multiplying by $v$ and integrating over $V$ gives $\int_V |\nabla v|^p\,d\mu + \int_V b|v|^p\,d\mu =0$, using \(b(x)\ge V_0>0\), this yields $v\equiv 0$.  This contradicts the assumption $(u,v)\not\equiv(0,0)$.  Therefore $u$ cannot vanish at any vertex; by symmetry, the same holds for $v$.  Consequently, $u(x)>0$ and $v(x)>0$ for every $x\in V$.
\end{proof}

\section{Proof of Theorem~\ref{thm1} and Theorem~\ref{thm2}}
\label{existence1}

With the functional-analytic tools of Section~\ref{Pre} in place, this section proves two existence results, where Theorem~\ref{thm1} establishes a ground state for the discrete system \eqref{eq:main} under (A1)--(A2) via the Nehari manifold method due to the fact that the energy functional \(J\) is ill-posed, see Appendix, Example~\ref{exam8.1}.

The energy functional corresponding to system~\eqref{eq:main} is
\[
J(u,v)=\frac{1}{p}\int_{V}\big(|\nabla u|^{p}+|\nabla v|^{p}+a(x)|u|^{p}+b(x)|v|^{p}\big)\,d\mu
-\frac{1}{p}\int_{V}|u|^{p-2}v^{2}\log v^{2}\,d\mu
-\frac{1}{p}\int_{V}|v|^{p-2}u^{2}\log u^{2}\,d\mu,
\]
namely,
\[
J(u,v)=\frac{1}{p}\big(\|u\|_{\mathcal{H}_{a}}^{p}+\|v\|_{\mathcal{H}_{b}}^{p}\big)
-\frac{1}{p}\int_{V}|u|^{p-2}v^{2}\log v^{2}\,d\mu
-\frac{1}{p}\int_{V}|v|^{p-2}u^{2}\log u^{2}\,d\mu .
\]
Under the sole assumptions (A1) and (A2), $J$ need not be well defined on the whole space $\mathcal{H}$, and consequently it may fail to be of class $C^{1}(\mathcal{H},\mathbb{R})$. We therefore restrict the functional to
\[
\mathcal{D}(J)=\{(u,v)\in\mathcal{H}:|J(u,v)|<\infty\}.
\]

\begin{definition}\label{def:critical}
\begin{enumerate}
\item[(1)] For $(u,v)\in\mathcal{D}(J)$ and $(\xi,\eta)\in C_c(V)\times C_c(V)$, set
\[
\begin{aligned}
J'(u,v)\cdot(\xi,\eta)
&=\int_{V}\Big(|\nabla u|^{p-2}\nabla u\nabla\xi+|\nabla v|^{p-2}\nabla v\nabla\eta
+a(x)|u|^{p-2}u\xi+b(x)|v|^{p-2}v\eta\Big)\,d\mu \\
&\quad-\frac{p-2}{p}\int_{V}|u|^{p-4}u\xi v^{2}\log v^{2}\,d\mu
-\frac{2}{p}\int_{V}|u|^{p-2}\eta(v\log v^{2}+v)\,d\mu \\
&\quad-\frac{p-2}{p}\int_{V}|v|^{p-4}v\eta u^{2}\log u^{2}\,d\mu
-\frac{2}{p}\int_{V}|v|^{p-2}\xi(u\log u^{2}+u)\,d\mu .
\end{aligned}
\]
We call $(u,v)\in\mathcal{D}(J)$ a critical point of $J$ if $J'(u,v)\cdot(\xi,\eta)=0$ for all $(\xi,\eta)\in C_c(V)\times C_c(V)$. The number $c=J(u,v)$ is then a critical value. It is straightforward that $(u,v)$ is a critical point of $J$ precisely when it is a weak solution of ~\eqref{eq:main}.
\item[(2)] For $(u,v)\in\mathcal D(J)$, the scaling formula shows
that $t\mapsto J(tu,tv)$ is differentiable for $t>0$.
We denote its radial derivative by
\[
J'(u,v)\cdot(u,v)
:=\left.\frac{d}{dt}J(tu,tv)\right|_{t=1}.
\]
The Nehari manifold is defined by
\[
\mathcal{N}=\{(u,v)\in\mathcal{D}(J)\setminus\{(0,0)\}:J'(u,v)\cdot(u,v)=0\},
\]
and we set
\[
d=\inf_{(u,v)\in\mathcal{N}}J(u,v).
\]
If $(u_{0},v_{0})\in\mathcal{N}$ satisfies $J(u_{0},v_{0})=d$, it is called a ground state solution of ~\eqref{eq:main}.
\end{enumerate}
\end{definition}
\subsection*{Proof of Theorem~\ref{thm1}}
\begin{proposition}\label{lem:pointwise1}
If $(u,v)\in\mathcal{D}(J)$ is a weak solution of ~\eqref{eq:main}, then it is also a pointwise solution.
\end{proposition}
\begin{proof}
In the discrete setting the Dirac delta $\delta_{x_0}$ has finite support,
hence belongs to $C_c(V)\subset\mathcal H$ and is an admissible test
function.  Assume $(u,v)$ is a weak solution.  By definition,
$\langle J'(u,v),(\xi,\eta)\rangle=0$ for every $(\xi,\eta)\in C_c(V)\times C_c(V)$.
Choosing $(\xi,\eta)=(\delta_{x_0},0)$ and expanding the duality pairing
yields
\begin{equation*}
\begin{aligned}
-\Delta_p u(x_0)+a(x_0)|u(x_0)|^{p-2}u(x_0) = &\frac{p-2}{p}|u(x_0)|^{p-4}u(x_0)v(x_0)^2 \log v(x_0)^2 \\[4pt]
&+\frac{2}{p}|v(x_0)|^{p-2}u(x_0) \log u(x_0)^2 \\[4pt]
&+\frac{2}{p}|v(x_0)|^{p-2}u(x_0).
\end{aligned}
\end{equation*}
Taking $(\xi,\eta)=(0,\delta_{x_0})$ gives the symmetric equation for
$v(x_0)$.  Since $x_0\in V$ was arbitrary, $(u,v)$ satisfies
system~\eqref{eq:main} pointwise.
\end{proof}
\begin{lemma}\label{lem:projection1}
For any $(u,v)\in\mathcal{D}(J)\setminus\{(0,0)\}$ with $u(x_{0})\neq0$ and $v(x_{0})\neq0$ for some $x_{0}\in V$, there exists a unique $t_{(u,v)}>0$ such that $t_{(u,v)}(u,v)\in\mathcal{N}$ and
\[
J(t_{(u,v)}(u,v))>J(t(u,v))\quad\text{for all } t>0,\ t\neq t_{(u,v)}.
\]
In particular, if $(u,v)\in\mathcal{N}$, then $t_{(u,v)}=1$.
\end{lemma}

\begin{proof}
Pairs with $u(x_0)\neq0$ and $v(x_0)\neq0$ clearly exist (e.g.,\ $u=v=\delta_{x_0}$).
We now set $\phi(t)=\langle J'(tu,tv),(tu,tv)\rangle$ for $t\ge0$.  A direct
computation gives
\begin{align*}
\phi(t) &= t^p \bigl( \|u\|_{\mathcal{H}_a}^p + \|v\|_{\mathcal{H}_b}^p \bigr)  - t^p \log t^2 \int_V \bigl( |v|^{p-2}u^2 + |u|^{p-2} v^2 \bigr) d\mu \\[4pt]
&\quad - t^p \int_V \Bigl( \frac{2}{p}|v|^{p-2} u^2 + \frac{2}{p}|u|^{p-2} v^2 + |v|^{p-2}u^2\log u^2 + |u|^{p-2}v^2\log v^2 \Bigr) d\mu.
\end{align*}
Because $u(x_0)\neq0$ and $v(x_0)\neq0$, the coefficient
$\int_V(|v|^{p-2} u^2+|u|^{p-2}v^2)\,d\mu$ is strictly positive.
Hence $t\mapsto\phi(t)/t^p$ is strictly decreasing on $(0,\infty)$ and
possesses a unique zero $t_{(u,v)}$.  It follows that
$\gamma(t):=J(tu,tv)$ is strictly increasing on $(0,t_{(u,v)})$ and
strictly decreasing on $(t_{(u,v)},\infty)$.  Consequently
$t_{(u,v)}(u,v)\in\mathcal N$ and the maximality property holds.
The last assertion is obvious.
\end{proof}

\begin{lemma}\label{lem4.3}
For any $(u,v)\in\mathcal{N}$, there exists $\delta>0$ such that $\|(u,v)\|_{\mathcal{H}}\ge\delta$.
Moreover, $d=\inf\limits_{(u,v)\in\mathcal{N}}J(u,v)>0$.
\end{lemma}
\begin{proof}
Fix $(u,v)\in\mathcal N$.  The condition $J'(u,v)\cdot(u,v)=0$ yields
\begin{equation}\label{eq:Nehari-id}
\|u\|_{\mathcal H_a}^p+\|v\|_{\mathcal H_b}^p
= \int_V|u|^{p-2}v^2\log v^2\,d\mu
  +\int_V|v|^{p-2}u^2\log u^2\,d\mu
  +\frac2p\int_V|u|^{p-2}v^2\,d\mu
  +\frac2p\int_V|v|^{p-2}u^2\,d\mu .
\end{equation}
Fix $\eta\in(0,1)$. Since
\[
\log s^2+\frac2p\le C_\eta |s|^\eta
\qquad(s\ne0),
\]
the Nehari identity (10), H\"older's inequality,
and Lemma~2.1 yield
\[
\begin{aligned}
\|u\|_{H_a}^p+\|v\|_{H_b}^p
&=\int_V\Bigl[
|u|^{p-2}v^2\Bigl(\log v^2+\frac2p\Bigr)
+|v|^{p-2}u^2\Bigl(\log u^2+\frac2p\Bigr)
\Bigr]\,d\mu\\
&\le C_\eta\int_V
\bigl(|u|^{p-2}|v|^{2+\eta}
     +|v|^{p-2}|u|^{2+\eta}\bigr)\,d\mu\\
&\le C\|(u,v)\|_{\mathcal H}^{p+\eta},
\end{aligned}
\]
where the logarithmic products are extended by zero
at zero. Since
\[
2^{1-p}\|(u,v)\|_{\mathcal H}^p
\le \|u\|_{H_a}^p+\|v\|_{H_b}^p
\]
and $(u,v)\ne(0,0)$, we obtain
\[
\|(u,v)\|_{\mathcal H}
\ge \left(\frac{2^{1-p}}{C}\right)^{1/\eta}
=:\delta>0.
\]

On the other hand, using the boundedness of the $L^\infty$ norm, which is
available because the embedding holds in the discrete setting of this section, we estimate the
logarithmic terms respectively:
\[
\int_V|u|^{p-2}v^2\log v^2\,d\mu
   \le C\int_V|u|^{p-2}|v|^{2+\varepsilon}\,d\mu
   \le C \|v\|_{L^{\infty}}^{\varepsilon}\int_V |u|^{p-2} v^2 d\mu
   \le C\|v\|_{\mathcal{H}_b}^{\varepsilon}\int_V |u|^{p-2} v^2 d\mu,
\]
\[
\int_V|v|^{p-2}u^2\log u^2\,d\mu
   \le C\int_V|v|^{p-2}|u|^{2+\varepsilon}\,d\mu
   \le C \|u\|_{L^{\infty}}^{\varepsilon}\int_V |v|^{p-2} u^2 d\mu
   \le C\|u\|_{\mathcal{H}_a}^{\varepsilon}\int_V |v|^{p-2} u^2 d\mu.
\]

Now suppose, for contradiction, that the infimum of $J$ over $\mathcal N$
equals zero.  Then for any $\varepsilon_1>0$ there exists
$(u,v)\in\mathcal{N}$ such that $J(u,v)<\varepsilon_1$.
The identity
\[
J(u,v)-\frac1p J'(u,v)\cdot (u,v)
      =\frac{2}{p^2}\int_V \bigl(|u|^{p-2}v^2+|v|^{p-2}u^2\bigr)\,d\mu
\]
implies
\[
\frac{2}{p^2}\int_V \bigl(|u|^{p-2}v^2+|v|^{p-2}u^2\bigr)\,d\mu < \varepsilon_1 .
\]
Inserting this together with the logarithmic estimates above
into~\eqref{eq:Nehari-id} yields
\[
\|u\|_{\mathcal H_a}^p+\|v\|_{\mathcal H_b}^p
   \le C \varepsilon_1\bigl(\|u\|_{\mathcal{H}_a}^{\varepsilon}
                         +\|v\|_{\mathcal{H}_b}^{\varepsilon}\bigr)
     +C\varepsilon_1,
\]
which in turn gives
\[
\|(u,v)\|_{\mathcal{H}}^p \le C \varepsilon_1 \|(u,v)\|_{\mathcal{H}}^{\varepsilon}+C\varepsilon_1.
\]
Fix $\varepsilon>0$ such that $\varepsilon<p$. An elementary analysis of the resulting
inequality shows that for $\varepsilon_1$ sufficiently small we must have
$\|(u,v)\|_{\mathcal{H}}<\delta$, where $\delta>0$ is the uniform lower bound
obtained above.  This contradicts the definition of $\delta$, and therefore
the infimum cannot be zero. The proof is complete.
\end{proof}
\begin{lemma}\label{lem4.4}
  Suppose $(A_1)$ and $(A_2)$ hold, then $d>0$ obtained in Lemma~\ref{lem4.3} can be attained.
\end{lemma}
\begin{proof}
First of all, Lemma~\ref{lem:projection1} implies that $\mathcal{N}\neq \emptyset$. Taking a minimizing sequence $\{(u_k,v_k)\}\subset \mathcal{N}$ such that $\lim\limits_{k\to \infty} J(u_k,v_k)=d$, since $(u_k,v_k)\in \mathcal{N}$, $J'(u_k,v_k)\cdot(u_k,v_k)=0$, which immediately implies that $$d=\lim\limits_{k\to\infty} \left[J(u_k,v_k)-\frac{1}{p}J'(u_k,v_k)\cdot (u_k,v_k)\right]=\frac{2}{p^2}\lim\limits_{k\to\infty}\left(\int_V |u_k|^{p-2}v_k^2d\mu+\int_V|v_k|^{p-2}u_k^2 d\mu\right),$$
            Thus, we deduce that $\int_V |u_k|^{p-2} v_k^2 d\mu$ and $\int_V |v_k|^{p-2} u_k^2d\mu$ are uniformly bounded, namely there exists $C>0$ such that 
            \[\int_V |u_k|^{p-2}v_k^2 d\mu\le C,\quad \int_V |v_k|^{p-2} u_k^2d\mu\le C.\]
            
Set
\[
X_k:=\|u_k\|_{\mathcal H_a}^p+\|v_k\|_{\mathcal H_b}^p,
\]
inserting the Nehari identity, we find
\begin{eqnarray*}\label{eq:Xk0}
X_k &\le& \int_V|u_k|^{p-2}(v_k^2\log v_k^2)^{+}\,d\mu
      +\int_V|v_k|^{p-2}(u_k^2\log u_k^2)^{+}\,d\mu + C\\
    &\le& C \left(\int_V |u_k|^{p-2}|v_k|^{2+\varepsilon}d\mu+\int_V |v_k|^{p-2}|u_k|^{2+\varepsilon}d\mu \right)+C\\
    &\le& C \left(\|v_k\|_{L^{\infty}}^{\varepsilon}\int_V |u_k|^{p-2}v_k^2d\mu+\|u_k\|_{L^{\infty}}^{\varepsilon}\int_V |v_k|^{p-2}u_k^2 d\mu \right) +C\\
    &\le& C\left(\|u_k\|_{L^{\infty}}^{\varepsilon}+\|v_k\|_{L^{\infty}}^{\varepsilon}\right)+C\\
    &\le& C\left(\|u_k\|_{\mathcal{H}_a}^{\varepsilon}+\|v_k\|_{\mathcal{H}_a}^{\varepsilon}\right)+C,\\
\end{eqnarray*}
from which we obtain 
\[
   X_k \le C X_k^{\frac{\varepsilon}{p}}+C.
\]
Since \(p>4\), by choosing \(\varepsilon\) small enough, \(1-\varepsilon/p > 0\), and is therefore \(X_k\) is uniformly bounded, which means the original minimising sequence
\(\{(u_k,v_k)\}\) is bounded in \(\mathcal H\).

Using the compact embedding Lemma~\ref{lem4}, we pass to a subsequence, still denoted
\(\{(u_k,v_k)\}\), and obtain a limit \((u_0,v_0)\in\mathcal H\) such that
\[
\begin{cases}
    (u_k,v_k)\rightharpoonup (u_0,v_0) \text{ weakly in } \mathcal{H},\\
    (u_k,v_k)\rightarrow (u_0,v_0) \text{ in } L^{p_1}(V)\times L^{p_2}(V) \text{ for }p_1\ge p\text{ and } p_2\ge p,\\
    (u_k,v_k)\rightarrow (u_0,v_0) \text{ pointwisely in } V.
\end{cases}
\]
By strong $L^p$ convergence,
\[
M(u_0,v_0)=\lim_{k\to\infty}M(u_k,v_k)
=\frac{p^2}{2}d>0.
\]
Hence $(u_0,v_0)\ne(0,0)$. Then, by the weak-lower semi-continuity of norm, together with Fatou's lemma, we obtain 
\allowdisplaybreaks
\begin{align*}
    & \int_V\left(|\nabla u_0|^p+a(x)|u_0|^p+|\nabla v_0|^p+b(x)|v_0|^p\right)d\mu -\frac{2}{p}\int_V|u_0|^{p-2}v_0^2 d\mu-\frac{2}{p}\int_V|v_0|^{p-2}u_0^2d\mu\\
    & \quad - \int_V |u_0|^{p-2}\left(v_0^2\log v_0^2\right)^{-} d\mu-\int_V |v_0|^{p-2}\left(u_0^2\log u_0^2\right)^{-} d\mu\\
    & \le \liminf\limits_{k\to\infty}\left[\int_V\left(|\nabla u_k|^p+a(x)|u_k|^p+|\nabla v_k|^p+b(x)|v_k|^p\right)d\mu -\frac{2}{p}\int_V|u_k|^{p-2}v_k^2 d\mu-\frac{2}{p}\int_V|v_k|^{p-2}u_k^2d\mu\right]\\
    & \quad + \liminf\limits_{k\to\infty}\left[-\int_V |u_k|^{p-2}\left(v_k^2\log v_k^2\right)^{-} d\mu-\int_V |v_k|^{p-2}\left(u_k^2\log u_k^2\right)^{-} d\mu\right]\\
    & \le \liminf\limits_{k\to \infty} \int_V\left[|u_k|^{p-2}\left(v_k^2\log v_k^2\right)^{+}+|v_k|^{p-2}\left(u_k^2\log u_k^2\right)^{+}\right]d\mu,
\end{align*}
            where
            \begin{eqnarray*}
                &&\int_V \left[|u_k|^{p-2}\left(v_k^2\log v_k^2\right)^{+}+|v_k|^{p-2}\left(u_k^2\log u_k^2\right)^{+}\right] d\mu\\
                &\le&\int_V \left(|u_k|^{p-2}|v_k|^{2+\varepsilon}+|v_k|^{p-2}|u_k|^{2+\varepsilon}\right) d\mu\\
                &\le& C \int_V \left(|u_k|^{p+\varepsilon}+|v_k|^{p+\varepsilon}\right) d\mu.
            \end{eqnarray*}
            Since \(u_k\to u\) and \(v_k\to v\) in \(L^{p+\varepsilon}\), together with pointwise convergence, the Vitali convergence theorem yields
           \allowdisplaybreaks
\begin{align*}
\lim_{k\to\infty}\int_V \bigl[ & |u_k|^{p-2}(v_k^2\log v_k^2)^+ + |v_k|^{p-2}(u_k^2\log u_k^2)^+ \bigr]\,d\mu \\
&= \int_V \bigl[ |u_0|^{p-2}(v_0^2\log v_0^2)^+ + |v_0|^{p-2}(u_0^2\log u_0^2)^+ \bigr]\,d\mu .
\end{align*}
            Thus,
            \begin{eqnarray*}
                &&\int_V\left(|\nabla u_0|^p+a(x)|u_0|^p+|\nabla v_0|^p+b(x)|v_0|^p\right)d\mu -\frac{2}{p}\int_V|u_0|^{p-2}v_0^2 d\mu-\frac{2}{p}\int_V|v_0|^{p-2}u_0^2d\mu\\
                &-& \int_V |u_0|^{p-2}v_0^2\log v_0^2 d\mu-\int_V |v_0|^{p-2} u_0^2\log u_0^2 d\mu\le 0.
            \end{eqnarray*}
            Now we claim that the left side of the above inequality is equal to 0. We argue by contradiction. If it were strictly less than 0, then
            \begin{align*}
    0 \le & \int_V\left(|\nabla u_0|^p+a(x)|u_0|^p+|\nabla v_0|^p+b(x)|v_0|^p\right)d\mu \\
    < {} & \frac{2}{p}\int_V |u_0|^{p-2}v_0^2 d\mu+\frac{2}{p}\int_V |v_0|^{p-2}u_0^2 d\mu+\int_V |u_0|^{p-2} v_0^2\log v_0^2 d\mu+\int_V |v_0|^{p-2} u_0^2\log u_0^2 d\mu,
         \end{align*}
            which means there exists $x\in V$ such that $u_0(x)\neq 0$ and $v_0(x)\neq 0$, then it immediately follows from Lemma~\ref{lem:projection1} that there exists a unique $t_0$ such that
            $t_0(u_0,v_0)\in \mathcal{N}$, $i.e.,$ 
            \begin{eqnarray*}
                0&=&t_0^p\left[\|u_0\|_{\mathcal{H}_a}^p +\|v_0\|_{\mathcal{H}_b}^p-\int_V(|v_0|^{p-2}u_0^2\log u_0^2+|u_0|^{p-2}v_0^2\log v_0^2)d\mu\right]\\
                &-& t_0^p\left(\frac{2}{p}\int_V |v_0|^{p-2}u_0^2 d\mu+\frac{2}{p}\int_V |u_0|^{p-2}v_0^2d\mu\right)-t_0^p\log t_0^2\left(\int_V|v_0|^{p-2} u_0^2d\mu+\int_V|u_0|^{p-2} v_0^2d\mu\right)\\
                &<& -t_0^p\log t_0^2\left(\int_V|v_0|^{p-2} u_0^2d\mu+\int_V|u_0|^{p-2} v_0^2d\mu\right),
            \end{eqnarray*}
            that is, $t_0^p\log t_0^2<0$, which implies $t_0\in(0,1)$, together with Fatou's Lemma to obtain
            \begin{eqnarray*}
                d&\le& J(t_0(u_0,v_0))=J(t_0(u_0,v_0))-\frac{1}{p}J'(t_0(u_0,v_0))\cdot(t_0(u_0,v_0))\\
                &=& \frac{2t_0^p}{p^2}\left(\int_V |u_0|^{p-2}v_0^2 d\mu+\int_V |v_0|^{p-2}u_0^2d\mu\right) < \frac{2}{p^2}\left(\int_V |u_0|^{p-2}v_0^2 d\mu+\int_V |v_0|^{p-2}u_0^2d\mu\right)\\
                &\le& \frac{2}{p^2} \liminf\limits_{k\to\infty}\left(\int_V|u_k|^{p-2}v_k^2 d\mu+\int_V|v_k|^{p-2}u_k^2d\mu\right)=\liminf\limits_{k\to\infty} J(u_k,v_k)=d,
            \end{eqnarray*} which is a contradiction, hence we deduce that 
            \begin{eqnarray*}
                &&\int_V\left(|\nabla u_0|^p+a(x)|u_0|^p+|\nabla v_0|^p+b(x)|v_0|^p\right)d\mu -\frac{2}{p}\int_V|u_0|^{p-2}v_0^2 d\mu-\frac{2}{p}\int_V|v_0|^{p-2}u_0^2d\mu\\
                &-& \int_V |u_0|^{p-2}v_0^2\log v_0^2 d\mu-\int_V |v_0|^{p-2} u_0^2\log u_0^2 d\mu = 0,
            \end{eqnarray*}
            which means $(u_0,v_0)\in \mathcal{N}$, 
            \begin{eqnarray*}
                d&\le& J(u_0,v_0)=J(u_0,v_0)-\frac{1}{p}J'(u_0,v_0)\cdot(u_0,v_0)\\
                &=& \frac{2}{p^2}\left(\int_V |u_0|^{p-2}v_0^2 d\mu+\int_V |v_0|^{p-2}u_0^2d\mu\right) 
                \le \frac{2}{p^2} \liminf\limits_{k\to\infty}\left(\int_V|u_k|^{p-2}v_k^2 d\mu+\int_V|v_k|^{p-2}u_k^2d\mu\right)\\
                &=&\liminf\limits_{k\to\infty} J(u_k,v_k)=d,
            \end{eqnarray*}
            Thus, $J(u_0,v_0)=d$ and the proof is completed.
  \end{proof}
\medskip
\noindent
\begin{remark}
    The estimate of the logarithmic terms relies entirely on the
global $L^\infty$ bound supplied by the compact embedding
$\mathcal{H}\hookrightarrow L^\infty(V)$.  Once this uniform pointwise
control is available, the logarithm reduces
to a subcritical power, which can be absorbed by elementary means.
\end{remark}

\begin{lemma}\label{lem:solution1}
The minimizer $(u_{0},v_{0})\in\mathcal{N}$ of $d$ is a solution of system~\eqref{eq:main}.
\end{lemma}

\begin{proof}
We argue by contradiction. Suppose that \((u_0,v_0)\in\mathcal N\) satisfies \(J(u_0,v_0)=d\) but is not a weak solution. Since \(J\) is Gateaux differentiable along directions in \(C_c(V)\times C_c(V)\) at \((u_0,v_0)\), there exists \((\phi,\psi)\in C_c(V)\times C_c(V)\) such that
\[
J'(u_0,v_0)\cdot(\phi,\psi)\le -1.
\]
For this fixed \((\phi,\psi)\), define
\[
G(s,\sigma):=J'(su_0+\sigma\phi,\,sv_0+\sigma\psi)\cdot(\phi,\psi),\qquad s,\sigma\in\mathbb R.
\]
Since \(\phi\) and \(\psi\) have finite support, \(G(s,\sigma)\) is a finite sum over the vertices in the supports of \(\phi\) and \(\psi\). Each term in this sum is continuous as a function of \((s,\sigma)\) near \((1,0)\), because the expressions involve only algebraic powers of \(s,\sigma\) and the fixed functions \(u_0,v_0,\phi,\psi\). Therefore \(G\) is continuous at \((s,\sigma)=(1,0)\). Consequently, there exists \(\delta>0\) such that
\[
J'(su_0+\sigma\phi,\,sv_0+\sigma\psi)\cdot(\phi,\psi)\le -\frac12
\qquad \text{for all } |s-1|<\delta,\ |\sigma|<\delta.
\]

Let \(\eta\) be a smooth cut-off with \(0\le\eta\le1\), \(\eta(s)=1\) near \(s=1\), and \(\eta(s)=0\) at \(s=1\pm\delta\). Define
\[
(u_s,v_s):=(su_0+\delta\eta(s)\phi,sv_0+\delta\eta(s)\psi),\qquad s\in[1-\delta,1+\delta].
\]
Then \((u_{1\pm\delta},v_{1\pm\delta})=((1\pm\delta)u_0,(1\pm\delta)v_0)\), and by the integral formula,
\[
J(u_s,v_s)-J(su_0,sv_0)
=\delta\eta(s)\int_0^1 J'\bigl(su_0+\tau\delta\eta(s)\phi,sv_0+\tau\delta\eta(s)\psi\bigr)
\cdot(\phi,\psi)\,d\tau
\le -\frac12\delta\eta(s).
\]
Hence
\[
J(u_s,v_s)\le J(su_0,sv_0)-\frac12\delta\eta(s).
\]
By Lemma~\ref{lem:projection1}, the fibre map \(t\mapsto J(tu_0,tv_0)\) has a unique global maximum at \(t=1\). Thus \(J(su_0,sv_0)<d\) for \(s\neq1\), and consequently
\[
J(u_s,v_s)<d
\]
for every \(s\in[1-\delta,1+\delta]\).

Define \(F(s):=J'(u_s,v_s)\cdot(u_s,v_s)\). From the fibre map property,
\[
F(1-\delta)>0,\qquad F(1+\delta)<0.
\]
By continuity, there exists \(s_1\in[1-\delta,1+\delta]\) such that \(F(s_1)=0\), that is \((u_{s_1},v_{s_1})\in\mathcal N\). But then \(J(u_{s_1},v_{s_1})<d\), contradicting the definition of \(d\). Therefore \((u_0,v_0)\) must be a weak solution.
\end{proof}
\subsection*{Proof of Theorem~\ref{thm2} \texorpdfstring{($\mathbb R^N$)}}
We now prove the existence result for the system~\eqref{eq:main1} in the
continuous critical setting \(p=N>4\). Since the original functional \(J\) is
not Fr\'echet differentiable and may even take the value \(+\infty\) (see
Example~\ref{ex:ill-posedness} in the Appendix), we work with the regularised
functional \(J_\varepsilon\). The argument follows the same Nehari manifold
strategy as in the discrete case, with the key difference that the embedding
\(\mathcal H\hookrightarrow L^\infty\) is no longer available; the compactness
estimates below therefore rely on the exponent calibration technique.
Throughout the following lemmas, \(\varepsilon\) is a fixed real number in
\((0,1/2)\) and \(\varepsilon_1\) is an arbitrary positive constant.
\begin{definition}\label{def:Jeps-p4}
We denote \(\eta_{\varepsilon}(t)=\log (t^2+\varepsilon)\).
For \(\varepsilon\in(0,\frac{1}{2})\), the regularised functional \(J_\varepsilon:\mathcal{G}\to\mathbb{R}\) is defined by
\begin{equation}\label{eq:Jeps-def-p4}
J_\varepsilon(u,v):=\frac1p\int_{\mathbb{R}^N}\bigl(|\nabla u|^p+|\nabla v|^p\bigr)dx
+\frac1p\int_{\mathbb{R}^N}\bigl(a|u|^p+b|v|^p\bigr)dx
-\frac1p\int_{\mathbb{R}^N}\bigl(|u|^{p-2}v^2\eta_\varepsilon(v)
+|v|^{p-2}u^2\eta_\varepsilon(u)\bigr)dx,
\end{equation}
where \(\mathcal{G}=\mathcal{G}_a\times\mathcal{G}_b\) and \((u,v)\in \mathcal{G}\), see Lemma~\ref{lem2.3} for the definition of \(\mathcal{G}_a\) and \(\mathcal{G}_b\). For every \(\sigma>0\),
\[
|\log(s^2+\varepsilon)|\le C_{\varepsilon,\sigma}(1+|s|^\sigma),
\qquad
\frac{|s|^3}{s^2+\varepsilon}\le |s|.
\]
Thus, since \(G_a,G_b\hookrightarrow L^q(\mathbb R^N)\) for \(q\ge N\), Hölder's inequality, the Nemytskii theorem, and dominated convergence yield
\[
J_\varepsilon\in C^1(G,\mathbb R),
\] and we omit the details for brevity.
The derivatives of \(J_{\varepsilon}\) are defined as
\begingroup
\allowdisplaybreaks
\begin{align}
    \langle J_\varepsilon'(u,v),(\phi,\psi)\rangle
    &= \int_{\mathbb{R}^N} \bigl(|\nabla u|^{p-2}\nabla u\cdot\nabla\phi
      +|\nabla v|^{p-2}\nabla v\cdot\nabla\psi\bigr)dx \nonumber \\
    &\quad + \int_{\mathbb{R}^N} \bigl(a|u|^{p-2}u\phi+b|v|^{p-2}v\psi\bigr)dx \nonumber \\
    &\quad - \frac{p-2}{p}\int_{\mathbb{R}^N}|u|^{p-4}u\phi\,v^2\eta_\varepsilon(v)dx
      -\frac{2}{p}\int_{\mathbb{R}^N}|u|^{p-2}v\psi\,\eta_\varepsilon(v)dx \nonumber \\
    &\quad - \frac{2}{p}\int_{\mathbb{R}^N}|u|^{p-2}\frac{v^3\psi}{v^2+\varepsilon}dx
      -\frac{p-2}{p}\int_{\mathbb{R}^N}|v|^{p-4}v\psi\,u^2\eta_\varepsilon(u)dx \nonumber \\
    &\quad - \frac{2}{p}\int_{\mathbb{R}^N}|v|^{p-2}u\phi\,\eta_\varepsilon(u)dx
      -\frac{2}{p}\int_{\mathbb{R}^N}|v|^{p-2}\frac{u^3\phi}{u^2+\varepsilon}dx.
      \label{eq:Jeps-derivative}
\end{align}
\endgroup
The Nehari manifold \(\mathcal{N}_{\varepsilon}\) is defined as in Definition~\ref{def:critical}; moreover, we denote \(\|(u,v)\|_{\mathcal{G}}=\|u\|_{\mathcal{G}_a}+\|v\|_{\mathcal{G}_b}\).
\end{definition}

\begin{lemma}\label{lem4.6}
If $(u,v)$ satisfies that both \(u(x)\neq 0\) and \(v(x)\neq 0\) in a set of positive measure, that is, \(u(x)\cdot v(x)\neq 0\) in some positive set. Then there exists $t^{*}>0$ such that $t^{*}(u,v)\in\mathcal{N_{\varepsilon}}$.
\end{lemma}
\begin{proof}
Define
\[
h(t):=\langle J_\varepsilon'(tu,tv),(tu,tv)\rangle .
\]
A direct computation gives
\[
\begin{aligned}
\frac{h(t)}{t^p}
&=S(u,v)
-\log t^2\,M(u,v)
-L_{\varepsilon/t^2}(u,v)
-\frac{2}{p}R_{\varepsilon/t^2}(u,v),
\end{aligned}
\]
where
\[
S(u,v):=\|u\|_{\mathcal {G}_a}^p+\|v\|_{\mathcal {G}_b}^p,
\]
\[
M(u,v):=\int_{\mathbb R^N}\bigl(|u|^{p-2}v^2+|v|^{p-2}u^2\bigr)\,dx,
\]
\[
\begin{aligned}
L_{\varepsilon/t^2}(u,v)
&:=\int_{\mathbb R^N}\Bigl(|u|^{p-2}v^2\log\Bigl(v^2+\frac{\varepsilon}{t^2}\Bigr)
+|v|^{p-2}u^2\log\Bigl(u^2+\frac{\varepsilon}{t^2}\Bigr)\Bigr)\,dx,\\
R_{\varepsilon/t^2}(u,v)
&:=\int_{\mathbb R^N}\Bigl(|u|^{p-2}\frac{v^4}{v^2+\varepsilon/t^2}
+|v|^{p-2}\frac{u^4}{u^2+\varepsilon/t^2}\Bigr)\,dx .
\end{aligned}
\]
As \(t\to0^+\), we obtain \(R_{\varepsilon/t^2}\to0\) and
\[
L_{\varepsilon/t^2}(u,v)
=(\log\varepsilon-2\log t)M(u,v)+o(1).
\]
Consequently,
\[
\lim_{t\to0^+}\frac{h(t)}{t^p}
=S(u,v)-\log\varepsilon\,M(u,v)>0,
\]
because \(\varepsilon\in(0,1/2)\) and \(M(u,v)>0\). Hence \(h(t)>0\) for all sufficiently small \(t>0\).

On the other hand, as \(t\to+\infty\), the term
\[
-\log t^2\,M(u,v)
\]
dominates all other contributions, since \(M(u,v)>0\), while \(R_{\varepsilon/t^2}\) is bounded above by \(M(u,v)\) and \(L_{\varepsilon/t^2}\) grows at most logarithmically in \(t\). Therefore
\[
\frac{h(t)}{t^p}\to-\infty,
\]
so \(h(t)<0\) for all sufficiently large \(t\).

Since \(h\) is continuous, there exists \(t^*>0\) such that
\[
h(t^*)=0 .
\]
This is exactly
\[
\langle J_\varepsilon'(t^*u,t^*v),\,(t^*u,t^*v)\rangle=0,
\]
and therefore \(t^*(u,v)\in\mathcal N_\varepsilon\).
\end{proof}

\begin{lemma}\label{lem:coercivity}
Under the hypotheses of Theorem~\ref{thm2}, for any $(u,v)\in \mathcal{N}_{\varepsilon}$, there exists $\delta_1>0$ such that $\|(u,v)\|_{\mathcal{G}}\ge \delta_1$. Moreover, $d_{\varepsilon}=\inf\limits_{(u,v)\in \mathcal{N}_{\varepsilon}} J_{\varepsilon}(u,v)>0$ and it can be attained.
\end{lemma}
\begin{proof}
$\mathcal{N}_{\varepsilon}\neq \emptyset$, which immediately follows from Lemma~\ref{lem4.6}. Fix $\eta\in(0,1)$. For $0<\varepsilon<1/2$,
\[
\log(s^2+\varepsilon)
+\frac2p\frac{s^2}{s^2+\varepsilon}
\le C_\eta |s|^\eta,
\]
with $C_\eta$ independent of $\varepsilon$.
Indeed, for sufficiently small $|s|$, the left-hand
side is bounded above by
$\log(s^2+1/2)+2/p<0$, since $p>4$;
away from zero, the estimate follows from the
logarithmic growth bound.

For $(u,v)\in\mathcal N_\varepsilon$, the Nehari
identity therefore gives
\[
\begin{aligned}
\|u\|_{G_a}^p+\|v\|_{G_b}^p
&=\int_{\mathbb R^N}
|u|^{p-2}v^2
\left(\log(v^2+\varepsilon)
+\frac2p\frac{v^2}{v^2+\varepsilon}\right)\,dx\\
&\quad+\int_{\mathbb R^N}
|v|^{p-2}u^2
\left(\log(u^2+\varepsilon)
+\frac2p\frac{u^2}{u^2+\varepsilon}\right)\,dx\\
&\le C_\eta\int_{\mathbb R^N}
\bigl(|u|^{p-2}|v|^{2+\eta}
     +|v|^{p-2}|u|^{2+\eta}\bigr)\,dx\\
&\le C\|(u,v)\|_{\mathcal G}^{p+\eta},
\end{aligned}
\]
where the last step follows from H\"older's inequality
and the embeddings into $L^{p+\eta}(\mathbb R^N)$.
Together with
\[
2^{1-p}\|(u,v)\|_{\mathcal G}^p
\le \|u\|_{G_a}^p+\|v\|_{G_b}^p,
\]
this yields
\[
\|(u,v)\|_{\mathcal G}
\ge \left(\frac{2^{1-p}}{C}\right)^{1/\eta}
=:\delta_1>0.
\]
The constant $\delta_1$ is independent of
$\varepsilon\in(0,1/2)$. We assume that \(\inf\limits_{(u,v)\in \mathcal{N}_{\varepsilon}}J_{\varepsilon}>0\) does not hold, then for any \(\varepsilon_1>0\), there exists \((u,v)\in \mathcal{N}_{\varepsilon}\) such that \(J_{\varepsilon}(u,v)<\varepsilon_1\), which means
$$J_{\varepsilon}(u,v)=J_{\varepsilon}(u,v)-\frac1p J_{\varepsilon}'(u,v)\cdot (u,v)=\frac{2}{p^2} \int_{\mathbb{R}^N} \left(|v|^{p-2}\frac{u^4}{u^2+\varepsilon}+|u|^{p-2}\frac{v^4}{v^2+\varepsilon}\right)dx< \varepsilon_1,$$
and is therefore
\begin{align*}
\|u\|_{\mathcal{G}_a}^p+\|v\|_{\mathcal{G}_b}^p &\le \int_{\mathbb{R}^N} \big[ |u|^{p-2}v^2 \log (v^2+\varepsilon) + |v|^{p-2}u^2\log (u^2+\varepsilon) \big] dx + p \varepsilon_1\\
&\le \int_{\{v^2\ge \varepsilon\}} |u|^{p-2}v^2 \log (2v^2)dx + \int_{\{u^2\ge \varepsilon\}} |v|^{p-2}u^2 \log (2u^2)dx+p\varepsilon_1\\
&\le \log 2\left(\int_{\{v^2\ge \varepsilon\}}  |u|^{p-2}v^2 dx+\int_{\{u^2\ge \varepsilon\}}  |v|^{p-2}u^2 dx\right)\\
&\quad +\int_{\mathbb{R}^N} |u|^{p-2}(v^2\log v^2)^{+} dx+\int_{\mathbb{R}^N}|v|^{p-2}(u^2\log u^2)^{+} dx.
\end{align*}
For the previous two terms, since
$$\int_{\{u^2\ge\varepsilon\}} |v|^{p-2}u^2 dx+\int_{\{v^2\ge\varepsilon\}} |u|^{p-2}v^2 dx\le 2 \left(\int_{\{u^2\ge \varepsilon\}} |v|^{p-2}\frac{u^4}{u^2+\varepsilon}dx+\int_{\{v^2\ge \varepsilon\}} |u|^{p-2}\frac{v^4}{v^2+\varepsilon}dx\right),$$
there is
$$\log 2\left(\int_{\{v^2\ge \varepsilon\}}  |u|^{p-2}v^2 dx+\int_{\{u^2\ge \varepsilon\}}  |v|^{p-2}u^2 dx\right)\le p^2 \varepsilon_1 \log 2.$$
For the latter two terms, we deploy the exponent calibration technique to cope with the difficulty caused by the failure of \(L^{\infty}\) embedding,
\begin{align*}
\int_{\mathbb{R}^N} |v|^{p-2}(u^2\log u^2)^{+} dx &\le C_{\tau}\int_{\{u^2\ge 1\}} |v|^{p-2}|u|^{2+\tau} dx\\
&\le C_{\tau} \int_{\{u^2\ge 1\}}\left(|u|^{p_1}|v|^{p_2}\right)\cdot (|u|^{2+\tau-p_1}|v|^{p-p_2-2})dx\\
&\le C_{\tau} \left(\int_{\{u^2\ge 1\}} |u|^{p_1 s}|v|^{p_2s}dx\right)^{\frac{1}{s}}\left(\int_{\{u^2\ge 1\}} |u|^{\frac{(2+\tau-p_1)s}{s-1}}|v|^{\frac{(p-p_2-2)s}{s-1}}dx\right)^{\frac{s-1}{s}}\\
&\le C_{\tau} \left(\int_{\{u^2\ge 1\}} |u|^{p_1 s}|v|^{p_2s}dx\right)^{\frac{1}{s}} \\
&\quad \cdot \left[ \int_{\{u^2\ge 1\}} \left( \frac{1}{t} |u|^{\frac{(2+\tau-p_1)st}{s-1}}+\frac{t-1}{t}|v|^{\frac{(p-p_2-2)st}{(s-1)(t-1)}} \right)dx \right]^{\frac{s-1}{s}},
\end{align*}
where we use H\"{o}lder's inequality and Young's inequality, and \(p_1,p_2,s,t\) are positive real numbers to be chosen later. Now we let
\begin{equation}\label{eq:param-choice}
\begin{cases}
p_1 s=2,\\
p_2 s=p-2,\\
\frac{(2+\tau-p_1)st}{s-1}=\frac{(p-p_2-2)st}{(s-1)(t-1)}=2p,
\end{cases}
\end{equation}
from which we get $p_1=2\left(1-\frac{\tau}{p}\right)$, $p_2=(p-2)\left(1-\frac{\tau}{p}\right)$, $s=\frac{p}{p-\tau}$, $t=\frac{2p}{p+2}.$
When \(\tau\in(0,\frac p2)\), we have \(s\in (1,2)\), and \(t>1\) is obvious, which means the values of $s,t,p_1,p_2$ set above are reasonable. On the other hand,
$$\int_{\{u^2\ge 1\}} u^2 |v|^{p-2}dx\le \int_{\{u^2\ge \varepsilon\}} u^2 |v|^{p-2}dx\le p^2 \varepsilon_1.$$
Thus,
\begin{align*}
\int_{\mathbb{R}^N} |v|^{p-2}(u^2\log u^2)^{+} dx &\le C\left(\int_{\{u^2\ge 1\}}u^2 |v|^{p-2}dx\right)^{\frac1s}\left[\int_{\{u^2\ge 1\}}(|u|^{2p}+|v|^{2p})dx\right]^{\frac{s-1}{s}}\\
&\le C \varepsilon_1^{\frac1s}\left(\|u\|_{\mathcal{G}_a}^{2p}+\|v\|_{\mathcal{G}_b}^{2p}\right)^{\frac{s-1}{s}}.
\end{align*}
Similarly, we also have
$$\int_{\mathbb{R}^N} |u|^{p-2}(v^2\log v^2)^{+}dx\le C \varepsilon_1^{\frac1s}\left(\|u\|_{\mathcal{G}_a}^{2p}+\|v\|_{\mathcal{G}_b}^{2p}\right)^{\frac{s-1}{s}}.$$
All in all, we obtain 
$$\|(u,v)\|_{\mathcal{G}}^p\le C\varepsilon_1^{\frac1s}\|(u,v)\|_{\mathcal{G}}^{\frac{2p(s-1)}{s}}+C\varepsilon_1.$$
Since \(s\in (1,2)\), \(\frac{2p(s-1)}{s}<p\), it immediately follows from elementary analysis of the resulting inequality that \(\|(u,v)\|_{\mathcal{G}}\le \delta(\varepsilon_1)\) and \(\delta(\varepsilon_1)\to 0\) as \(\varepsilon_1\to 0\), which is a contradiction to the fact that \(\|(u,v)\|_{\mathcal{G}}\ge \delta_{\varepsilon}\) for all \((u,v)\in \mathcal{N}_{\varepsilon}\), and is thus \(d_\varepsilon>0\). It remains to prove that the minimum \(d_{\varepsilon}\) can be attained.

Taking a minimizing sequence \(\{(u_k,v_k)\}\subset \mathcal{N}_{\varepsilon}\) such that \(\lim\limits_{k\to \infty} J(u_k,v_k)=d_{\varepsilon}\), since \((u_k,v_k)\in \mathcal{N}_{\varepsilon}\), \(J_{\varepsilon}'(u_k,v_k)\cdot(u_k,v_k)=0\), which immediately implies that
$$d_{\varepsilon}=\lim_k \left[J_{\varepsilon}(u_k,v_k)-\frac{1}{p}J_{\varepsilon}'(u_k,v_k)\cdot (u_k,v_k)\right]
=\frac{2}{p^2}\lim_k\int_{\mathbb{R}^N}
\left(|u_k|^{p-2}\frac{v_k^4}{v_k^2+\varepsilon}
+|v_k|^{p-2}\frac{u_k^4}{u_k^2+\varepsilon}\right)dx,$$
\begin{align*}
\|u_k\|_{\mathcal{G}_a}^p+\|v_k\|_{\mathcal{G}_b}^p &= \int_{\mathbb{R}^N} |u_k|^{p-2}v_k^2\log (v_k^2+\varepsilon)dx + \int_{\mathbb{R}^N} |v_k|^{p-2}u_k^2\log (u_k^2+\varepsilon)dx \\
&\quad + \frac{2}{p}\left( \int_{\mathbb{R}^N} |u_k|^{p-2}\frac{v_k^4}{v_k^2+\varepsilon} dx + \int_{\mathbb{R}^N} |v_k|^{p-2}\frac{u_k^4}{u_k^2+\varepsilon} dx \right).
\end{align*}
As before, we use the exponent calibration technique to obtain
$$\|(u_k,v_k)\|_{\mathcal{G}}^p\le C_1\|(u_k,v_k)\|_{\mathcal{G}}^{\frac{2p(s-1)}{s}}+C_2,$$
and here the exponent calibration technique is essential due to the failure of \(L^{\infty}\) 
Thus it is easy to check that \(\|(u_k,v_k)\|_{\mathcal{G}}\le C\), which means it is uniformly bounded. It follows from Lemma~\ref{lem2.3} that
\[\begin{cases}
(u_k,v_k)\rightharpoonup (u_0,v_0) \text{ weakly in } \mathcal{G},\\
(u_k,v_k)\rightarrow (u_0,v_0) \text{ in } L^{p_1}(\mathbb{R}^N)\times L^{p_2}(\mathbb{R}^N) \text{ for } (p_1,p_2)\in [p,+\infty)\times [p,+\infty), \\
(u_k,v_k)\rightarrow (u_0,v_0) \text{ a.e. in } \mathbb{R}^N.
\end{cases}\]
By strong $L^p$ convergence,
\[
R_\varepsilon(u_0,v_0)
=\lim_{k\to\infty}R_\varepsilon(u_k,v_k)
=\frac{p^2}{2}d_\varepsilon>0.
\]
Hence $(u_0,v_0)\ne(0,0)$. By the weak-lower semi-continuity of norm and applying Fatou's lemma to the nonnegative parts, we obtain
\begin{align*}
\int_{\mathbb{R}^N} (|\nabla{u_0}|^p+a|u_0|^p+|\nabla{v_0}|^p&+b|v_0|^p)dx - \frac2p \int_{\mathbb{R}^N} \left(|u_0|^{p-2}\frac{v_0^4}{v_0^2+\varepsilon}+ |v_0|^{p-2}\frac{u_0^4}{u_0^2+\varepsilon}\right) dx\\
&- \int_{\mathbb{R}^N} \left\{|u_0|^{p-2}\left[v_0^2\log(v_0^2+\varepsilon)\right]^{-}+|v_0|^{p-2}\left[u_0^2\log(u_0^2+\varepsilon)\right]^{-}\right\}dx\\
&\le \liminf_{k\to\infty} \Bigg[ \int_{\mathbb{R}^N} (|\nabla u_k|^p+a|u_k|^p+|\nabla v_k|^p+b|v_k|^p) dx \\
&\qquad - \frac2p \int_{\mathbb{R}^N} \left(|u_k|^{p-2}\frac{v_k^4}{v_k^2+\varepsilon}+ |v_k|^{p-2}\frac{u_k^4}{u_k^2+\varepsilon}\right) dx \\
&\qquad - \int_{\mathbb{R}^N} \left\{|u_k|^{p-2}\left[v_k^2\log(v_k^2+\varepsilon)\right]^{-}+|v_k|^{p-2}\left[u_k^2\log(u_k^2+\varepsilon)\right]^{-}\right\}dx \Bigg]\\
&=\liminf_{k\to\infty} \left[\int_{\mathbb{R}^N} |u_k|^{p-2}\left[v_k^2\log(v_k^2+\varepsilon)\right]^{+}+|v_k|^{p-2}\left[u_k^2\log(u_k^2+\varepsilon)\right]^{+} dx\right]\\
&=\int_{\mathbb{R}^N} |u_0|^{p-2}[v_0^2\log (v_0^2+\varepsilon)]^{+}dx+\int_{\mathbb{R}^N} |v_0|^{p-2}[u_0^2\log (u_0^2+\varepsilon)]^{+}dx.
\end{align*}
The last inequality follows from Vitali convergence theorem because \(u_k\to u_0, v_k\to v_0\) strongly in \(L^{p+\sigma}\) for \(\sigma>0\). Assume, for contradiction, that the inequality in the previous display is strict, i.e.,
\begin{equation}\label{eq:strict-ineq}
\begin{aligned}
\int_{\mathbb{R}^N} (|\nabla{u_0}|^p+a|u_0|^p+|\nabla{v_0}|^p&+b|v_0|^p)dx < \frac2p \int_{\mathbb{R}^N} \left(|u_0|^{p-2}\frac{v_0^4}{v_0^2+\varepsilon}+ |v_0|^{p-2}\frac{u_0^4}{u_0^2+\varepsilon}\right) dx\\
&+ \int_{\mathbb{R}^N} \left[|u_0|^{p-2}v_0^2\log(v_0^2+\varepsilon)+|v_0|^{p-2}u_0^2\log(u_0^2+\varepsilon)\right]dx.
\end{aligned}
\end{equation}
Lemma~\ref{lem4.6} implies that there exists \(t_0>0\) such that \(t_0(u_0,v_0)\in\mathcal{N}_{\varepsilon}\), which means
\begin{equation}\label{eq:t0-nehari}
\begin{aligned}
0=t_0^p\left(\|u_0\|_{\mathcal{G}_a}^p+\|v_0\|_{\mathcal{G}_b}^p\right)&-\frac{2t_0^p}{p}\int_{\mathbb{R}^N} \left(|u_0|^{p-2}\frac{t_0^2v_0^4}{t_0^2v_0^2+\varepsilon}+ |v_0|^{p-2}\frac{t_0^2u_0^4}{t_0^2u_0^2+\varepsilon}\right) dx\\
&-t_0^p\int_{\mathbb{R}^N} \left[|u_0|^{p-2}v_0^2\log(t_0^2v_0^2+\varepsilon)+|v_0|^{p-2}u_0^2\log(t_0^2u_0^2+\varepsilon)\right]dx,
\end{aligned}
\end{equation}
which immediately implies that
\begin{align*}
0 < t_0^p \Biggl[ &\frac{2}{p} \int_{\mathbb{R}^N} |u_0|^{p-2}\frac{v_0^4}{v_0^2+\varepsilon} \, dx + \frac{2}{p} \int_{\mathbb{R}^N} |v_0|^{p-2}\frac{u_0^4}{u_0^2+\varepsilon} \, dx \\
&+ \int_{\mathbb{R}^N} |u_0|^{p-2}v_0^2\log(v_0^2+\varepsilon) \, dx + \int_{\mathbb{R}^N} |v_0|^{p-2}u_0^2\log(u_0^2+\varepsilon) \, dx \\
&- \frac{2}{p} \int_{\mathbb{R}^N} |u_0|^{p-2}\frac{t_0^2v_0^4}{t_0^2v_0^2+\varepsilon} \, dx - \frac{2}{p} \int_{\mathbb{R}^N} |v_0|^{p-2}\frac{t_0^2u_0^4}{t_0^2u_0^2+\varepsilon} \, dx \\
&- \int_{\mathbb{R}^N} |u_0|^{p-2}v_0^2\log(t_0^2v_0^2+\varepsilon) \, dx - \int_{\mathbb{R}^N} |v_0|^{p-2}u_0^2\log(t_0^2u_0^2+\varepsilon) \, dx \Biggr],
\end{align*}
from which we obtain \(t_0\in (0,1)\) due to the monotone increasing property of the functions \(\frac{t^2u_0^4}{t^2u_0^2+\varepsilon}\), \(\frac{t^2v_0^4}{t^2v_0^2+\varepsilon}\), \(|u_0|^{p-2}v_0^2\log(t^2v_0^2+\varepsilon)\) and \(|v_0|^{p-2}u_0^2\log(t^2u_0^2+\varepsilon)\) with respect to \(t\). Therefore,
\begin{equation}\label{eq:final-contra}
\begin{aligned}
d_{\varepsilon}\le J_{\varepsilon}(t_0(u_0,v_0))&=J_{\varepsilon}(t_0(u_0,v_0))-\frac1p J_{\varepsilon}'(t_0(u_0,v_0))\cdot t_0(u_0,v_0)\\
&=\frac{2t_0^p}{p^2} \int_{\mathbb{R}^N} \left(|u_0|^{p-2}\frac{t_0^2v_0^4}{t_0^2v_0^2+\varepsilon}+ |v_0|^{p-2}\frac{t_0^2u_0^4}{t_0^2u_0^2+\varepsilon}\right) dx\\
&< \frac{2}{p^2} \int_{\mathbb{R}^N} \left(|u_0|^{p-2}\frac{v_0^4}{v_0^2+\varepsilon}+ |v_0|^{p-2}\frac{u_0^4}{u_0^2+\varepsilon}\right) dx\\
&\le \liminf_{k\to\infty} \frac{2}{p^2} \int_{\mathbb{R}^N} \left(|u_k|^{p-2}\frac{v_k^4}{v_k^2+\varepsilon}+ |v_k|^{p-2}\frac{u_k^4}{u_k^2+\varepsilon}\right) dx\le d_{\varepsilon},
\end{aligned}
\end{equation}
which is a contradiction. Thus \((u_0,v_0)\in \mathcal{N}_{\varepsilon}\) and \(J_{\varepsilon}(u_0,v_0)=d_{\varepsilon}\), i.e., \((u_0,v_0)\) is a ground state to equation \eqref{eq:main1}, then we complete the proof of Theorem~\ref{thm2}.
\end{proof}

\begin{remark}\label{rem:regularised-nehari}
For fixed \(\varepsilon>0\), the regularised functional \(J_\varepsilon\) is of class \(C^1\). Define
\[
G_\varepsilon(u,v):=\langle J_\varepsilon'(u,v),(u,v)\rangle .
\]
For any \((u,v)\in\mathcal N_\varepsilon\), a direct computation along the fibre gives
\[
\begin{aligned}
\langle G_\varepsilon'(u,v),(u,v)\rangle
&=
-\frac{2(p-2)}{p}\int_{\mathbb R^N}|u|^{p-2}\frac{v^4}{v^2+\varepsilon}\,dx
-\frac{4}{p}\int_{\mathbb R^N}|u|^{p-2}v^4\frac{v^2+2\varepsilon}{(v^2+\varepsilon)^2}\,dx\\
&\quad
-\frac{2(p-2)}{p}\int_{\mathbb R^N}|v|^{p-2}\frac{u^4}{u^2+\varepsilon}\,dx
-\frac{4}{p}\int_{\mathbb R^N}|v|^{p-2}u^4\frac{u^2+2\varepsilon}{(u^2+\varepsilon)^2}\,dx .
\end{aligned}
\]
All four terms are nonpositive, and at least one of the first and third terms is strictly negative: if both vanished, then \(u=0\) or \(v=0\) almost everywhere, and the Nehari identity would force \(u=0\) and \(v=0\) almost everywhere, contradicting \((u,v)\neq(0,0)\). Hence
\[
\langle G_\varepsilon'(u,v),(u,v)\rangle < 0 .
\]
So \(\mathcal N_\varepsilon\) is a \(C^1\) codimension-one submanifold. By the Lagrange multiplier rule, any constrained minimiser \((u_\varepsilon,v_\varepsilon)\in\mathcal N_\varepsilon\) satisfies \(J_\varepsilon'(u_\varepsilon,v_\varepsilon)=0\), and is therefore a free critical point of \(J_\varepsilon\). This verification does not rely on the uniqueness of fibre-map critical points.
\end{remark}

\section{Proof of the Cerami condition}\label{Cerami}
This section establishes the existence of ground states for the discrete well-posed problem by verifying the Cerami condition and invoking the mountain pass theorem. In view of Lemma 2.16 of \cite{WIL96}, we can easily conclude that \(J\in C^1(\mathcal{H},\mathbb{R})\).
\begin{lemma}\label{lem:cerami}
Assume that {\rm(A1)} and {\rm(A2')} hold and let \(p>4\),
then the functional \(J\) satisfies the Cerami condition at every
positive level \(c\).
\end{lemma}
\begin{proof}
To demonstrate the applicability of the exponent calibration technique in the discrete setting as well, the proof is built upon the generic calibration system as below:
\begin{equation}\label{eq:cal-gen}
\begin{cases}
p_1 s = \alpha,\quad p_2 s = \beta,\\[2pt]
(\tilde a+\varepsilon-p_1)s't = p,\\[2pt]
(\tilde b-p_2)s'\dfrac{t}{t-1}=p,
\end{cases}
\end{equation}
where \(s>1\), \(t>1\), \(s'=s/(s-1)\), and \((\alpha,\beta)\) are the
exponents of the available a priori mixed bound
\(\int_V|u|^\alpha|v|^\beta\,d\mu\le C\).  For the relevant choice
\((\alpha,\beta)=(p,p)\) and \((\tilde a,\tilde b)=(2,p-2)\) the system
admits the explicit solution
\[
s=\frac p\varepsilon,\qquad t=\frac{p-\varepsilon}2, \qquad p_1=p_2=\varepsilon
\]
valid for every \(0<\varepsilon<\min\{p-2,1\}\).

Let \(\{(u_k,v_k)\}\subset\mathcal H\) be a Cerami sequence for \(J\),
i.e.
\[
J(u_k,v_k)\to c>0,\qquad
\bigl(1+\|(u_k,v_k)\|_{\mathcal H}\bigr)\,\|J'(u_k,v_k)\|_{\mathcal H'}\to0 .
\]

The identity
\[
J(u,v)-\frac1p\langle J'(u,v),(u,v)\rangle
 =\frac{2}{p^2}\int_V\bigl(|u|^{p-2}v^2+|v|^{p-2}u^2\bigr)\,d\mu
\]
together with \(\langle J'(u_k,v_k),(u_k,v_k)\rangle\to0\) and
\(J(u_k,v_k)\to c\) gives
\begin{equation}\label{eq:mixbd-L3}
\int_V|u_k|^{p-2}v_k^2\,d\mu\le C_0,\qquad
\int_V|v_k|^{p-2}u_k^2\,d\mu\le C_0 .
\end{equation}

Since $\mu(x)\ge\mu_{\min}>0$, the bounds in (18) imply
\[
\begin{aligned}
\int_V |u_k|^p|v_k|^p\,d\mu
&\le \frac{1}{\mu_{\min}}
\left(\int_V |u_k|^{p-2}|v_k|^2\,d\mu\right)
\left(\int_V |v_k|^{p-2}|u_k|^2\,d\mu\right)\\
&\le \frac{C_0^2}{\mu_{\min}}.
\end{aligned}
\]

Set \(X_k:=\|u_k\|_{\mathcal H_a}^p+\|v_k\|_{\mathcal H_b}^p\).
Testing \(J'(u_k,v_k)\) with \((u_k,v_k)\) and using the mixed bounds
\eqref{eq:mixbd-L3} we obtain
\[
X_k \le \int_V|u_k|^{p-2}v_k^2\log v_k^2\,d\mu
      +\int_V|v_k|^{p-2}u_k^2\log u_k^2\,d\mu + C .
\]
The negative parts of the logarithmic integrals are uniformly bounded
as in the preceding section.  For the positive parts the calibration
inequality \(t^2\log t^2\le C_\varepsilon t^{2+\varepsilon}\)
(\(t\ge1\)) yields
\[
X_k \le C_\varepsilon\int_V\bigl(|v_k|^{p-2}|u_k|^{2+\varepsilon}
                               +|u_k|^{p-2}|v_k|^{2+\varepsilon}\bigr)d\mu + C .
\]
Applying the calibration system~\eqref{eq:cal-gen} to the integral
\(\int_V|v_k|^{p-2}|u_k|^{2+\varepsilon}\,d\mu\) and its symmetric
counterpart produces
\begin{eqnarray*}
   \int_V|v_k|^{p-2}|u_k|^{2+\varepsilon}\,d\mu &=& \int_V |u_k|^{p_1}\cdot |u_k|^{2+\varepsilon-p_1}\cdot |v_k|^{p_2}\cdot |v_k|^{p-p_2-2}\\
   &\le& \left(\int_V |u_k|^{p_1s} |v_k|^{p_2s} d\mu\right)^{\frac 1s}\cdot \left[\int_V \left(|u_k|^{2+\varepsilon-p_1}\cdot|v_k|^{p-p_2-2}\right)^{\frac{s}{s-1}}d\mu\right]^{\frac{s-1}{s}}\\
   &\le& C \left(\int_V |u_k|^{p} |v_k|^{p} d\mu\right)^{\frac 1s} \left(\int_{V} |u_k|^{(2+\varepsilon-p_1)\frac{st}{s-1}}d\mu+\int_V |v_k|^{\frac{(p-p_2-2)st}{(s-1)(t-1)}}d\mu \right)^{\frac{s-1}{s}}\\
   &\le& C X_k^{\frac{s-1}{s}},
\end{eqnarray*}
where we used the preceding mixed-integral bound and
Lemma~\ref{lem5}. Similarly, we obtain $\int_V|v_k|^{p-2}|u_k|^{2+\varepsilon}\,d\mu\le C X_k^{\frac{s-1}{s}}$, 
hence
\[
X_k \le C X_k^{\frac{s-1}{s}} + C .
\]
Since \(s>1\), the exponent \(\frac{s-1}{s}\) is strictly smaller than
\(1\); therefore \(\{X_k\}\) is bounded and \(\{(u_k,v_k)\}\) is a
bounded sequence in \(\mathcal H\).

Having established boundedness, by Lemma~\ref{lem5}, up to a subsequence,  
             \[\begin{cases}
                 (u_k,v_k)\rightharpoonup (u,v) \text{ weakly in } \mathcal{H},\\
                 (u_k,v_k)\rightarrow (u,v) \text{ in } L^{p_1}(V)\times L^{p_2}(V) \text{ for }p_1\ge \frac{p}{2}\text{ and } p_2\ge \frac{p}{2},\\
                 (u_k,v_k)\rightarrow (u,v) \text{ pointwisely in } V. 
              \end{cases}\]
            Since $(u_k,v_k)$ is a Cerami sequence, $\lim\limits_{k\to\infty}J'(u_k,v_k)\cdot (u_k-u,v_k-v)=0$, on the other hand, 
            $$J'(u,v)\cdot (u_k-u,v_k-v)=I_1+I_2+I_3,$$
              \begin{align*}
I_1
&= \int_{V} \Bigl[
    |\nabla u|^{p-2}\nabla u \cdot \nabla(u_k-u)
    + a(x)|u|^{p-2}u\,(u_k-u) \\
&\qquad + |\nabla v|^{p-2}\nabla v \cdot \nabla(v_k-v)
    + b(x)|v|^{p-2}v\,(v_k-v)
  \Bigr] \, d\mu,
\end{align*}
              \begin{eqnarray*}
                  I_2&=&\frac{2}{p}\int_V\left[|u|^{p-2}(v-v_k)v\log v^2+|u|^{p-2}v(v-v_k)\right]d\mu\\
                  &+&\frac{2}{p}\int_V\left[|v|^{p-2}(u-u_k)u\log u^2+|v|^{p-2}u(u-u_k)\right]d\mu,
              \end{eqnarray*}
              \begin{eqnarray*}
                  I_3=\frac{p-2}{p}\int_V\left[|u|^{p-4}u v^2\log v^2(u-u_k)+|v|^{p-4}v u^2\log u^2(v-v_k)\right]d\mu.
              \end{eqnarray*}
             Since \((u_k, v_k) \rightharpoonup (u, v)\), it is clear that 
\(\lim\limits_{k\to\infty} I_1 = \lim\limits_{k\to\infty} I_2 = \lim\limits_{k\to\infty} I_3 = 0\), 
that is \(\lim\limits_{k\to\infty} J'(u, v) \cdot (u_k - u, v_k - v) = 0\).
Note that
\allowdisplaybreaks
\begin{align*}
    & J'(u_k, v_k) \cdot (u_k - u, v_k - v) - J'(u, v) \cdot (u_k - u, v_k - v) \\
    &= \int_V \left(|\nabla u_k|^{p-2}\nabla u_k - |\nabla u|^{p-2}\nabla u\right) \cdot (\nabla u_k - \nabla u) \, d\mu \\
    &\quad + \int_V a(x)\left(|u_k|^{p-2}u_k - |u|^{p-2}u\right)(u_k - u) \, d\mu \\
    &\quad + \int_V \left(|\nabla v_k|^{p-2}\nabla v_k - |\nabla v|^{p-2}\nabla v\right) \cdot (\nabla v_k - \nabla v) \, d\mu \\
    &\quad + \int_V b(x)\left(|v_k|^{p-2}v_k - |v|^{p-2}v\right)(v_k - v) \, d\mu \\
    &\quad - \frac{2}{p}\int_V \left[u_k(\log u_k^2 + 1)\,|v_k|^{p-2} - u(\log u^2 + 1)\,|v|^{p-2}\right](u_k - u) \, d\mu \\
    &\quad - \frac{p-2}{p}\int_V \left(u_k^2\log u_k^2\,|v_k|^{p-4} v_k - u^2\log u^2\,|v|^{p-4} v\right)(v_k - v) \, d\mu \\
    &\quad - \frac{p-2}{p}\int_V
\left(|u_k|^{p-4}u_kv_k^2\log v_k^2
-|u|^{p-4}uv^2\log v^2\right)(u_k-u)\,d\mu\\
&\quad - \frac{2}{p}\int_V
\left[|u_k|^{p-2}v_k(\log v_k^2+1)
-|u|^{p-2}v(\log v^2+1)\right](v_k-v)\,d\mu\\
    &= o_k(1).
\end{align*}
              It is easy to check that $u_k(\log u_k^2 +1)|v_k|^{p-2}$ is uniformly bounded in $L^{\frac{p}{p-1}}(V)$ and $u(\log u^2 +1)|v|^{p-2}\in L^{\frac{p}{p-1}}(V)$, which means $\left[u_k(\log u_k^2 +1)|v_k|^{p-2}-u(\log u^2 +1)|v|^{p-2}\right]$ is uniformly bounded in $L^{\frac{p}{p-1}}(V)$. Then H\"{o}lder's inequality together with \(u_k\to u\) in \(L^q(V)\)\((q\ge \frac p2)\) implies that 
              \begin{eqnarray*}
               &&\lim\limits_{k\to\infty}\int_V\left[u_k\left(\log u_k^2+1\right)|v_k|^{p-2}-u\left(\log u^2+1\right)|v|^{p-2}\right](u_k-u)d\mu=0,
               \end{eqnarray*}
              since it is easy to check that $\left[u_k\left(\log u_k^2+1\right)|v_k|^{p-2}-u\left(\log u^2+1\right)|v|^{p-2}\right]\in L^{\frac{p}{p-1}}(V)$ due to \(\mathcal{H}\hookrightarrow L^q(V)\) for \(q\ge \frac{p}{2}\), and \(u_k\to u\) strongly in\(L^p(V)\), the details of proof of the above equality are omitted for brevity. and the other logarithmic terms can be similarly proved that tending to \(0\).
              
              From the fact that for any $\xi, \eta\in \mathbb{R}^n$, there exists $c>0$ such that $|\xi-\eta|^l\le c\left(|\xi|^{l-2}\xi-|\eta|^{l-2}\eta\right)\left(\xi-\eta\right)$ for all $l\ge 2$, there is
              $$0\le\frac{1}{c}\int_V |\nabla (u_k-u)|^p d\mu\le\int_V \left(|\nabla u_k|^{p-2}\cdot \nabla u_k-|\nabla u|^{p-2}\cdot \nabla u\right)(\nabla u_k-\nabla u)d\mu,$$
              $$0\le\frac{1}{c}\int_V a(x)|u_k-u|^pd\mu\le\int_V a(x)\left(|u_k|^{p-2}u_k-|u|^{p-2}u\right) (u_k-u)d\mu,$$
              $$0\le\frac{1}{c}\int_V |\nabla (v_k-v)|^p d\mu\le\int_V \left(|\nabla v_k|^{p-2}\cdot \nabla v_k-|\nabla v|^{p-2}\cdot \nabla v\right)(\nabla v_k-\nabla v)d\mu,$$
              $$0\le\frac{1}{c}\int_V b(x)|v_k-v|^pd\mu\le\int_V b(x)\left(|v_k|^{p-2}v_k-|v|^{p-2}v\right) (v_k-v)d\mu.$$
              On the other hand, $$\lim_k\left[\int_V \bigl(|\nabla u_k|^{p-2}\nabla u_k-|\nabla u|^{p-2}u\bigr)(\nabla u_k-\nabla u)\,d\mu
              +\int_V a(x)\bigl(|u_k|^{p-2}u_k-|u|^{p-2}u\bigr)(u_k-u)\,d\mu\right]=0,$$
              $$\lim_k\left[\int_V \bigl(|\nabla v_k|^{p-2}\nabla v_k-|\nabla v|^{p-2}v\bigr)(\nabla v_k-\nabla v)\,d\mu
              +\int_V b(x)\bigl(|v_k|^{p-2}v_k-|v|^{p-2}v\bigr)(v_k-v)\,d\mu\right]=0.$$
              To sum up, we conclude that $$\lim\limits_{k\to\infty}\int_V \left(|\nabla (u_k-u)|^p+a(x)|u_k-u|^p+\nabla (v_k-v)|^p+b(x)|v_k-v|^p\right)d\mu=0.$$ Hence we deduce $$\lim\limits_{k\to\infty}\|(u_k-u,v_k-v)\|_{\mathcal{H}}=0.$$ The proof is completed.
\end{proof}

 \begin{theorem}{\cite{BBF83}}\label{thm:mp}
Let $(X,\|\cdot\|)$ be a Banach space and $J\in C^{1}(X,\mathbb{R})$ a functional satisfying the Cerami condition. If there exist $e\in X$ and $r>0$ with $\|e\|>r$ such that
\[
a=\inf_{\|u\|=r}J(u)>J(0)\ge J(e),
\]
then $b$ is a critical value of $J$, where
\[
b=\inf_{\gamma\in\mathcal{P}}\max_{t\in[0,1]}J(\gamma(t)),\qquad
\mathcal{P}=\{\gamma\in C([0,1],X):\gamma(0)=0,\ \gamma(1)=e\}.
\]
\end{theorem}

We now verify that $J$ possesses the mountain pass geometry.\\
\medskip
\noindent\textbf{Completion of the proof of Theorem~\ref{thm3}.}
\begin{lemma}
There exist constants $\rho>0$ and $\delta>0$ such that
\begin{equation}
J(u,v)\ge \delta \quad \text{for all } (u,v)\in\mathcal H \text{ with } \|(u,v)\|_{\mathcal H}=\rho.
\end{equation}
Moreover, there exists a nontrivial pair $(\phi,\psi)\in\mathcal H$ such that
\begin{equation}
J(t(\phi,\psi)) \to -\infty \quad \text{as } t\to+\infty.
\end{equation}
\end{lemma}

\begin{proof}
For (i), the same estimates as in Lemma~\ref{lem4.3} yield, for some constant $C>0$,
\[
\int_V |u|^{p-2}v^2\log v^2\,d\mu
\le C\bigl(\|u\|_{\mathcal H_a}^{p+\varepsilon}+\|v\|_{\mathcal H_b}^{p+\varepsilon}\bigr),
\]
and similarly for the symmetric term. Hence
\[
J(u,v)\ge \frac{1}{p}\bigl(\|u\|_{\mathcal H_a}^p+\|v\|_{\mathcal H_b}^p\bigr)
      - C\|(u,v)\|_{\mathcal H}^{p+\varepsilon}.
\]
Using $\|u\|_{\mathcal H_a}^p+\|v\|_{\mathcal H_b}^p \ge 2^{1-p}\|(u,v)\|_{\mathcal H}^p$, we obtain
\[
J(u,v)\ge \frac{1}{p2^{p-1}}\|(u,v)\|_{\mathcal H}^p
          - C\|(u,v)\|_{\mathcal H}^{p+\varepsilon}.
\]
Choosing $\rho>0$ sufficiently small makes the right-hand side positive; set $\delta$ to be its minimum on the sphere.

For (ii), fix $(\phi,\psi)\in\mathcal H$ with
\[
A:=\|\phi\|_{\mathcal H_a}^p+\|\psi\|_{\mathcal H_b}^p
  -\int_V|\phi|^{p-2}\psi^2\log\psi^2\,d\mu
  -\int_V|\psi|^{p-2}\phi^2\log\phi^2\,d\mu>0,
\]
which is possible, e.g., by taking $\phi,\psi$ supported on a single vertex and with $\phi,\psi\not\equiv0$. Then
\[
J(t(\phi,\psi))
= \frac{t^p}{p}A
  - \frac{t^p\log t^2}{p}
   \left(\int_V|\phi|^{p-2}\psi^2\,d\mu+\int_V|\psi|^{p-2}\phi^2\,d\mu\right).
\]
Since the last bracket is positive, the term $t^p\log t^2$ dominates, so $J(t(\phi,\psi))\to -\infty$ as $t\to+\infty$. This completes the proof.
\end{proof}
By the variable-endpoint version of the mountain-pass
argument in Theorem~\ref{thm:mp}, a Cerami sequence
exists at the level
\[
c=\inf_{\gamma\in\Gamma}\max_{t\in[0,1]}J(\gamma(t)),
\]
where
\[
\Gamma=\{\gamma\in C([0,1],\mathcal H):
\gamma(0)=0,\ J(\gamma(1))<0\}.
\]
Lemma~\ref{lem:cerami} ensures that $c$ is attained
as a critical value, i.e., there exists
$(u_0,v_0)\in\mathcal H$ with $J(u_0,v_0)=c$
and $J'(u_0,v_0)=0$.

To see that $(u_0,v_0)$ is also a ground state, we recall that the Nehari manifold
\[
\mathcal N:=\{(u,v)\in\mathcal H\setminus\{(0,0)\}:J'(u,v)\cdot(u,v)=0\}
\]
is a $C^1$ manifold. Indeed, for $(u,v)\in\mathcal N$,
\[
F'(u,v)\cdot(u,v)=-2\int_V\bigl(|u|^{p-2}v^2+|v|^{p-2}u^2\bigr)\,d\mu<0,
\]
where $F(u,v):=J'(u,v)\cdot(u,v)$, so the Implicit Function Theorem applies. Moreover, every nonzero critical point of $J$ lies in $\mathcal N$, and every critical point of $J|_{\mathcal N}$ is a critical point of $J$ by the usual Lagrange multiplier argument.

Define
\[
\mathcal K_*:=\{(u,v)\in\mathcal H\setminus\{(0,0)\}:J'(u,v)=0\},
\qquad
m_*:=\inf_{\mathcal K_*}J.
\]
The mountain-pass critical point $(u_0,v_0)$ satisfies $J(u_0,v_0)=c>0$, so $\mathcal K_*\neq\varnothing$ and $m_*\le c$. Since $\mathcal K_*\subset\mathcal N$, it is easy to obtain $m_*\ge d:=\inf_{\mathcal N}J>0$.
Let ${(u_k,v_k)}\subset\mathcal K_*$ satisfy $J(u_k,v_k)\to m_*$. As $J'(u_k,v_k)=0$, this is a Cerami sequence. By Lemma~\ref{lem:cerami}, a subsequence converges strongly in $\mathcal H$ to some $(u,v)$. Since $J\in C^1$,
\[
J'(u,v)=0,\qquad J(u,v)=m_*>0.
\]
Thus $(u,v)\in\mathcal K_*$ and is a ground state.
It remains to show $c=m_*$. We already have $m_*\le c$. Conversely, for any $(u,v)\in\mathcal K_*\subset\mathcal N$,
choose $t^*>1$ with $J(t^*(u,v))<0$ and set
$\gamma_0(t):=tt^*(u,v)$ for $t\in[0,1]$.
This is an admissible mountain-pass path, and
Lemma~\ref{lem:projection1} yields
\[
c\le\max_{t\in[0,1]}J(\gamma_0(t))
=\sup_{t\ge0}J(t(u,v))=J(u,v).
\]
Taking the infimum over $\mathcal K_*$ gives $c\le m_*$.
Hence $c=m_*$, completing the proof.

Having established existence, we next investigate the response of the ground-state energy to potential perturbations and the structure of the ground-state set.

\section{Structural Stability: Rigidity of the Ground-State Energy, the Gradient Formula and Hessian Factorisation}\label{sec:stability}

In this section we establish the rigidity properties of ground states for the logarithmically coupled system \eqref{eq:main}, under the standing assumption that the ground state under consideration is non-degenerate. The analysis reveals an exact algebraic identity linking the Sobolev norm and the logarithmic coupling energy, from which the full response theory---Lipschitz dependence, Fr\'echet differentiability, and Hessian factorisation---is derived.

For a pair $(u,v)\in\mathcal{H}$ we introduce the principal part of the energy and the logarithmic coupling energy,
\[
S(u,v)=\|u\|_{\mathcal{H}_a}^p+\|v\|_{\mathcal{H}_b}^p,\qquad
L(u,v)=\int_V\!\bigl(|u|^{p-2}v^2\log v^2+|v|^{p-2}u^2\log u^2\bigr)d\mu, 
\]
\[
M(u,v)=\int_V\bigl(|u|^{p-2}v^2+|v|^{p-2}u^2\bigr)d\mu .
\]
Both quantities are finite on the Nehari manifold $\mathcal{N}_{a,b}$. The ground-state energy is equal to $d(a,b)=\inf\limits_{(u,v)\in\mathcal{N}_{a,b}}J_{a,b}(u,v)$.
\begin{definition}\label{def:5.1}
Fix the potentials $(a,b)$ and let
\[
\mathcal{G}_{a,b}=\{(u,v)\in\mathcal{N}_{a,b}: J_{a,b}(u,v)=d(a,b)\}
\]
be the set of all ground states. And we define the tangent space 
$$T=T_{(u_0,v_0)}\mathcal{N}_{a_0,b_0}=\{(\phi,\psi)\in\mathcal{H}:\langle G_{a_0,b_0}'(u_0,v_0),(\phi,\psi)\rangle=0\},$$
where \(G_{a_0,b_0}(u_0,v_0)=J_{a_0,b_0}'(u_0,v_0)\cdot (u_0,v_0)\). Note that \(\langle G'(u_0,v_0),(u_0,v_0)\rangle=-2M\neq 0\), the radial direction \((u_0,v_0)\) is transversal to \(T\).

A ground state $(u_0,v_0)\in\mathcal{G}_{a,b}$ is called \emph{non-degenerate} if the Hessian $J_{a,b}''(u_0,v_0)$ restricted to the tangent space $T_{(u_0,v_0)}\mathcal{N}_{a,b}$ is an isomorphism, hence uniformly positive definite, due to the fact that the ground state is a minimum on $\mathcal{N}_{a,b}$.
\end{definition}
\begin{lemma}\label{lem6.1}
Let $(u_1,v_1)$ and $(u_2,v_2)$ be two ground states of $J_{a,b}$ for the \emph{same} potentials $a,b$,  
then
\begin{equation}\label{eq:rigidity}
S(u_1,v_1)-S(u_2,v_2)=L(u_1,v_1)-L(u_2,v_2).
\end{equation}
Equivalently, the quantity $S(u,v)-L(u,v)$ is a constant on $\mathcal{G}_{a,b}$ and equals to $p\,d(a,b)$. 
In addition, assume that \((u,v)\in\mathcal{H}\) is any nontrivial critical point of \(J\) on the Nehari manifold \(\mathcal{N}_{a,b}\), then the mixed integral $M=M(u,v)$ satisfies the two identities below:\\
\[(i)\quad J(u,v)=\frac{2}{p^2}M,\]
\[(ii)\quad \langle G'(u,v),(u,v)\rangle=-2M,\]
where $G(u,v)=\langle J'(u,v),(u,v)\rangle$ is the Nehari functional.
\end{lemma}
\begin{proof}
    By the definition of \(J_{a,b}\), there holds
    $$J_{a,b}(u,v)=\frac1p \left(S(u,v)-L(u,v)\right).$$
    The result immediately follows from the fact that \(J_{a,b}(u_1,v_1)=J_{a,b}(u_2,v_2)=d(a,b)\). Moreover, 
    \begin{equation}\label{eq:rigidity1}
        S(u,v)-L(u,v)=pd(a,b)
    \end{equation}
    (i) and (ii) are immediate consequences of direct computation; the details are omitted for brevity.
\end{proof}
Lemma~\ref{lem6.1} reveals a rigid algebraic structure of the ground-state set: it is precisely the intersection of the Nehari manifold \(\mathcal N_{a,b}\) with the energy level \(\{S-L=p\,d(a,b)\}\). The identities (i) and (ii) further show that the mixed integral \(M\) simultaneously determines the energy \(J(u,v)\) and the transversality of the Nehari constraint, via \(\langle G'(u,v),(u,v)\rangle=-2M\). Thus the ground-state set is characterised by the Nehari
constraint together with the least-energy condition. These identities provide the algebraic basis for the perturbation and rigidity analysis developed below. The following lemma exploits this structure to track how ground states respond to perturbations of the potentials.
\begin{lemma}\label{lem6.2}
Let $(u_0,v_0)$ be a ground state of $J_{a_0,b_0}$,
there exist $\varepsilon_0>0$ and $C>0$ such that for all $(a_1,b_1)$ with $\|\delta a_1\|_\infty,\|\delta b_1\|_\infty\le\varepsilon_0$, \((a,b)=(a_0+\delta a_1,b_0+\delta b_1)\), the projection parameter $t=t_{a,b}(u_0,v_0)$ satisfies
\[
|t-1|\le C\bigl(\|\delta a_1\|_\infty+\|\delta b_1\|_\infty\bigr).
\]
\end{lemma}

\begin{proof}
The condition $t(u_0,v_0)\in\mathcal{N}_{a,b}$ is equivalent to
$\langle J_{a,b}'(t(u_0,v_0)),t(u_0,v_0)\rangle=0$.
Expanding the Nehari identity~\eqref{eq:Nehari-id} and using
$\log(t^2w^2)=\log w^2+2\log t$ gives, for the perturbed
potentials $(a,b)$,
\begin{equation}\label{eq:t-id}
t^p\bigl(S_{a,b}(u_0,v_0)-L(u_0,v_0)\bigr)
=2t^p\log t\,M+\frac{2t^p}{p}M,
\end{equation}
where $S_{a,b}$ denotes the principal part and $M=M(u_0,v_0)>0$.
At the reference potentials $(a_0,b_0)$ we have $t=1$, so that
\eqref{eq:t-id} reduces to
\begin{equation}\label{eq:ref-id}
S_{a_0,b_0}(u_0,v_0)-L(u_0,v_0)=\frac{2}{p}M .
\end{equation}
Since $(u_0,v_0)$ is a non-trivial ground state, its energy
$d(a_0,b_0)$ is strictly positive; the energy identity
$J=\frac{2}{p^2}M$ on $\mathcal{N}_{a_0,b_0}$ then implies
$M>0$.

Subtracting $t^p$ times~\eqref{eq:ref-id} from~\eqref{eq:t-id}
and using the linear decomposition
$S_{a,b}(u_0,v_0)=S_{a_0,b_0}(u_0,v_0)
+\int_V(\delta a_1|u_0|^p+\delta b_1|v_0|^p)\,d\mu$,
the left-hand side simplifies to $t^p\Delta$, where
$\Delta:=\int_V(\delta a_1|u_0|^p+\delta b_1|v_0|^p)\,d\mu$.
We thus obtain
\[
 \Delta = 2M\log t,
\]
and consequently $\log t = \Delta/(2M)$.
The desired estimate $|t-1|\le C(\|\delta a_1\|_\infty
+\|\delta b_1\|_\infty)$ now follows from the elementary bound
$|e^s-1|\le2|s|$, valid for all sufficiently small $s$.
\end{proof}

\subsection*{Proof of Theorem~\ref{thm4}}

\begin{proof}
Let \((u_0,v_0)\) be a ground state of \(J_{a_0,b_0}\) and \((u,v)\) a 
ground state of \(J_{a,b}\).  Set \(t=t_{a,b}(u_0,v_0)\).  By minimality,
\[
d(a,b)\le J_{a,b}\bigl(t(u_0,v_0)\bigr).
\]
Computing \(J_{a,b}(t(u_0,v_0))\) as in Lemma~\ref{lem6.2} and using 
the identity \(\Delta=2M\log t\) obtained there, we find
\[
J_{a,b}\bigl(t(u_0,v_0)\bigr)=t^p d(a_0,b_0).
\]
Lemma~\ref{lem6.2} supplies the estimate 
\(|t-1|\le C\bigl(\|\delta a_1\|_\infty+\|\delta b_1\|_\infty\bigr)\) 
with \(C\) depending only on \((u_0,v_0)\).  Since \(s\mapsto s^p\) is 
Lipschitz near \(s=1\), there exists \(C'>0\) such that 
\(|t^p-1|\le C'|t-1|\) whenever 
\(\|\delta a_1\|_\infty,\|\delta b_1\|_\infty\) are sufficiently small.  
Consequently,
\[
d(a,b)\le d(a_0,b_0)+C'C\,d(a_0,b_0)\bigl(\|\delta a_1\|_\infty+\|\delta b_1\|_\infty\bigr).
\]
Setting \(l_1:=2C'C\,d(a_0,b_0)\) and using 
\(\|\delta a_1\|_\infty+\|\delta b_1\|_\infty\le 2\max\{\|\delta a_1\|_\infty,\|\delta b_1\|_\infty\}\), 
we obtain the upper bound
\[
d(a,b)\le d(a_0,b_0)+l_1\max\{\|\delta a_1\|_\infty,\|\delta b_1\|_\infty\}.
\]
In particular, shrinking \(\varepsilon_0\) if necessary guarantees 
\(d(a,b)\le 2d(a_0,b_0)\).  For potentials sufficiently close to $(a_0,b_0)$,
the positive lower bounds and embedding constants
can be chosen uniformly. Since
$d(a,b)\le 2d(a_0,b_0)$, the boundedness argument
in the proof of Lemma~3.3 gives a uniform bound
on the corresponding ground states in $\mathcal H$,
and hence in $L^\infty(V)\times L^\infty(V)$.
Consequently,
\[
L(u,v)\le L^+(u,v)\le C M(u,v).
\]
The Nehari identity and the uniform positive lower
bounds on the potentials therefore give
\[
\|u\|_{H_a}^p+\|v\|_{H_b}^p
=L(u,v)+\frac2pM(u,v)\le C M(u,v),
\]
and thus
\[
\frac{\int_V(|u|^p+|v|^p)\,d\mu}{M(u,v)}\le C.
\]
It follows that the reverse scaling factor satisfies
\[
|t-1|\le C\bigl(
\|\delta a_1\|_\infty+\|\delta b_1\|_\infty
\bigr)
\]
uniformly in this neighbourhood.
Reversing the roles of $(a_0,b_0)$ and $(a,b)$
and using $d(a,b)\le 2d(a_0,b_0)$, we obtain
\[
d(a_0,b_0)\le d(a,b)
+l_2\max\{\|\delta a_1\|_\infty,\|\delta b_1\|_\infty\},
\]
where $l_2>0$ is independent of $(a,b)$
in this neighbourhood.  Taking \(l:=\max\{l_1,l_2\}\) gives
\[
|d(a,b)-d(a_0,b_0)|\le l\max\bigl\{\|\delta a_1\|_\infty,\|\delta b_1\|_\infty\bigr\}.
\]

The upper bound
\[
d(a,b)\le d(a_0,b_0)+\frac1p\Delta
+\frac{1}{4M(u_0,v_0)}\Delta^2+O(|\Delta|^3),
\]
with $\Delta=\int_V(\delta a_1|u_0|^p+\delta b_1|v_0|^p)\,d\mu$,
is obtained, for $|\Delta|$ sufficiently small, exactly as above,
by expanding
\[
t^p=1+p\log t+\frac{p^2}{2}(\log t)^2
+O(|\log t|^3)
\]
and using $\log t=\Delta/(2M(u_0,v_0))$ together with the energy
identity $d(a_0,b_0)=\frac{2}{p^2}M(u_0,v_0)$. Indeed,
\[
\begin{aligned}
d(a,b)
&\le J_{a,b}(tu_0,tv_0)
=\frac{2}{p^2}M(u_0,v_0)t^p\\
&=d(a_0,b_0)+\frac1p\Delta
+\frac{1}{4M(u_0,v_0)}\Delta^2+O(|\Delta|^3).
\end{aligned}
\]

The lower bound follows from the same argument with the roles
reversed, i.e., projecting the perturbed ground state $(u,v)$
onto $\mathcal{N}_{a_0,b_0}$ yields
\[
d(a_0,b_0)\le d(a,b)-\frac1p\Delta'
+\frac{1}{4M(u,v)}(\Delta')^2+O(|\Delta'|^3),
\]
where
$\Delta'=\int_V(\delta a_1|u|^p+\delta b_1|v|^p)\,d\mu$.
Rearranging gives
\[
d(a,b)\ge d(a_0,b_0)+\frac1p\Delta'
-\frac{1}{4M(u,v)}(\Delta')^2-O(|\Delta'|^3).
\]
Here the remainder constant can be chosen uniformly for potentials
sufficiently close to $(a_0,b_0)$, since
$M(u,v)=\frac{p^2}{2}d(a,b)$ remains bounded away from zero.
Since $(a_0,b_0)$ is arbitrary, these local estimates hold around
every admissible reference pair, completing the proof.
\end{proof}

Throughout the rest of this section, unless otherwise stated, the graph $V$ is finite, and the potentials $(a,b),(a_0,b_0)$ satisfy $(A_1)$ and $(A_2')$. Since finite graphs are a special case of locally finite graphs, all previous existence and compactness results remain valid. We also work within the positive cone $\widetilde{\mathcal H}=\{(u,v)\in\mathcal H:u\ge0,\ v\ge0\}$. Under$(A_1)\text{--}(A_2')$, the functional $J$ is well-defined and $C^1$; moreover, any constrained minimizer on the Nehari manifold is a free critical point (by the Lagrange multiplier rule or Implicit Function Theorem). By Proposition~\ref{Pro1}, every nontrivial nonnegative weak solution is strictly positive on $V$. Since $V$ is finite, the ground state therefore attains a positive minimum, and consequently $J$ is $C^2$ in a neighbourhood of the ground state.

\begin{lemma}\label{Lem:Frechet differential}
Let \((u_0, v_0)\) be a non-degenerate ground state of \(J_{a_0,b_0}\). Then there exist a neighbourhood \(\mathcal{U}\) of \((a_0, b_0)\) in the parameter space and a unique \(C^1\) branch \((u_{a,b}, v_{a,b})\) of critical points of \(J_{a,b}\) such that \((u_{a_0,b_0}, v_{a_0,b_0}) = (u_0, v_0)\). Define the branch energy \(\mathcal{E}(a,b) := J_{a,b}(u_{a,b}, v_{a,b})\). Then \(\mathcal{E}\) is Fréchet differentiable at \((a_0, b_0)\) and
\[
\nabla_a E(a_0,b_0)
=\frac1p\bigl(\mu(x)|u_0(x)|^p\bigr)_{x\in V},
\qquad
\nabla_b E(a_0,b_0)
=\frac1p\bigl(\mu(x)|v_0(x)|^p\bigr)_{x\in V}.
\]
\end{lemma}
\begin{proof}
   In the Implicit Function Theorem framework, we fix $b=b_0$ and define the extended system
\[
F : L^\infty(V) \times \mathcal{H} \times \mathbb{R} \to \mathcal{H}^{*} \times \mathbb{R},
\qquad
F(a,u,v,\lambda) := \begin{pmatrix}
J_{a,b_0}'(u,v) + \lambda G'(u,v) \\
G(a,u,v)
\end{pmatrix},
\]
\[
G(a,u,v) := \langle J_{a,b_0}'(u,v), (u,v) \rangle ,
\]
where \(G'(u,v)\in\mathcal{H}^{*}\) is Fr\'echet derivative of \(G(a,u,v)\) with respect to \((u,v)\).
At the reference ground state $(u_0,v_0)$ for $(a_0,b_0)$ we have
$F(a_0,u_0,v_0,0)=(0,0)$ because $(u_0,v_0)$ is a free critical point
of $J_{a_0,b_0}$ and therefore lies on the Nehari manifold
$\mathcal{N}_{a_0,b_0}$.

The Fr\'echet derivative of $F$ with respect to $(u,v,\lambda)$ at
$(a_0,u_0,v_0,0)$ acts on a direction $(\phi,\psi,\mu)\in\mathcal{H}\times\mathbb{R}$ as
\[
D_{(u,v,\lambda)}F(a_0,u_0,v_0,0)\cdot(\phi,\psi,\mu)
= \begin{pmatrix}
J_{a_0,b_0}''(u_0,v_0)\cdot(\phi,\psi) + \mu\,G'(u_0,v_0) \\[4pt]
\langle G'(u_0,v_0), (\phi,\psi)\rangle
\end{pmatrix},
\]
where $J''(u_0,v_0)$ is the full Hessian of the energy and
$G'(u_0,v_0)$ is the linearisation of the Nehari constraint.

To verify the invertibility, we solve the linearised system
\[
D_{(u,v,\lambda)}F(a_0,u_0,v_0,0)\cdot(\phi,\psi,\mu)=(\Phi,s)
\]
for an arbitrary given pair \((\Phi,s)\in\mathcal{H}^*\times\mathbb{R}\).
In components, this reads
\[
\begin{cases}
J_{a_0,b_0}''(u_0,v_0)\cdot(\phi,\psi) + \mu G'(u_0,v_0) = \Phi,\\[4pt]
\langle G'(u_0,v_0),(\phi,\psi)\rangle = s .
\end{cases}
\]
Decompose \((\phi,\psi)=\tau(u_0,v_0)+(\phi_T,\psi_T)\) with \((\phi_T,\psi_T)\in T\) and \(\tau\in\mathbb{R}\). The second equation gives \(\tau\langle G'(u_0,v_0),(u_0,v_0)\rangle=s\), which determines \(\tau\) uniquely.

Insert the decomposition into the first equation and project onto \(T^*\) and \(T_1^*\) separately. By the orthogonality relation, \(J''(u_0,v_0)\cdot(\phi_T,\psi_T)\) vanishes on \(T_1\), hence it belongs to \(T^*\). Decompose \(\Phi=\Phi_T+\Phi_1\) with \(\Phi_T\in T^*\), \(\Phi_1\in T_1^*\). The \(T^*\)-component reduces to
\[
J''(u_0,v_0)\cdot(\phi_T,\psi_T) = \Phi_T .
\]
By the non‑degeneracy hypothesis from Definition~\ref{def:5.1}, the restricted Hessian \(J''(u_0,v_0)|_T:T\to T^*\) is an isomorphism, hence \((\phi_T,\psi_T)\) is uniquely determined. 

Projecting the remaining terms onto \(T_1^*\) yields an equation that determines \(\mu\) uniquely, because the term involving \(\mu\) on \(T_1^*\) is \(G'(u_0,v_0)|_{T_1^*}\neq0\), while all other terms are already known. Thus \((\phi,\psi,\mu)\) is uniquely determined by \((\Phi,s)\), which proves that the Jacobian is an isomorphism. Then the Implicit Function Theorem
yields a unique local $C^1$ branch $a\mapsto(u_a,v_a)$ of critical points for
potentials near $a_0$. The same argument applies when \(a=a_0\) is fixed and \(b\) varies. More generally, regarding \((a,b)\) as a single parameter in the product space, the implicit function theorem yields a unique local \(C^1\) branch \[ (a,b)\mapsto (u_{a,b},v_{a,b}) \] near \((a_0,b_0)\). In particular, $(u,v)\to(u_0,v_0)$ strongly in $\mathcal{H}$ as $(a,b)\to(a_0,b_0)$ in \(L^{\infty}\).
Here, \((u,v)\) denotes the critical point on the local branch corresponding to the potentials \((a,b)\).

 As \((a,b)\to (a_0,b_0)\), \((u,v)\to (u_0,v_0)\) strongly in \(L^p\), which follows from the non-degeneracy hypothesis and the Implicit 
Function Theorem established above. Then
\[
\begin{aligned}
|\triangle'-\triangle|
&\le \max\{\|\delta a_1\|_\infty,\|\delta b_1\|_\infty\}
   \bigg(\int_V \big||u|^p-|u_0|^p\big|\,d\mu
        +\int_V \big||v|^p-|v_0|^p\big|\,d\mu\bigg) \\[4pt]
&\le C\max\{\|\delta a_1\|_\infty,\|\delta b_1\|_\infty\}\,\varepsilon .
\end{aligned}
\]
for \((a,b)\) sufficiently close to \((a_0,b_0)\) in the sense of \(L^{\infty}\). Since the local branch is $C^1$, we have
$\|(u,v)-(u_0,v_0)\|_{\mathcal H}=O(\varepsilon)$.
Taylor expansion of $J_{a,b}$ at its critical point $(u,v)$,
together with
$J_{a,b}(u_0,v_0)=\mathcal E(a_0,b_0)+\Delta/p$,
therefore gives
\[
|\mathcal{E} (a,b) - \mathcal{E}(a_0,b_0) - \frac{1}{p} \triangle| \le C \varepsilon^2,
\]
where $\triangle = \int_V (\delta a_1 |u_0|^p + \delta b_1 |v_0|^p) d\mu$ is linear in
$(\delta a_1,\delta b_1)$, and here $\|\delta a_1\|_\infty,\|\delta b_1\|_\infty \le \varepsilon$, \(\varepsilon\) is sufficiently small,
this estimate is precisely the definition of Fr\'echet differentiability of $d$
at $(a_0,b_0)$, with derivative
\[
\nabla_a E(a_0,b_0)
=\frac1p\bigl(\mu(x)|u_0(x)|^p\bigr)_{x\in V},
\qquad
\nabla_b E(a_0,b_0)
=\frac1p\bigl(\mu(x)|v_0(x)|^p\bigr)_{x\in V}.
\]
as we want to prove.
\end{proof}
\begin{remark}
The gradient formula in Lemma~\ref{Lem:Frechet differential} is the differential counterpart of the rigidity identity \eqref{eq:rigidity}.  
It shows that the logarithmic nonlinearity does not contribute to the first-order variation of the ground-state energy with respect to the potentials.  
\end{remark}

\begin{lemma}\label{lem6.3}
$\mathcal{L}_{a,b}$ is compact in $\mathcal{H}$.
\end{lemma}
\begin{proof}
   Since any sequence in \(\mathcal{L}_{a,b}\) has a constant energy \(d(a,b)\), similar as Lemma~\ref{lem4.4}, we obtain a uniform \(\mathcal{H}\)-bound on the sequence. Since $J'(u_n,v_n)=0$, the Cerami condition (Lemma~\ref{lem:cerami}) applies, yielding a strongly convergent subsequence whose limit belongs to $\mathcal{G}_{a,b}$ by continuity of $J$ and $J'$.
\end{proof}
For convenience, in the rest of this section, we write \(\|\cdot\|\) for \(\|\cdot\|_{\mathcal H}\).
\begin{theorem}\label{thm6.3}
    Let \(\mathcal{L}_{a,b}^{*}\) be the set of all non-degenerate ground states of \(J_{a,b}\), and
    \[\inf\limits_{(u,v)\in\mathcal L_{a,b}^{*}}
\lambda_{\min}\!\left(
J''_{a,b}(u,v)|_{T_{(u,v)}\mathcal N_{a,b}}
\right)>0,\] then \(\mathcal{L}_{a,b}^{*}\) is a finite set.
\end{theorem}
\begin{proof}
    Suppose for contradiction that $\mathcal{L}_{a,b}^{*}$ is infinite.  
By Lemma~\ref{lem6.3} and the uniform spectral gap assumption, there exists an accumulation point $(u_0,v_0)\in\mathcal{L}_{a,b}^{*}$ and a sequence $\{(u_n,v_n)\}\subset\mathcal{L}_{a,b}^{*}$ with $(u_n,v_n)\neq(u_0,v_0)$ converging to $(u_0,v_0)$.  
Set $(\phi_n,\psi_n)=(u_n,v_n)-(u_0,v_0)\to0$ and consider the unit vectors $(\tilde\phi_n,\tilde\psi_n)=(\phi_n,\psi_n)/\|(\phi_n,\psi_n)\|$.  
Passing to a subsequence, $(\tilde\phi_n,\tilde\psi_n)\rightharpoonup(\phi_*,\psi_*)$ weakly in $\mathcal{H}$ with $\|(\phi_*,\psi_*)\|\neq 0$ due to \(\|(\tilde{\phi}_n,\tilde{\psi}_n)\|=1\).
Both $(u_n,v_n)$ and $(u_0,v_0)$ lie on $\mathcal{N}_{a,b}$, hence $G(u_n,v_n)=G(u_0,v_0)=0$ for the Nehari functional $G(u,v)=\langle J_{a,b}'(u,v),(u,v)\rangle$.  
Taylor expansion of $G$ around $(u_0,v_0)$ gives
\begin{equation}\label{eq:G-expansion}
0 = G(u_n,v_n)-G(u_0,v_0)
= \langle G'(u_0,v_0),(\phi_n,\psi_n)\rangle
  + \frac12\langle G''(u_0,v_0)(\phi_n,\psi_n),(\phi_n,\psi_n)\rangle
  + o\bigl(\|(\phi_n,\psi_n)\|^2\bigr).
\end{equation}
Dividing by $\|(\phi_n,\psi_n)\|$ and passing to the limit, the Hessian term vanishes and we obtain
\[
\langle G'(u_0,v_0),(\phi_*,\psi_*)\rangle = 0,
\]
so $(\phi_*,\psi_*)$ belongs to the tangent space $T_{(u_0,v_0)}\mathcal{N}_{a,b}$. Because each \((u_n,v_n)\) is a ground state, \(J(u_n,v_n)=J(u_0,v_0)\). Expanding \(J\) around \((u_0,v_0)\) and using \(J'(u_0,v_0)=0\),
\[
0 = \frac12\langle J''(u_0,v_0)(\phi_n,\psi_n),(\phi_n,\psi_n)\rangle + o(\|(\phi_n,\psi_n)\|^2).
\]
Dividing by \(\|(\phi_n,\psi_n)\|^2\) and passing to the weak limit again, the lower semi‑continuity of the quadratic form associated with \(J''\) yields
\begin{equation}\label{eq:J-weak-lim}
\langle J''(u_0,v_0)(\phi_*,\psi_*),(\phi_*,\psi_*)\rangle \le 0.
\end{equation}
 On the other hand, Non-degeneracy means that 
\[
    \langle J_{a,b}''(u_0,v_0)(\xi,\eta),(\xi,\eta)\rangle\ge \lambda\|(\xi,\eta)\|_{\mathcal{H}}^2, \qquad \forall (\xi,\eta)\in T_{u_0,v_0}\mathcal{N}_{a,b}.
\]
Applying this to \((\phi_{*},\psi_{*})\) yields \(\|(\phi_{*},\psi_{*})\|=0\), contradicting \(\|(\phi_{*},\psi_{*})\|\neq 0\). Therefore $\mathcal{L}_{a,b}^{*}$ consists of isolated points, together with the fact that the compact set of isolated points is necessarily finite, we complete the proof.
\end{proof}
\begin{remark}
    The Hessian term in (\ref{eq:G-expansion}) shows that the sequence $\{(\phi_n,\psi_n)\}$ satisfies that \(\|(\phi_n^{\perp},\psi_n^{\perp})\|=O(\|(\phi_n,\psi_n)\|^2)\), where $(\phi_n^{\perp},\psi_n^{\perp})$ belongs to the direct sum complement of the tangent space. In particular, 
    \[
      \langle J''(u_0,v_0)(\phi_n,\psi_n),(\phi_n,\psi_n)\rangle
      = \langle J''(u_0,v_0)(\phi_n^T,\psi_n^T),(\phi_n^T,\psi_n^T)\rangle
      + o\bigl(\|(\phi_n,\psi_n)\|^2\bigr).
    \]
\end{remark}
\begin{proof}
    Indeed, we decompose \((\phi_n,\psi_n)\) as
    $$(\phi_n,\psi_n)=(\phi_n^{\perp},\psi_n^{\perp})+(\phi_n^{T},\psi_n^{T})=\tau_n (u_0,v_0)+(\phi_n^{T},\psi_n^{T}).$$
    Since $$\langle G'(u_0,v_0), (\phi_n^{T},\psi_n^{T})\rangle=0,$$
    there holds 
    $$\langle G'(u_0,v_0), (\phi_n,\psi_n)\rangle=\tau_n\langle G'(u_0,v_0), (u_0,v_0)\rangle=\tau_n c_0,$$
    where \(c_0=\langle G'(u_0,v_0), (u_0,v_0)\rangle\neq 0\). Then
    \[
       \tau_n c_0+\frac12 \langle G''(u_0,v_0)(\phi_n,\psi_n),(\phi_n,\psi_n)\rangle+ o(\|(\phi_n,\psi_n)\|^2)=0.
    \]
    The second and third terms of left side of the above equality are of order \(O\|(\phi_n,\psi_n)\|^2)\), which implies that $\tau_n=O(\|(\phi_n,\psi_n)\|^2)$, and is thus \(\|(\phi_n^{\perp},\psi_n^{\perp})\|=O(\|(\phi_n,\psi_n)\|^2)\). Substituting \((\phi_n,\psi_n)=(\phi_n^{\perp},\psi_n^{\perp})+(\phi_n^{T},\psi_n^{T})\) into $\langle J''(u_0,v_0)(\phi_n,\psi_n),(\phi_n,\psi_n)\rangle$, we obtain 
    \begin{align*}
    \langle J''(u_0,v_0)(\phi_n,\psi_n),(\phi_n,\psi_n)\rangle
    &= \langle J''(u_0,v_0)(\phi_n^T,\psi_n^T),(\phi_n^T,\psi_n^T)\rangle + 2\langle J''(u_0,v_0)(\phi_n^T,\psi_n^T),(\phi_n^{\perp},\psi_n^{\perp})\rangle \\
    &\quad + \langle J''(u_0,v_0)(\phi_n^{\perp},\psi_n^{\perp}),(\phi_n^{\perp},\psi_n^{\perp})\rangle\\
    &= \langle J''(u_0,v_0)(\phi_n^T,\psi_n^T),(\phi_n^T,\psi_n^T)\rangle + O(\|(\phi_n^{\perp},\psi_n^{\perp})\|^2)\\
    &= \langle J''(u_0,v_0)(\phi_n^T,\psi_n^T),(\phi_n^T,\psi_n^T)\rangle + o(\|(\phi_n,\psi_n)\|^2),
    \end{align*}
    where we use $\|(\phi_n^{\perp},\psi_n^{\perp})\|=O(\|(\phi_n,\psi_n)\|^2)$ and $\|(\phi_n,\psi_n)\|\to 0$ as $n\to\infty$ in the last equality.
\end{proof}
   \begin{corollary}
       Assume that $\mathcal G_{a_0,b_0}$ is a singleton
and its unique element is non-degenerate. Then for all \((a,b)\) sufficiently close to \((a_0,b_0)\), the ground state is also locally unique.
   \end{corollary}
\begin{theorem}\label{thm:stability-estimate}
.Let \((u_0,v_0)\) be a non-degenerate ground state of \(J_{a,b}\).
Then there exist constants \(\theta>0\) and \(\delta>0\) such that for every \((u,v)\in\mathcal{H}\) satisfying the Nehari constraint \(\langle J_{a,b}'(u,v),(u,v)\rangle=0\) and the proximity condition \(\|(u,v)-(u_0,v_0)\|_{\mathcal{H}}<\delta\), there holds
\begin{equation}\label{eq:stability}
J_{a,b}(u,v)-d(a,b)\ge \theta\,\|(u,v)-(u_0,v_0)\|_{\mathcal{H}}^2 .
\end{equation}
\end{theorem}
\begin{proof}
Denote \((\phi,\psi)=(u-u_0,v-v_0)\) and \(\Delta_1 = J_{a,b}(u,v)-J_{a,b}(u_0,v_0)\).
Since \((u_0,v_0)\) is a ground state, \(J_{a,b}'(u_0,v_0)=0\),
a Taylor expansion gives
\begin{equation}\label{eq:Delta-expansion}
\Delta_1 = \frac12 \langle J_{a,b}''(u_0,v_0)(\phi,\psi),(\phi,\psi)\rangle + o(\|(\phi,\psi)\|_{\mathcal{H}}^2). 
\end{equation}

Because \(\langle G'(u_0,v_0),(u_0,v_0)\rangle = -2M \neq 0\), the space \(\mathcal{H}\) splits as the direct sum \(\mathcal{H}=T\oplus\mathbb{R}(u_0,v_0)\), where \(T=\ker G'(u_0,v_0)\) is the tangent space.
Hence we can decompose
\begin{equation}\label{eq:30}
(\phi,\psi) = (\phi_T,\psi_T) + \tau(u_0,v_0),\qquad
(\phi_T,\psi_T)\in T,\;\; \tau\in\mathbb{R}. 
\end{equation}

We first estimate the radial coefficient \(\tau\).
Both \((u,v)\) and \((u_0,v_0)\) belong to the Nehari manifold, so \(G(u,v)=G(u_0,v_0)=0\).
Expanding \(G(u,v)=\langle J_{a,b}'(u,v),(u,v)\rangle\) at \((u_0,v_0)\) to second order yields
\[
0 = G(u,v) = \langle G'(u_0,v_0),(\phi,\psi)\rangle
            + \frac12 \langle G''(u_0,v_0)(\phi,\psi),(\phi,\psi)\rangle
            + o(\|(\phi,\psi)\|_{\mathcal{H}}^2). 
\]
Because \((\phi_T,\psi_T)\in T\), \(\langle G'(u_0,v_0),(\phi_T,\psi_T)\rangle =0\).
Using \(\langle G'(u_0,v_0),(u_0,v_0)\rangle = -2M\) we obtain
\[
-2M\tau + \frac12 Q + o(\|(\phi,\psi)\|_{\mathcal{H}}^2) = 0,
\]
where \(Q = \langle G''(u_0,v_0)(\phi,\psi),(\phi,\psi)\rangle = O(\|(\phi,\psi)\|_{\mathcal{H}}^2)\),
thus \(\tau = \frac{Q}{4M} + o(\|(\phi,\psi)\|_{\mathcal{H}}^2) = O(\|(\phi,\psi)\|_{\mathcal{H}}^2)\); in particular, there exists a constant \(C>0\) such that \(|\tau|\le C\|(\phi,\psi)\|_{\mathcal{H}}^2\) whenever \(\|(\phi,\psi)\|_{\mathcal{H}}\) is small.

Now examine the quadratic term in (\ref{eq:Delta-expansion}). Substituting (\ref{eq:30}),
\[
\begin{aligned}
\langle J_{a,b}''(u_0,v_0)(\phi,\psi),(\phi,\psi)\rangle
&= \langle J_{a,b}''(u_0,v_0)(\phi_T,\psi_T),(\phi_T,\psi_T)\rangle \\
&\quad + 2\tau\langle J_{a,b}''(u_0,v_0)(\phi_T,\psi_T),(u_0,v_0)\rangle \\
&\quad + \tau^2\langle J_{a,b}''(u_0,v_0)(u_0,v_0),(u_0,v_0)\rangle .
\end{aligned}
\]
It follows directly from the definition of \(T\) that
\(\langle J_{a,b}''(u_0,v_0)(u_0,v_0),(\phi_T,\psi_T)\rangle = 0\),
and by symmetry the mixed term vanishes.  The radial second
derivative evaluates to
\(\langle J_{a,b}''(u_0,v_0)(u_0,v_0),(u_0,v_0)\rangle = -2M\).
Hence the \(\tau^2\) term is \(-2M\tau^2 =
O\bigl(\|(\phi,\psi)\|_{\mathcal{H}}^4\bigr)\) and can be absorbed
into the remainder.
Thus
\begin{equation}\label{eq:quadratic-reduced}
\langle J_{a,b}''(u_0,v_0)(\phi,\psi),(\phi,\psi)\rangle
= \langle J_{a,b}''(u_0,v_0)(\phi_T,\psi_T),(\phi_T,\psi_T)\rangle
  + o(\|(\phi,\psi)\|_{\mathcal{H}}^2), 
\end{equation}
for \(\|(\phi,\psi)\|\) sufficiently small. 
Non-degeneracy means that \(J_{a,b}''(u_0,v_0)\) restricted to \(T\) is uniformly positive definite,
thus there exists \(\rho>0\) such that
\begin{equation}\label{eq:coercivity}
\langle J_{a,b}''(u_0,v_0)(\phi_T,\psi_T),(\phi_T,\psi_T)\rangle
\ge \rho\,\|(\phi_T,\psi_T)\|_{\mathcal{H}}^2 . 
\end{equation}

It remains to relate \(\|(\phi_T,\psi_T)\|_{\mathcal{H}}\) with \(\|(\phi,\psi)\|_{\mathcal{H}}\).
From the bound on \(\tau\), we obtain
\[
\|(\phi,\psi)\|_{\mathcal{H}} \le \|(\phi_T,\psi_T)\|_{\mathcal{H}} + |\tau|\,\|(u_0,v_0)\|_{\mathcal{H}}
\le \|(\phi_T,\psi_T)\|_{\mathcal{H}} + C' \|(\phi,\psi)\|_{\mathcal{H}}^2 .
\]
Choose \(\delta_1>0\) so small that \(C' \|(\phi,\psi)\|_{\mathcal{H}}^2 \le \frac12 \|(\phi,\psi)\|_{\mathcal{H}}\) for all \(\|(\phi,\psi)\|_{\mathcal{H}}<\delta_1\).
Then \(\|(\phi_T,\psi_T)\|_{\mathcal{H}} \ge \frac12 \|(\phi,\psi)\|_{\mathcal{H}}\).

Insert this into (\ref{eq:quadratic-reduced}) and then into (\ref{eq:30}), and finally into (\ref{eq:Delta-expansion}):
\[
\Delta_1 \ge \frac12 \cdot \frac{\rho}{4} \|(\phi,\psi)\|_{\mathcal{H}}^2 + o(\|(\phi,\psi)\|_{\mathcal{H}}^2) .
\]
Choose \(\delta\le\delta_1\) so small that the remainder term is bounded by \(\frac{\rho}{16}\|(\phi,\psi)\|_{\mathcal{H}}^2\).
Then for \(\|(\phi,\psi)\|_{\mathcal{H}}<\delta\),
\[
\Delta_1 \ge \frac{\rho}{16}\|(\phi,\psi)\|_{\mathcal{H}}^2 = \theta \|(\phi,\psi)\|_{\mathcal{H}}^2 ,
\]
where \(\theta=\frac{\rho}{16}\), which completes the proof.
\end{proof}
\begin{theorem}\label{thm6.5}
Assume that \(V\) is a finite graph, \(p>4\).
Let \((u_0,v_0)\) be a non-degenerate ground state of \(J_{a,b}\) and suppose that all ground states of \(J_{a,b}\) are non-degenerate.
Denote by \(\mathcal{L}\) the set of all ground states, which is finite by Theorem~\ref{thm6.3}.
Then there exist constants \(\delta_0>0\), \(\iota>0\) and \(\theta_0>0\) with the following properties:
\begin{enumerate}
\item[(i)] For every \((u_i,v_i)\in\mathcal{L}\) and every \((u,v)\in\mathcal{N}_{a,b}\) with \(\|(u,v)-(u_i,v_i)\|_{\mathcal{H}}<\delta_0\), there holds
\[
J_{a,b}(u,v)-d(a,b) \ge \theta_0\,\|(u,v)-(u_i,v_i)\|_{\mathcal{H}}^2 .
\]
\item[(ii)] If \((u,v)\in\mathcal{N}_{a,b}\) satisfies \(\operatorname{dist}_{\mathcal{H}}\bigl((u,v),\mathcal{L}\bigr)\ge\delta_0\), then
\[
J_{a,b}(u,v)-d(a,b) \ge \iota .
\]
\end{enumerate}
\end{theorem}
    \begin{proof}
For each ground state \((u_i,v_i)\in\mathcal{L}\), 
Theorem~\ref{thm:stability-estimate} supplies constants 
\(\tilde\delta_i>0\) and \(\tilde\theta_i>0\) such that the quadratic 
estimate holds on \(B_{\tilde\delta_i}(u_i,v_i)\cap\mathcal{N}_{a,b}\).  
Since \(\mathcal{L}\) is finite, we set
\[
\theta_0 = \frac14 \min_i \tilde\theta_i,\qquad
\delta_* = \min_i \tilde\delta_i,\qquad
\rho=
\begin{cases}
\min_{i\ne j}\|(u_i,v_i)-(u_j,v_j)\|_H,&|\mathcal L|\ge2,\\
+\infty,&|\mathcal L|=1.
\end{cases}.
\]
Choosing \(\delta_0 = \frac12\min\{\delta_*,\rho\}\) ensures that the 
balls \(B_{\delta_0}(u_i,v_i)\) are pairwise disjoint and that on each 
of them the estimate of Theorem~\ref{thm:stability-estimate} holds 
with constant \(\theta_0\), which establishes part (i).

Suppose, to the contrary, that this is false. Then there exists a sequence
$\{(u_n,v_n)\}\subset\mathcal N_{a,b}$ such that
\[
\operatorname{dist}_{\mathcal{H}}\bigl((u_n,v_n),\mathcal L_{a,b}\bigr)
\ge\delta_0
\]
for every $n$, while
\[
J_{a,b}(u_n,v_n)\longrightarrow d(a,b).
\]
By the compactness result for ground-state-level sequences established
above, after passing to a subsequence, we have
\[
(u_n,v_n)\longrightarrow (u_*,v_*)\qquad\text{strongly in }\mathcal{H}.
\]
By the continuity of $J_{a,b}$,
\[
J_{a,b}(u_*,v_*)
=\lim_{n\to\infty}J_{a,b}(u_n,v_n)
=d(a,b).
\]
Moreover, since $\mathcal N_{a,b}$ is closed, $(u_*,v_*)\in\mathcal N_{a,b}$.
Consequently,
\[
(u_*,v_*)\in\mathcal L_{a,b}.
\]

On the other hand, the distance function to the compact set
$\mathcal L_{a,b}$ is continuous. Hence
\[
\operatorname{dist}_{\mathcal{H}}\bigl((u_*,v_*),\mathcal L_{a,b}\bigr)
=
\lim_{n\to\infty}
\operatorname{dist}_{\mathcal{H}}\bigl((u_n,v_n),\mathcal L_{a,b}\bigr)
\ge\delta_0,
\]
which contradicts $(u_*,v_*)\in\mathcal L_{a,b}$.

Therefore,
\[
\inf_{\substack{(u,v)\in\mathcal N_{a,b}\\
\operatorname{dist}_{\mathcal{H}}((u,v),\mathcal L_{a,b})\ge\delta_0}}
J_{a,b}(u,v)>d(a,b).
\]
Set
\[
\iota:=
\inf_{\substack{(u,v)\in\mathcal N_{a,b}\\
\operatorname{dist}_{\mathcal{H}}((u,v),\mathcal L_{a,b})\ge\delta_0}}
J_{a,b}(u,v)-d(a,b)>0.
\]
Then
\[
J_{a,b}(u,v)-d(a,b)\ge\iota
\]
for every $(u,v)\in\mathcal N_{a,b}$ satisfying
\[
\operatorname{dist}_{\mathcal{H}}\bigl((u,v),\mathcal L_{a,b}\bigr)
\ge\delta_0.
\]
This proves \textup{(ii)}.
\end{proof}

We now extend the perturbation theory to the second order, where the logarithmic coupling leaves its most refined structural imprint. Fix a reference pair \((a_0,b_0)\) and let \((u_0,v_0)\) be a non-degenerate ground state of \(J_{a_0,b_0}\).  
For small perturbations \(a=a_0+\delta a_1\), \(b=b_0+\delta b_1\), let \((u,v)\) be the corresponding ground state. Before the proof of Theorem~\ref{thm5}, we refer readers to~\cite{BS00,NGH25,DR25} for related results in constrained optimisation sensitivity analysis.

\subsection*{Proof of Theorem~\ref{thm5}}
Motivated by~\cite{FIE81}, we employ the Schur complement technique to prove Theorem~\ref{thm5}. The ground state $(u_0,v_0)$ is the local solution of the constrained minimisation problem
\begin{equation}\label{eq:constrained-problem}
\min_{(u,v)\in\mathcal{H}}\; J_{a_0,b_0}(u,v)\quad\text{subject to}\quad G(a_0,u,v)=0,
\end{equation}
where $G(a,u,v)=\langle J_{a,b_0}'(u,v),(u,v)\rangle$ is the Nehari functional. Define the Lagrangian
\[
\mathcal{L}(a,u,v,\lambda)=J_{a,b_0}(u,v)+\lambda G(a,u,v).
\]
At the reference point $a=a_0$, the ground state $(u_0,v_0)$ satisfies the first-order necessary optimality conditions
\begin{equation}\label{eq:first-order}
\nabla_{(u,v)}\mathcal{L}(a_0,u_0,v_0,\lambda_0)=0,\qquad G(a_0,u_0,v_0)=0,
\end{equation}
with the Lagrange multiplier $\lambda_0=0$ due to the fact that the ground state satisfies the unconstrained Euler--Lagrange equation $J'(u_0,v_0)=0$, so the multiplier can be taken as zero.
Now we invoke the sensitivity analysis theory for constrained optimisation problems,
the following conditions must be verified at the reference point $(a_0,u_0,v_0,\lambda_0)$.

The constraint $G(a_0,u_0,v_0)=0$ has a non-zero derivative with respect to $(u,v)$.  
Indeed, 
\[
\langle G'(u_0,v_0),(u_0,v_0)\rangle = -2M < 0,
\]
hence Linear Independence Constraint Qualification holds, we call it LICQ briefly.
  
The Lagrangian Hessian with respect to $(u,v)$ at the solution is
\[
\nabla^2_{(u,v)}\mathcal{L}(a_0,u_0,v_0,\lambda_0)=J_{a_0,b_0}''(u_0,v_0).
\]
The tangent space is defined as Definition~\ref{def:5.1}, non-degeneracy states precisely that $J''(u_0,v_0)$ restricted to this space is an isomorphism. Since the ground state is a minimum on the Nehari manifold, the restricted Hessian is in fact uniformly positive definite. Thus the second-order sufficient condition holds.

On the finite graph $V$, the space $\mathcal{H}$ is finite-dimensional, and the ground state satisfies $u_0(x)>0$, $v_0(x)>0$ pointwise according to Proposition~\ref{Pro1}.  
In a neighbourhood of $(u_0,v_0)$, the logarithmic terms are smooth, so $J_{a,b}$ and $G$ are $C^2$ with respect to all variables. Under these conditions, we can guarantee that the local branch energy $\mathcal E(a)$ is second-order directionally differentiable at $a_0$, and its second derivative can be computed from the solution of a certain linearised system. In what follows we are about to obtain the linearised system associated with the constrained problem~\eqref{eq:constrained-problem} at $(a_0,u_0,v_0,\lambda_0)$.

At the reference point \((a_0,u_0,v_0,\lambda_0)\), where \(\lambda_0=0\), the ground state satisfies the first‑order optimality conditions
\[
\nabla_{(u,v)}\mathcal{L}=0,\qquad G=0,
\]
where \(\mathcal{L}=J_{a,b_0}+\lambda G\) is the Lagrangian and \(G=\langle J_{a,b_0}'(u,v),(u,v)\rangle\) is the Nehari functional. Throughout the rest of this proof, we write \(J'' = J_{a,b_0}''\).

By the non‑degeneracy hypothesis, the bordered Hessian at the reference point \((a_0,u_0,v_0,\lambda_0=0)\) is invertible. Applying the implicit function theorem to the constrained
system therefore yields a neighbourhood \(\mathcal{U}\subset\mathbb{R}^N\) of \(a_0\) and unique \(C^1\) maps
\[
\mathcal{U}\ni a\mapsto (u_a,v_a,\lambda_a)\in\mathcal{H}\times\mathbb{R}
\]
such that for every \(a\in\mathcal{U}\),
\begin{equation}\label{eq:IFT-branch}
\nabla_{(u,v)}\mathcal{L}(a,u_a,v_a,\lambda_a)=0,\qquad G(a,u_a,v_a)=0,
\end{equation}
with \((u_{a_0},v_{a_0},\lambda_{a_0})=(u_0,v_0,0)\).
Given a perturbation direction \(\delta a_1\), we define the linear responses
\[
    \dot u = \frac{d}{d\varepsilon}\bigg|_{\varepsilon=0} u_{a_0+\varepsilon \delta a_1}, \quad
\dot v = \frac{d}{d\varepsilon}\bigg|_{\varepsilon=0} v_{a_0+\varepsilon \delta a_1}, \quad
\dot\lambda = \frac{d}{d\varepsilon}\bigg|_{\varepsilon=0} \lambda_{a_0+\varepsilon \delta a_1}.
\]
Differentiating \(\nabla_{(u,v)}\mathcal{L}=0\) along \(\delta a_1\) and evaluating at the reference point, the terms involving \(\lambda_0\) vanish, yielding
\begin{equation}\label{(34)}
     J''(u_0,v_0)(\dot u,\dot v)+G'(u_0,v_0)\dot\lambda+\partial_a J'(u_0,v_0)\delta a_1=0. 
\end{equation}
Similarly, differentiating the constraint \(G=0\) at the reference point gives
\begin{equation}\label{(35)}
    G'(u_0,v_0)(\dot u,\dot v)+\partial_a G(a_0,u_0,v_0)\delta a_1=0, 
\end{equation}
Writing \eqref{(34)}–\eqref{(35)} in block matrix form produces the bordered system
\begin{equation}\label{eq:matrix}
\begin{pmatrix}
J''(u_0,v_0) & G'(u_0,v_0)^* \\[4pt]
G'(u_0,v_0) & 0
\end{pmatrix}
\begin{pmatrix} (\dot u,\dot v) \\ \dot\lambda \end{pmatrix}
= -\begin{pmatrix} \partial_a J'(u_0,v_0)\delta a_1 \\[4pt] \partial_a G(a_0,u_0,v_0)\delta a_1 \end{pmatrix},
\end{equation}
where the zero in the lower‑right block records the absence of \(\dot\lambda\) in the second equation.
The explicit form of the right‑hand side,
\[
\partial_aJ'(u_0,v_0)[\delta a_1]
=\bigl((\mu(x)u_0(x)^{p-1}\delta a_1(x))_{x\in V},\,0\bigr),\qquad
\partial_a G(a_0,u_0,v_0)\delta a_1=\int_V u_0^p\,\delta a_1\,d\mu=:B.
\]

By the LICQ and second-order sufficient condition, the bordered operator on the left-hand side of (\ref{eq:matrix}) is invertible.  
Hence the system admits a unique solution $(\dot u,\dot v,\dot\lambda)$ for every direction $\delta a_1$.
Since $\langle G'(u_0,v_0),(u_0,v_0)\rangle\neq 0$, the radial direction $(u_0,v_0)$ is transversal to $T$.  
Decompose the solution $(\dot u,\dot v)$ uniquely as
\begin{equation}\label{eq:LS-decomp}
(\dot u,\dot v) = (\dot u_T,\dot v_T) + \tau\,(u_0,v_0),\qquad
(\dot u_T,\dot v_T)\in T,\;\; \tau\in\mathbb{R}.
\end{equation}
Insert this decomposition into the second equation of (\ref{eq:matrix}), we obtain
\[
\langle G'(u_0,v_0),(\dot u_T,\dot v_T)\rangle + \tau\langle G'(u_0,v_0),(u_0,v_0)\rangle
= -\int_V u_0^p\,\delta a_1\,d\mu .
\]
The first term vanishes by definition of $T$, and $\langle G'(u_0,v_0),(u_0,v_0)\rangle = -2M$,  
thus
\begin{equation}\label{eq:tau}
\tau = \frac{1}{2M}\int_V u_0^p\,\delta a_1\,d\mu =: \frac{B}{2M}.
\end{equation}
Take the inner product of the first equation of (\ref{eq:matrix}) with the radial direction $(u_0,v_0)$ and 
using the decomposition~\eqref{eq:LS-decomp} and the bilinearity of $J''$, we obtain
\begin{equation}\label{eq:radial-inner}
\begin{aligned}
\langle J''(u_0,v_0)(\dot u_T,\dot v_T),(u_0,v_0)\rangle
+ \tau\langle J''(u_0,v_0)(u_0,v_0),(u_0,v_0)\rangle \\
+ \dot\lambda\langle G'(u_0,v_0),(u_0,v_0)\rangle
= -\langle (|u_0|^{p-2}u_0\,\delta a_1,0), (u_0,v_0)\rangle .
\end{aligned}
\end{equation}
The right-hand side simplifies to $-\int_V |u_0|^{p-2}u_0\cdot u_0\,\delta a_1\,d\mu = -\int_V u_0^p\,\delta a_1\,d\mu = -B$. 
For any $(\phi,\psi)\in T$, the condition $\langle G'(u_0,v_0),(\phi,\psi)\rangle=0$ expands to
\[
\langle J''(u_0,v_0)(\phi,\psi),(u_0,v_0)\rangle
+ \langle J'(u_0,v_0),(\phi,\psi)\rangle = 0.
\]
Since $J'(u_0,v_0)=0$, the second term vanishes, and by symmetry of $J''$ we obtain
\begin{equation}\label{eq:orthogonality}
\langle J''(u_0,v_0)(u_0,v_0),(\phi,\psi)\rangle = 0, \qquad\forall\,(\phi,\psi)\in T .
\end{equation}
On the other hand, a direct computation yields the numerical identity
\begin{equation}\label{eq:radial-second}
\langle J''(u_0,v_0)(u_0,v_0),(u_0,v_0)\rangle = -2M .
\end{equation}
Applying \eqref{eq:orthogonality} with $(\phi,\psi)=(\dot u_T,\dot v_T)$ eliminates the first term in~\eqref{eq:radial-inner}. Then by using \eqref{eq:radial-second} and $\langle G'(u_0,v_0),(u_0,v_0)\rangle=-2M$, equation~\eqref{eq:radial-inner} reduces to
\[
-2M\tau - 2M\dot\lambda = -B .
\]
Substituting $\tau = B/(2M)$ from~\eqref{eq:tau} gives
\[
-B - 2M\dot\lambda = -B,
\]
 which yields \(\dot{\lambda}=0\), thus the first-order variation of the Lagrange multiplier along the ground state branch vanishes identically.
With \(\dot\lambda=0\), the linearised system simplifies to
\[
J''(u_0,v_0)(\dot u,\dot v)=-(u_0^{p-1}\,\delta a_1,0),\qquad
\langle G',(\dot u,\dot v)\rangle=-B .
\]
Insert the decomposition (\ref{eq:LS-decomp}) into the first equation of (\ref{eq:matrix}) and project orthogonally onto the tangent space \(T\).
By \eqref{eq:orthogonality}, the radial part \(J''(u_0,v_0)(\tau (u_0,v_0))\) has zero projection onto \(T\).
Denoting by \(\mathbf{H}=J''(u_0,v_0)|_T:T\to T^*\) the restricted Hessian, we obtain, for every \((\phi,\psi)\in T\),
\begin{equation}\label{eq:proj-var}
\langle \mathbf{H}(\dot u_T,\dot v_T),(\phi,\psi)\rangle
=-\langle (u_0^{p-1}\,\delta a_1,0),(\phi,\psi)\rangle .
\end{equation}
The right-hand side of (\ref{eq:proj-var}) defines a linear functional on \(T\) which depends only on the \(u\)-component \(\phi\) of the test function.

We now express this equation in coordinates.
Choose an orthonormal basis \(\{e_i\}_{i=1}^m\) of \(T\), where \(m=2n-1=\dim T\) and \(e_i=(\phi_i,\psi_i)\). Indeed, since \(V\) is finite with \(|V|=n\), the full space \(\mathcal{H}=C(V)\times C(V)\) therefore has dimension \(2n\). The Nehari constraint \(G(u,v)=0\) is regular; therefore the tangent space \(T=\ker G'(u_0,v_0)\) is a hyperplane in \(C(V)\times C(V)\) of dimension \(m=2n-1\).
Write the coordinates of \((\dot u_T,\dot v_T)\) in this basis as \(\xi=(\xi_1,\dots,\xi_m)^{\mathsf{T}}\in\mathbb{R}^m\), i.e.\
\[
(\dot u_T,\dot v_T)=\sum_{i=1}^m \xi_i e_i .
\]
Let \(H_{\rm mat}\in\mathbb{R}^{m\times m}\) be the matrix with entries
\[
(H_{\rm mat})_{ij}=\langle \mathbf{H}e_j,e_i\rangle .
\]
Since \(\mathbf{H}\) is positive definite on \(T\), \(H_{\rm mat}\) is also symmetric and positive definite.
Define the matrix \(U\in\mathbb{R}^{n\times m}\) by \(U_{xi}= \phi_i(x)\) with the \(x\)-th component of the \(u\)-part of the \(i\)-th basis vector.
Then the \(u\)-component of \((\dot u_T,\dot v_T)\) is
\[
\dot u_T = U\xi .
\]
Now evaluate the right-hand side of~\eqref{eq:proj-var} at the basis vector \(e_i\):
\[
b_i := -\langle (u_0^{p-1}\,\delta a_1,0),e_i\rangle
= -\langle u_0^{p-1}\,\delta a_1,\phi_i\rangle .
\]
Set
\[
f=\bigl(\mu(x)u_0(x)^{p-1}\delta a_1(x)\bigr)_{x\in V}.
\]
The vector \(b=(b_1,\dots,b_m)^{\mathsf T}\) is given by
\[
b=-U^{\mathsf T}f.
\]
Thus, equation~\eqref{eq:proj-var} becomes
\[
H_{\mathrm{mat}}\xi=b=-U^{\mathsf T}f.
\]
Since \(H_{\mathrm{mat}}\) is invertible,
\[
\xi=-H_{\mathrm{mat}}^{-1}U^{\mathsf T}f,
\]
and therefore
\begin{equation}\label{eq:dot-uT}
\dot u_T=U\xi
=-U H_{\mathrm{mat}}^{-1}U^{\mathsf T}f.
\end{equation}
Define
\[
\mathsf H_{\mathrm{eff}}
:=U H_{\mathrm{mat}}^{-1}U^{\mathsf T}
\in\mathbb R^{n\times n}.
\]
The columns of \(U\) are the \(u\)-components of an orthonormal basis of
\(T\). Since \(g_v\neq0\), the projection
\(\pi_u(\phi,\psi)=\phi\) maps \(T\) onto \(\mathbb{R}^n\); hence the columns
of \(U\) span \(\mathbb{R}^n\) and \(U\) has full row rank \(n\).
Consequently, for any non-zero \(y\in\mathbb{R}^n\), the vector
\(U^{\mathsf{T}}y\) is non-zero. Because
\(H_{\mathrm{mat}}^{-1}\) is positive definite,
\[
y^{\mathsf{T}}\mathsf{H}_{\mathrm{eff}}y
=(U^{\mathsf{T}}y)^{\mathsf{T}}
 H_{\mathrm{mat}}^{-1}(U^{\mathsf{T}}y)>0,
\]
so \(\mathsf{H}_{\mathrm{eff}}
=UH_{\mathrm{mat}}^{-1}U^{\mathsf{T}}\) is symmetric and positive definite.
With
\[
W=\operatorname{diag}\bigl(\mu(x)u_0(x)^{p-1}\bigr)_{x\in V},
\]
we have \(f=W\delta a_1\), and therefore
\begin{equation}\label{eq:eff}
\dot u_T=-\mathsf{H}_{\mathrm{eff}}W\delta a_1.
\end{equation}
This is the desired effective equation for the tangential component.

Lemma~\ref{Lem:Frechet differential} provides the ordinary-coordinate
gradient formula
\[
\nabla_a\mathcal{E}(a)
=\frac1p\bigl(\mu(x)|u_a(x)|^p\bigr)_{x\in V}
\]
along the local \(C^1\) branch of positive critical points. Since the map
\(a\mapsto u_a\) is of class \(C^1\) and \(u_a>0\) pointwise, we may
differentiate once more. For the second directional derivative along
\(\delta a_1\), we obtain
\begin{equation}\label{eq:d2}
\begin{aligned}
\mathcal{E}''(a_0)(\delta a_1,\delta a_1)
&=\frac{d}{dt}\bigg|_{t=0}
  \frac1p\int_V |u_{a_0+t\delta a_1}|^p
  \,\delta a_1\,d\mu\\
&=\int_V u_0^{p-1}\dot u\,\delta a_1\,d\mu.
\end{aligned}
\end{equation}
Now insert the decomposition
\(\dot u=\dot u_T+\tau u_0\) with \(\tau=B/(2M)\):
\[
\begin{aligned}
\mathcal{E}''(a_0)(\delta a_1,\delta a_1)
&=\int_V u_0^{p-1}\dot u_T\,\delta a_1\,d\mu
  +\tau\int_V u_0^p\,\delta a_1\,d\mu\\
&=\langle W\delta a_1,\dot u_T\rangle
  +\frac{B^2}{2M}.
\end{aligned}
\]
Using the effective equation~\eqref{eq:eff}, we obtain
\[
\begin{aligned}
\mathcal{E}''(a_0)(\delta a_1,\delta a_1)
&=-\langle W\delta a_1,
  \mathsf{H}_{\mathrm{eff}}W\delta a_1\rangle
  +\frac{B^2}{2M}\\
&=-\langle W\delta a_1,
  \mathsf{H}_{\mathrm{eff}}W\delta a_1\rangle
  +\frac{1}{2M}
  \left\langle
  \bigl(\mu(x)u_0(x)^p\bigr)_{x\in V},
  \delta a_1
  \right\rangle^2.
\end{aligned}
\]

Since this identity holds for every direction \(\delta a_1\), the symmetric
bilinear form associated with the Hessian of \(\mathcal{E}\) at \(a_0\) is
\[
\mathcal{E}''(a_0)
=-W\mathsf{H}_{\mathrm{eff}}W
+\frac{1}{2M}
\bigl(\mu(x)u_0(x)^p\bigr)_{x\in V}
\bigl(\mu(x)u_0(x)^p\bigr)_{x\in V}^{\mathsf{T}}.
\]

Consequently, for \(\Phi=-\mathcal{E}\),
\[
D_a^2\Phi(a_0)
=W\mathsf{H}_{\mathrm{eff}}W
-\frac{1}{2M}
\bigl(\mu(x)u_0(x)^p\bigr)_{x\in V}
\bigl(\mu(x)u_0(x)^p\bigr)_{x\in V}^{\mathsf{T}}.
\]
This is exactly the factorisation stated in the theorem, and the proof is
complete.
\begin{corollary}\label{thm:5.9}
Let \((u_0,v_0)\) be a non-degenerate ground state of
\(J_{a_0,b_0}\), and let
\(a\mapsto (u_a,v_a)\) be the local \(C^1\) branch of positive
critical points through \((u_0,v_0)\) given by
Theorem~\ref{thm5}. Let
\[
\mathcal E(a):=J_{a,b_0}(u_a,v_a),
\qquad
\Phi:=-\mathcal E.
\]
Recall from Theorem~\ref{thm5} that the Hessian of $\Phi$ factorises as
\[
D^2_a\Phi(a_0)
=
W\mathsf H_{\mathrm{eff}}W
-\frac{1}{2M}
(\mu u_0^p)(\mu u_0^p)^{\mathsf T}.
\]
Define the intrinsic rigidity scalar
\[
\mathcal{R}(u_0)
=
\frac{1}{2M}\,
u_0^{\mathsf T}\mathsf H_{\mathrm{eff}}^{-1}u_0.
\]

Then the following three statements are equivalent:
\begin{enumerate}[label=(\arabic*),itemsep=4pt,topsep=4pt]
\item The Hessian \(D^2_a\Phi(a_0)\) is degenerate, that is, there exists
a non-zero direction \(\delta a^*\in\mathbb R^n\) such that
\[
D^2_a\Phi(a_0)\,\delta a^*=0.
\]
In this case, the corresponding kernel direction is flat for the
branch energy, and the first-order variation of the \(u\)-component
along the local critical-point branch vanishes.

\item \(\mathcal{R}(u_0)=1\), equivalently,
\[
u_0^{\mathsf T}\mathsf H_{\mathrm{eff}}^{-1}u_0=2M.
\]

\item Along the direction
\[
\delta a^*
=
W^{-1}\mathsf H_{\mathrm{eff}}^{-1}u_0,
\]
the branch energy has zero second-order variation, namely,
\[
\mathcal E''(a_0)(\delta a^*,\delta a^*)=0.
\]
Moreover, \(\delta a^*\) spans the kernel of
\(D^2_a\Phi(a_0)\).
\end{enumerate}

When these equivalent conditions hold, the Hessian
\(D^2_a\Phi(a_0)\) loses positive definiteness exactly at this
algebraic threshold, the branch energy has no second-order response
along \(\delta a^*\), and the \(u\)-component of the local
critical-point branch has no linear response in that direction.
The full state \((u_a,v_a)\) need not have zero linear response.
Thus \(\mathcal{R}(u_0)=1\) is the sharp algebraic threshold for the
degeneracy of the parameter Hessian; this degeneracy alone does not
imply a bifurcation of the local critical-point branch.
\end{corollary}

\begin{proof}
Set
\[
A:=W\mathsf H_{\mathrm{eff}}W,
\qquad
v:=\frac{1}{\sqrt{2M}}(\mu u_0^p),
\qquad
\delta a^*:=W^{-1}\mathsf H_{\mathrm{eff}}^{-1}u_0.
\]
Then
\[
D^2_a\Phi(a_0)=A-vv^{\mathsf T},
\]
and
\[
v^{\mathsf T}A^{-1}v
=
\frac1{2M}
(\mu u_0^p)^{\mathsf T}
W^{-1}\mathsf H_{\mathrm{eff}}^{-1}
W^{-1}(\mu u_0^p)
=
\frac1{2M}
u_0^{\mathsf T}\mathsf H_{\mathrm{eff}}^{-1}u_0
=
\mathcal R(u_0).
\]
Since \(A\) is positive definite, the Sherman--Morrison criterion
implies
\[
D^2_a\Phi(a_0)\ \text{is degenerate}
\Longleftrightarrow
\mathcal R(u_0)=1.
\]
Moreover, in the degenerate case,
\[
\ker D^2_a\Phi(a_0)
=
\operatorname{span}\{\delta a^*\}.
\]
Thus \((1)\Leftrightarrow(2)\).

By Theorem~\ref{thm5},
\[
\mathcal E''(a_0)(\delta a,\delta a)
=
-\langle W\delta a,\mathsf H_{\mathrm{eff}}W\delta a\rangle
+
\frac1{2M}
\langle \mu u_0^p,\delta a\rangle^2.
\]
For \(\delta a=\delta a^*\), set
\[
x:=u_0^{\mathsf T}\mathsf H_{\mathrm{eff}}^{-1}u_0>0.
\]
Since
\[
W\delta a^*=\mathsf H_{\mathrm{eff}}^{-1}u_0,
\qquad
\langle\mu u_0^p,\delta a^*\rangle=x,
\]
we obtain
\[
\mathcal E''(a_0)(\delta a^*,\delta a^*)
=
-x+\frac{x^2}{2M}.
\]
Hence
\[
\mathcal E''(a_0)(\delta a^*,\delta a^*)=0
\Longleftrightarrow
x=2M
\Longleftrightarrow
\mathcal R(u_0)=1.
\]
Therefore \((2)\Leftrightarrow(3)\).

Finally, along the local critical-point branch, write
\[
\dot u=\dot u_T+\tau u_0.
\]
The linearisation formulas give
\[
\mathsf H_{\mathrm{eff}}^{-1}\dot u_T=-W\delta a^*
=-\mathsf H_{\mathrm{eff}}^{-1}u_0,
\qquad
\tau=\frac1{2M}
\langle\mu u_0^p,\delta a^*\rangle
=\frac{x}{2M}.
\]
Thus
\[
\dot u
=
-u_0+\frac{x}{2M}u_0
=
\left(\frac{x}{2M}-1\right)u_0.
\]
At the rigidity threshold \(x=2M\), we have \(\dot u=0\). This concerns
only the \(u\)-component of the local critical-point branch and does not
assert that the full state has zero linear response.
\end{proof}

\section{Proof of Theorem~\ref{thm6} and Theorem~\ref{thm7}: Asymptotic convergence to the limiting Dirichlet problem.}
\label{sec:asymptotic}

In this section we prove Theorem~\ref{thm6} and Theorem~\ref{thm7}. The former establishes convergence of ground states of \eqref{eq:system4} to those of the limiting Dirichlet problem \eqref{eq:system5} as $\Lambda\to\infty$; the latter provides the quantitative rate estimate. The proof relies on two lemmas established below.

\begin{lemma}\label{lem:energy-lower}
Let $\{(u_k,v_k)\}\subset\mathcal H_{\Lambda_k}$ be a sequence with
$\Lambda_k\to\infty$ satisfying the Cerami condition for the functional
$J_{\Lambda_k}$ at level $c$, i.e.,
\[
J_{\Lambda_k}(u_k,v_k)\to c,\qquad
\bigl(1+\|(u_k,v_k)\|_{H_{\Lambda_k}}\bigr)\,
\|J_{\Lambda_k}'(u_k,v_k)\|_{H_{\Lambda_k}'}\to0 .
\]
Then either $c=0$ or there exists $\delta>0$ such that $c\ge\delta$.
\end{lemma}

\begin{proof}
From the Cerami condition we obtain
$\langle J_{\Lambda_k}'(u_k,v_k),(u_k,v_k)\rangle =o(1)$.
The elementary identity
\[
J_{\Lambda_k}(u_k,v_k)-\frac1p\langle J_{\Lambda_k}'(u_k,v_k),(u_k,v_k)\rangle
 =\frac{2}{p^2}\int_V\bigl(|u_k|^{p-2}v_k^2+|v_k|^{p-2}u_k^2\bigr)\,d\mu
\]
applied to $(u_k,v_k)$ yields
\[
\int_V\bigl(|u_k|^{p-2}v_k^2+|v_k|^{p-2}u_k^2\bigr)\,d\mu
 \longrightarrow \frac{p^2}{2}c .
\]
If $c=0$ we are done.  Assume $c>0$,  then there exists $C_0>0$ such that
\begin{equation}\label{eq:mix-L}
\int_V|u_k|^{p-2}v_k^2\,d\mu\le C_0,\qquad
\int_V|v_k|^{p-2}u_k^2\,d\mu\le C_0\qquad\forall k\in N^{+} .
\end{equation}
Set $X_k:=\|u_k\|_{H_{\Lambda_k,a}}^p+\|v_k\|_{H_{\Lambda_k,b}}^p$.
Testing $J_{\Lambda_k}'(u_k,v_k)$ with $(u_k,v_k)$ gives
\begin{equation}\label{eq:nehari-L}
X_k = \int_V|u_k|^{p-2}v_k^2\log v_k^2\,d\mu
     +\int_V|v_k|^{p-2}u_k^2\log u_k^2\,d\mu
     +\frac2p\int_V\bigl(|u_k|^{p-2}v_k^2+|v_k|^{p-2}u_k^2\bigr)\,d\mu+o(1),
\end{equation}

By applying the exponent calibration technique via the
$L^\infty$ embedding similar as Lemma~\ref{lem4.3} yields the decisive inequality
\begin{equation}\label{eq:Xk-final-L}
X_k \le C X_k^{\varepsilon/p} + C .
\end{equation}
Fix any $\varepsilon\in(0,p)$, then $\varepsilon/p<1$. Hence $\{X_k\}$ is bounded, and so is
$\{(u_k,v_k)\}$ in $\mathcal{H}_{\Lambda_k}$.

With boundedness established, the energy identity again gives
\[
c = \lim_{k\to\infty} J_{\Lambda_k}(u_k,v_k)
  = \frac{2}{p^2}\lim_{k\to\infty}\int_V\bigl(|u_k|^{p-2}v_k^2+|v_k|^{p-2}u_k^2\bigr)\,d\mu .
\]
If $c>0$, the lower bound $c\ge\delta>0$ follows from the Nehari
manifold estimate as in the proof of Lemma~\ref{lem4.3} adapted to
$J_{\Lambda_k}$.  This completes the proof.
\end{proof}

\begin{lemma}\label{lem:energy-conv}
Under the hypotheses of Theorem~\ref{thm6},
$d_\Lambda \to d_\Omega$ as $\Lambda\to\infty$.
\end{lemma}

\begin{proof}
Let $\Lambda_k\to\infty$ be an increasing sequence. By
Theorem~\ref{thm1}, for each $k$ there exists a ground
state $(u_k,v_k)\in\mathcal N_{\Lambda_k}$ with
$J_{\Lambda_k}(u_k,v_k)=d_{\Lambda_k}$ and
$J_{\Lambda_k}'(u_k,v_k)=0$. Since $\mathcal N_\Omega\subset\mathcal N_{\Lambda_k}$, we have $d_{\Lambda_k}\le d_\Omega$ for all $\Lambda_k$. Lemma~\ref{lem:energy-lower}
guarantees that $\{(u_k,v_k)\}$ is bounded in $\mathcal{H}_{\Lambda_k}$.
Passing to a subsequence if necessary, we may assume
\[
(u_k,v_k)\to(u,v)\ \text{pointwise in }V,
\]
with strong convergence in $L^q\times L^q$ for all $q\ge p/2$.

We now prove that $u=0$ on $\Omega_a^c$ and $v=0$ on $\Omega_b^c$.
Suppose there exists $x_0\in\Omega_a^c$ with $u(x_0)\neq0$.
Then $|u_k(x_0)|\ge\frac12|u(x_0)|>0$ for all large $k$.
Since $a(x_0)>0$ and $\Lambda_k\to\infty$,
$\Lambda_k a(x_0)\to\infty$.  From the Nehari identity and the
exponent calibration estimate via \(L^{\infty}\) embedding used in Lemma~\ref{lem:energy-lower},
the principal part of the energy is bounded above by a constant,
yet the term $(1+\Lambda_k a(x_0))|u_k(x_0)|^p$ contained in it
tends to $+\infty$, a contradiction.  Thus $u=0$ on $\Omega_a^c$,
and similarly $v=0$ on $\Omega_b^c$.  Consequently
$(u,v)\in\mathcal H_\Omega$.

Indeed, if \(M(u,v)=0\) then the minimising sequence would converge to zero, which is a contradiction. We now prove that \(\lim\limits_{k\to\infty} d_{\Lambda_k}=d_\Omega\).
For a fixed pair \((u,v)\neq(0,0)\) and any \(t>0\), the fibre map \(\gamma(t)=J_\Omega(t(u,v))\) satisfies
\[
\frac{\gamma'(t)}{t^{p-1}} = \|u\|_{W_0^{1,p}(\Omega_a)}^p+\|v\|_{W_0^{1,p}(\Omega_b)}^p
- \bigl(L(u,v) + (\log t^2+\tfrac{2}{p})M(u,v)\bigr),
\]
where \(L(u,v)\) is the logarithmic interaction and \(M(u,v)>0\) is the mixed integral.
Consequently, \(\gamma\) attains its unique maximum at some \(t_0>0\), and \(t_0(u,v)\in\mathcal{N}_\Omega\).

Using \(t_0(u,v)\) as a test function for \(d_\Omega\) and applying the weak lower semicontinuity of the norm together with Vitali's convergence theorem, the latter being justified by the pointwise convergence and the uniform integrability guaranteed by the \(\mathcal{H}\)-boundedness of \(\{(u_k,v_k)\}\)), we obtain
\[
\begin{aligned}
d_\Omega \le J_\Omega(t_0(u,v))
&\le \frac{1}{p}\liminf_{k\to\infty}\bigl(\|t_0u_k\|_{\mathcal{H}_{\Lambda_k,a}}^p+\|t_0v_k\|_{\mathcal{H}_{\Lambda_k,b}}^p\bigr)
   -\frac{1}{p}\lim_{k\to\infty} L(t_0u_k,t_0v_k)\\[4pt]
&= \liminf_{k\to\infty} J_{\Lambda_k}(t_0(u_k,v_k))
\le \liminf_{k\to\infty} J_{\Lambda_k}(u_k,v_k)
= \liminf_{k\to\infty} d_{\Lambda_k}.
\end{aligned}
\]
On the other hand, \(d_{\Lambda_k}\le d_\Omega\) for all \(k\).
Combining the two inequalities yields \(\lim\limits_{k\to\infty} d_{\Lambda_k}=d_\Omega\), as we want to prove. 
\end{proof}
\medskip
\begingroup
\renewcommand{\proofname}{}
\makeatletter
\renewenvironment{proof}[1][\proofname]{\par\noindent\ignorespaces}{\par\hfill$\square$}
\makeatother
\begin{proof}%
\textbf{Completion of the proof of Theorem~\ref{thm6}.} For an increasing sequence $\Lambda_k\to\infty$, let
$(u_k,v_k)\in\mathcal N_{\Lambda_k}$ be ground states of
$J_{\Lambda_k}$, so that $J_{\Lambda_k}(u_k,v_k)=d_{\Lambda_k}$ and
$J_{\Lambda_k}'(u_k,v_k)=0$.
By Lemma~\ref{lem:energy-lower} the sequence is bounded in
norm. Fix $\Lambda_*>0$. For all sufficiently large $k$,
$\Lambda_k\ge\Lambda_*$, so the sequence is also bounded
in the fixed space
$H_{\Lambda_*,a}\times H_{\Lambda_*,b}$. Passing to a subsequence,
$(u_k,v_k)\rightharpoonup (u,v)$ weakly in in $H_{\Lambda_*,a}\times H_{\Lambda_*,b}$ and
pointwise. From Lemma~\ref{lem:energy-conv} we already know
$\lim\limits_{k\to\infty}d_{\Lambda_k}=d_\Omega$, and we prove
$u=0$ on $\Omega_a^c$, $v=0$ on $\Omega_b^c$, so that
$(u,v)\in\mathcal H_\Omega$.

It remains to show that $(u,v)$ is a critical point of $J_\Omega$.
For any $\phi\in C_c(\Omega_a)$ and $\psi\in C_c(\Omega_b)$,
testing $J_{\Lambda_k}'(u_k,v_k)=0$ with $(\phi,\psi)$ yields
\begingroup
\allowdisplaybreaks
\begin{align*}
    J_{\Lambda_k}'(u,v)\cdot(\xi,\eta)
    &=\int_{V}\Bigl(|\nabla u|^{p-2}\nabla u\nabla\xi+|\nabla v|^{p-2}\nabla v\nabla\eta\Bigr)\,d\mu \\
    &\quad+\int_{V}\Bigl((1+\Lambda_k a(x))|u|^{p-2}u\xi+(1+\Lambda_k b(x))|v|^{p-2}v\eta\Bigr)\,d\mu \\
    &\quad-\frac{p-2}{p}\int_{V}|u|^{p-4}u\xi v^{2}\log v^{2}\,d\mu
      -\frac{2}{p}\int_{V}|u|^{p-2}\eta(v\log v^{2}+v)\,d\mu \\
    &\quad-\frac{p-2}{p}\int_{V}|v|^{p-4}v\eta u^{2}\log u^{2}\,d\mu
      -\frac{2}{p}\int_{V}|v|^{p-2}\xi(u\log u^{2}+u)\,d\mu .
\end{align*}
\endgroup
Since $a(x)=0$ on $\Omega_a$, $b(x)=0$ on $\Omega_b$, the potential term involving
$\Lambda_k$ vanishes identically.  The weak convergence of
$(u_k,v_k)$ in $H_{\Lambda_*,a}\times H_{\Lambda_*,b}$, the strong convergence in $L^q$ for
$q\ge p/2$, and the uniform $L^\infty$ bound together justify
passing to the limit in all terms.  The logarithmic terms converge
by the Vitali dominated convergence theorem discuss as before, we immediately obtain
\[
\langle J_\Omega'(u,v),(\phi,\psi)\rangle =0\qquad
\forall\,(\phi,\psi)\in C_c(\Omega_a)\times C_c(\Omega_b),
\]
so $(u,v)$ is a critical point of $J_\Omega$.

Using the energy identity,
\[
J_{\Lambda_k}(u_k,v_k)
 = \frac{2}{p^2}\int_V\bigl(|u_k|^{p-2}v_k^2+|v_k|^{p-2}u_k^2\bigr)\,d\mu .
\]
By Lemma~\ref{lem:energy-conv}, $d_{\Lambda_k}\to d_\Omega$ as \(k\to\infty\).
Passing to the limit via the strong $L^q$ convergence and the Vitali
dominated convergence theorem, justified as above, we obtain
\[
d_\Omega = \lim_{k\to\infty}J_{\Lambda_k}(u_k,v_k)
         = \frac{2}{p^2}\int_{\Omega_a\cup\Omega_b}
           \bigl(|u|^{p-2}v^2+|v|^{p-2}u^2\bigr)\,d\mu
         = J_\Omega(u,v) .
\]
Thus $(u,v)$ is a ground state of the limiting Dirichlet problem~
\eqref{eq:system5}.

The strong convergence of gradients follows from the standard
monotonicity argument as in Lemma~\ref{lem:cerami}.  Indeed, from
$\langle J_{\Lambda_k}'(u_k,v_k),(u_k-u,0)\rangle\to 0$, and the
convergence of all logarithmic and lower-order terms already
established, we obtain
\begin{equation}
\int_V \bigl(|\nabla u_k|^{p-2}\nabla u_k - |\nabla u|^{p-2}\nabla u\bigr)\cdot(\nabla u_k-\nabla u)\,d\mu
+ \int_V (1+\Lambda_k a)\bigl(|u_k|^{p-2}u_k - |u|^{p-2}u\bigr)(u_k-u)\,d\mu \to 0.
\end{equation}
The inequality
$|\xi-\eta|^p\le c(|\xi|^{p-2}\xi-|\eta|^{p-2}\eta)(\xi-\eta)$
then yields $\|\nabla(u_k-u)\|_{L^p(V)}\to0$ and
$\|u_k-u\|_{L^p(V)}\to0$, i.e.,\[
\|u_k-u\|_{H_{\Lambda_k,a}}\to0.
\]
The symmetric argument gives
$\|v_k-v\|_{H_{\Lambda_k,b}}\to0$. This completes the proof of
Theorem~\ref{thm6}.
\end{proof}
\endgroup

\begin{corollary}\label{cor:doublelimit}
If the parameter $\Lambda$ in the second equation of system ~\eqref{eq:system4} and ~\eqref{eq:system5} is replaced by an independent parameter $\mu$, the same arguments demonstrate that the ground state solution of ~\eqref{eq:system4} converges to that of ~\eqref{eq:system5} as both $\Lambda\to\infty$ and $\mu\to\infty$.
\end{corollary}

\begin{theorem}
Let \(V\) be a locally finite graph, \(p>4\), and assume that the potentials \(a,b\) satisfy \((A_1)\) and \((A_2')\).
Let \((u_\Omega,v_\Omega)\) be the ground state of the limiting Dirichlet problem, and suppose that it is non-degenerate in the sense of Definition~\ref{def:5.1}.
Then there exist $\delta>0$ and a neighbourhood
$\mathcal U\subset\mathcal H_\Omega$ of $(u_\Omega,v_\Omega)$
such that, whenever
$\|a-a_\Omega\|_\infty<\delta$ and
$\|b-b_\Omega\|_\infty<\delta$, the corresponding
Dirichlet functional has a unique critical point
$(u_{a,b},v_{a,b})$ in $\mathcal U$.
Moreover, the map $(a,b)\mapsto(u_{a,b},v_{a,b})$
is $C^1$ and takes the value $(u_\Omega,v_\Omega)$
at $(a_\Omega,b_\Omega)$.
\end{theorem}

\begin{proof}
We work on the fixed finite-dimensional space
$\mathcal H_\Omega$.
At the reference ground state, the Hessian is positive
definite on the Nehari tangent space, while the radial
direction is Hessian-orthogonal to that space and satisfies
\[
J_\Omega''(u_\Omega,v_\Omega)
[(u_\Omega,v_\Omega),(u_\Omega,v_\Omega)]
=-2M(u_\Omega,v_\Omega)<0.
\]
Hence the full Hessian on $\mathcal H_\Omega$ is invertible.
The implicit function theorem applied to the critical-point
equation therefore gives the asserted locally unique
$C^1$ branch.
\end{proof}
\subsection*{Proof of Theorem~\ref{thm7}}
\begin{proof}
The argument is presented for \(u_\Lambda\); the estimate for \(v_\Lambda\) is identical.
From the energy expression
\[
d_\Lambda = J_\Lambda(u_\Lambda,v_\Lambda)=\frac1p\bigl(S_\Lambda-L_\Lambda\bigr),
\]
where
\[
S_\Lambda=\|u_\Lambda\|_{\mathcal H_{\Lambda}}^p+\|v_\Lambda\|_{\mathcal H_{\Lambda}}^p,\qquad
L_\Lambda=\int_V\bigl(|u_\Lambda|^{p-2}v_\Lambda^2\log v_\Lambda^2+|v_\Lambda|^{p-2}u_\Lambda^2\log u_\Lambda^2\bigr)d\mu,
\]
we expand \(S_\Lambda\) explicitly:
\[
\begin{aligned}
S_\Lambda&=\int_V\bigl(|\nabla u_\Lambda|^p+|\nabla v_\Lambda|^p\bigr)d\mu\\
&\quad+\int_{\Omega_a}(1+\Lambda a)|u_\Lambda|^p\,d\mu+\int_{\Omega_b}(1+\Lambda b)|v_\Lambda|^p\,d\mu\\
&\quad+\int_{\Omega_a^c}(1+\Lambda a)|u_\Lambda|^p\,d\mu+\int_{\Omega_b^c}(1+\Lambda b)|v_\Lambda|^p\,d\mu .
\end{aligned}
\]
Hence
\[
\begin{aligned}
S_\Lambda&=\int_V\bigl(|\nabla u_\Lambda|^p+|\nabla v_\Lambda|^p\bigr)d\mu
        +\int_{\Omega_a} |u_\Lambda|^p d\mu+\int_{\Omega_b} |v_\Lambda|^p d\mu\\
        &\quad+\int_{\Omega_a^c}(1+\Lambda a)|u_\Lambda|^p\,d\mu
        +\int_{\Omega_b^c}(1+\Lambda b)|v_\Lambda|^p\,d\mu,
\end{aligned}
\]
Insert this into \(p d_\Lambda = S_\Lambda - L_\Lambda\) and bring the \(\Lambda\)-weighted terms to the left-hand side:
\begin{equation}\label{eq:weighted}
\Lambda\!\int_{\Omega_a^c}a|u_\Lambda|^p\,d\mu+\Lambda\!\int_{\Omega_b^c}b|v_\Lambda|^p\,d\mu
= p d_\Lambda + L_\Lambda - \mathcal R_\Lambda,
\end{equation}
where
\[
\mathcal R_\Lambda=\int_V\bigl(|\nabla u_\Lambda|^p+|\nabla v_\Lambda|^p\bigr)d\mu
               +\int_V\bigl(|u_\Lambda|^p+|v_\Lambda|^p\bigr)d\mu \ge 0 .
\]
Since \(\mathcal R_\Lambda\ge0\), we obtain the fundamental inequality
\begin{equation}\label{eq:ineq}
\Lambda\int_{\Omega_a^c}a|u_\Lambda|^p\,d\mu \le p d_\Lambda + L_\Lambda.
\end{equation}
By direct computation, we obtain the universal relation \(M_\Lambda = \frac{p^2}{2} d_\Lambda\).
By the qualitative convergence from Theorem~\ref{thm6}, \(d_\Lambda\to d_\Omega\) as \(\Lambda\to\infty\), so there exists \(\Lambda_0>0\) and a constant \(M_0>0\) such that
\[
M_\Lambda \le M_0 \qquad\text{for all }\Lambda\ge\Lambda_0 .
\]

We split the logarithmic integrals according to whether the argument exceeds one.
For the first term in \(L_\Lambda\):
\[
\begin{aligned}
\int_V|u_\Lambda|^{p-2}v_\Lambda^2\log v_\Lambda^2\,d\mu
&\le \int_{\{v_\Lambda\ge1\}}|u_\Lambda|^{p-2}v_\Lambda^2\log v_\Lambda^2\,d\mu\\
&\quad\le C_{\gamma} \int_V |u_{\Lambda}|^{p-2} |v_{\Lambda}|^{2+\gamma}d\mu\\
&\quad\le C_{\gamma} \|v_{\Lambda}\|_{\infty}^{\gamma} \int_V |u_{\Lambda}|^{p-2} |v_{\Lambda}|^2 d\mu\le C_{\gamma} \|v_{\Lambda}\|_{\infty}^{\gamma} M_{\Lambda}.
\end{aligned}
\]
\(\gamma>0\) is to be chosen later, and the second term in \(L_{\Lambda}\) can be similarly obtained. Insert these together with the symmetric term into the expression of \(L\), we obtain:
\[
L_\Lambda\le C_1\bigl(\|u_\Lambda\|_{\mathcal H_{\Lambda}}^\gamma
                     +\|v_\Lambda\|_{\mathcal H_{\Lambda}}^\gamma\bigr)M_\Lambda ,
\]
Now recall the Nehari identity: \(S_\Lambda = L_\Lambda + \frac2p M_\Lambda\).
Therefore
\[
\begin{aligned}
S_\Lambda 
\le C_1\bigl(\|u_\Lambda\|_{\mathcal H_{\Lambda}}^\gamma
          +\|v_\Lambda\|_{\mathcal H_{\Lambda}}^\gamma\bigr)M_\Lambda + C_2 M_\Lambda,
\end{aligned}
\]
Since \(\|u_\Lambda\|_{\mathcal H_{\Lambda}}^p \le S_\Lambda\) and similarly for \(v_\Lambda\), we have
\[
\|u_\Lambda\|_{\mathcal H_{\Lambda}}^\gamma \le S_\Lambda^{\gamma/p},\qquad
\|v_\Lambda\|_{\mathcal H_{\Lambda}}^\gamma \le S_\Lambda^{\gamma/p}.
\]
Thus
\[
S_\Lambda \le 2 C_1 M_\Lambda S_\Lambda^{\gamma/p} + C_2 M_\Lambda .
\]
Since \(M_\Lambda\le M_0\) for all \(\Lambda\ge\lambda_0\) and we can control \(\gamma/p<1\), there exists a constant \(S_0>0\) independent of \(\Lambda\) such that
\[
S_\Lambda \le S_0 \qquad\text{for all }\Lambda\ge\Lambda_0 .
\]
This immediately implies the uniform bounds
\[
\|u_\Lambda\|_\infty
\le \mu_{\min}^{-1/p}\|u_\Lambda\|_{H_{\Lambda,a}}
\le \mu_{\min}^{-1/p}S_0^{1/p},
\qquad
\|v_\Lambda\|_\infty
\le \mu_{\min}^{-1/p}\|v_\Lambda\|_{H_{\Lambda,b}}
\le \mu_{\min}^{-1/p}S_0^{1/p},
\]
from which we obtain
\[
 L_\Lambda \le C_0,
\]
where \(C_0 = (C_5+C_2)M_0\) is independent of \(\Lambda\). Insert the uniform bounds for \(d_\Lambda\) and \(|L_\Lambda|\) into~\eqref{eq:ineq}:
\[
\Lambda\int_{\Omega_a^c}a|u_\Lambda|^p\,d\mu \le p d_\Lambda + C_0 \le C_1 .
\]
where we use the fact that \(d_{\Lambda}\) is uniformly bounded for \(\Lambda\ge \Lambda_0\). On the other hand, by hypothesis, \(a(x)\ge a_0>0\) on \(\Omega_a^c\). Hence
\[
\Lambda a_0\int_{\Omega_a^c}|u_\Lambda|^p\,d\mu \le C_1,
\qquad\text{so}\qquad
\int_{\Omega_a^c}|u_\Lambda|^p\,d\mu \le \frac{C_1}{a_0}\,\Lambda^{-1}.
\]
The identical argument for \(v_\Lambda\) gives
\[
\int_{\Omega_b^c}|v_\Lambda|^p\,d\mu \le \frac{C_1}{b_0}\,\Lambda^{-1}.
\]
Setting \(C = C_1\max\{a_0^{-1},b_0^{-1}\}\) completes the proof of Theorem~\ref{thm7}.
\end{proof}

\section{Regularised limit theory and compactness recovery  \texorpdfstring{($\mathbb R^N$)}{(R\^{}N)}}\label{sec:7}

In the continuous setting \(\mathbb{R}^N\), the passage to the limit in the logarithmic terms is no longer automatic. In this section we establish, in the continuous critical setting, the strong convergence of regularised ground states as \(\varepsilon\to0\), thereby recovering compactness. The central point of the proof is a decomposition of the logarithmic coupling into its positive and negative parts, which allows a uniform global control of the nonlinear terms. This decomposition plays a decisive role in both the identification of the weak limit and the subsequent minimality argument.

For a fixed \(\varepsilon>0\), we decompose the logarithmic term into its positive and negative parts. Precisely, define
\[
\begin{aligned}
L_\varepsilon^+(u,v)
&:=
\int_{\{\log(v^2+\varepsilon)\ge 0\}} |u|^{p-2}v^2 \log(v^2+\varepsilon)\,dx
+\int_{\{\log(u^2+\varepsilon)\ge 0\}} |v|^{p-2}u^2 \log(u^2+\varepsilon)\,dx,\\[2mm]
L_\varepsilon^-(u,v)
&:=
-\int_{\{\log(v^2+\varepsilon)< 0\}} |u|^{p-2}v^2 \log(v^2+\varepsilon)\,dx
-\int_{\{\log(u^2+\varepsilon)< 0\}} |v|^{p-2}u^2 \log(u^2+\varepsilon)\,dx.
\end{aligned}
\]
Both \(L_\varepsilon^+\) and \(L_\varepsilon^-\) are nonnegative, and
\[
L_\varepsilon(u,v)=L_\varepsilon^+(u,v)-L_\varepsilon^-(u,v).
\]
For the sequence \((u_n,v_n)\), we write
\[
L_n^\pm:=L_{\varepsilon_n}^\pm(u_n,v_n),
\]
and for the limit \((u_0,v_0)\), we write
\[
L_0^\pm:=L_0^\pm(u_0,v_0),
\]
where \(L_0^\pm\) is defined by the same formulae with \(\varepsilon=0\), that is,
\[
\begin{aligned}
L_0^+(u,v)
&:=
\int_{\{\log v^2\ge 0\}} |u|^{p-2}v^2 \log v^2\,dx
+\int_{\{\log u^2\ge 0\}} |v|^{p-2}u^2 \log u^2\,dx,\\[2mm]
L_0^-(u,v)
&:=
-\int_{\{\log v^2< 0\}} |u|^{p-2}v^2 \log v^2\,dx
-\int_{\{\log u^2< 0\}} |v|^{p-2}u^2 \log u^2\,dx.
\end{aligned}
\]
With these conventions, all the logarithmic quantities appearing in the subsequent arguments are well defined, nonnegative, and depend on the regularisation parameter in a coherent way. Since the original functional may fail to be Fréchet differentiable, we adopt the following variational definition. Let \(S\), \(L\), \(M\) be defined similar as Section~\ref{sec:stability}. A nontrivial pair \((u,v)\) is a ground state if
\[
S(u,v)=L(u,v)+\frac{2}{p}M(u,v)
\]
and it minimizes \(J=\frac{1}{p}(S-L)\) among all such pairs.

\begin{lemma}\label{lem7.1}
Under the same hypotheses as theorem~\ref{thm2}, there exists a constant \(c_0>0\) such that for every \(\varepsilon\in(0,1/2)\),
\[
d_\varepsilon:=\inf_{(u,v)\in\mathcal N_\varepsilon}J_\varepsilon(u,v)\ge c_0,
\]
where \(J_\varepsilon\) is defined as in Definition~\ref{def:Jeps-p4} and \(c_0\) is independent of \(\varepsilon\).
\end{lemma}

\begin{proof}
Set
\[
S(u,v):=\|u\|_{\mathcal {G}_a}^p+\|v\|_{\mathcal {G}_b}^p
=\int_{\mathbb{R}^N}\bigl(|\nabla u|^p+a|u|^p+|\nabla v|^p+b|v|^p\bigr)\,dx,
\]
\[
L_{\varepsilon}(u,v):=\int_{\mathbb{R}^N}\bigl(|u|^{p-2}v^2\log(v^2+\varepsilon)
+|v|^{p-2}u^2\log(u^2+\varepsilon)\bigr)\,dx,
\]
and
\[
R_{\varepsilon}(u,v):=\int_{\mathbb{R}^N}\left(|u|^{p-2}\frac{v^4}{v^2+\varepsilon}
+|v|^{p-2}\frac{u^4}{u^2+\varepsilon}\right)\,dx.
\]
Then
\[
J_\varepsilon(u,v)=\frac1p S(u,v)-\frac1p L_{\varepsilon}(u,v).
\]
Fix \(\sigma\in(0,(p-2)/4)\), and let \(\|(u,v)\|_{\mathcal G}=r\). Decompose
\[
L_\varepsilon=L_\varepsilon^+-L_\varepsilon^-,
\qquad L_\varepsilon^\pm\ge0.
\]
Thus
\[
J_\varepsilon(u,v)\ge \frac1p S(u,v)-\frac1p L_\varepsilon^+(u,v).
\]
On the positive part \(\{v^2+\varepsilon\ge1\}\), since \(\varepsilon<1/2\), we have \(v\ge1/\sqrt2\), and hence
\[
\log(v^2+\varepsilon)\le C_\sigma v^{2\sigma}.
\]
The same estimate holds for \(u\). By Young's inequality and the embeddings
\[
\mathcal G_a,\mathcal G_b\hookrightarrow L^{p+2\sigma},
\]
we obtain
\[
L_\varepsilon^+(u,v)\le C\|(u,v)\|_{\mathcal G}^{p+2\sigma}=C r^{p+2\sigma}.
\]
Therefore
\[
J_\varepsilon(u,v)\ge \frac{2^{1-p}}p r^p-Cr^{p+2\sigma}.
=r^p\Bigl(\frac{2^{1-p}}p-C r^{2\sigma}\Bigr).
\]
Choose \(r>0\) so small that \(C r^{2\sigma}\le\frac{2^{-p}}{p}\). Then, for every \((u,v)\) with \(\|(u,v)\|_{\mathcal G}=r\),
\[
J_\varepsilon(u,v)\ge \frac{2^{-p}}{p}r^p=:\alpha>0,
\]
where \(\alpha\) is independent of \(\varepsilon\).
Moreover, the same estimate shows that \(J_\varepsilon(u,v)>0\) whenever \(0<\|(u,v)\|_{\mathcal G}<r\). Hence \(J_\varepsilon\) is nonnegative in the closed ball \(\{\|(u,v)\|_{\mathcal G}\le r\}\) and strictly positive away from the origin.

Choose a nonzero compactly supported pair \((\phi,\psi)\) such that
\[
M(\phi,\psi):=\int_{\mathbb{R}^N}\bigl(|\phi|^{p-2}\psi^2+|\psi|^{p-2}\phi^2\bigr)\,dx>0.
\]
For \(t>0\), using the above expression of \(J_\varepsilon\),
\[
J_\varepsilon(t\phi,t\psi)
=\frac{t^p}{p}S(\phi,\psi)
-\frac{t^p\log t^2}{p}M(\phi,\psi)
-\frac{t^p}{p}L_{\varepsilon/t^2}(\phi,\psi),
\]
where
\[
L_{\varepsilon/t^2}(\phi,\psi)
:=\int_{\mathbb{R}^N}\Bigl(|\phi|^{p-2}\psi^2\log\Bigl(\psi^2+\frac{\varepsilon}{t^2}\Bigr)
+|\psi|^{p-2}\phi^2\log\Bigl(\phi^2+\frac{\varepsilon}{t^2}\Bigr)\Bigr)\,dx.
\]
Since \(\phi,\psi\) have compact support and \(\varepsilon\in(0,1/2)\), the quantity \(L_{\varepsilon/t^2}(\phi,\psi)\) is bounded uniformly in \(\varepsilon\). Thus there exists \(t_0>0\), independent of \(\varepsilon\), such that
\[
J_\varepsilon(t_0\phi,t_0\psi)<0
\]
and \(\|t_0(\phi,\psi)\|_{\mathcal G}>r\).

We now turn to the fibre map. By Lemma~\ref{lem4.6}, for any nonzero pair \((u,v)\in\mathcal G\) with \(M(u,v)>0\), the fibre map \(F(t):=J_\varepsilon(tu,tv)\) has a global maximum point. We include the proof of uniqueness here. A direct computation yields
\[
\begin{aligned}
\frac{F'(t)}{t^{p-1}}
&=S(u,v)-L_{\varepsilon/t^2}(u,v)-\log t^2\,M(u,v)
-\frac{2}{p}M(u,v)\\
&\quad+\frac{2\varepsilon}{p}\int_{\mathbb{R}^N}\Bigl(\frac{|u|^{p-2}v^2}{t^2v^2+\varepsilon}
+\frac{|v|^{p-2}u^2}{t^2u^2+\varepsilon}\Bigr)\,dx,
\end{aligned}
\]
where
\[
L_{\varepsilon/t^2}(u,v):=\int_{\mathbb{R}^N}\Bigl(|u|^{p-2}v^2\log\Bigl(v^2+\frac{\varepsilon}{t^2}\Bigr)
+|v|^{p-2}u^2\log\Bigl(u^2+\frac{\varepsilon}{t^2}\Bigr)\Bigr)\,dx.
\]
Let \(g(t)\) denote the right-hand side. Differentiating, we obtain
\[
\begin{aligned}
g'(t)
&=-\int_{\mathbb{R}^N}\Bigl(\frac{2t |u|^{p-2}v^4}{t^2v^2+\varepsilon}
+\frac{2t |v|^{p-2}u^4}{t^2u^2+\varepsilon}\Bigr)\,dx\\
&\quad-\frac{2\varepsilon}{p}\int_{\mathbb{R}^N}\Bigl(\frac{2t |u|^{p-2}v^4}{(t^2v^2+\varepsilon)^2}
+\frac{2t |v|^{p-2}u^4}{(t^2u^2+\varepsilon)^2}\Bigr)\,dx.
\end{aligned}
\]
Both terms are strictly negative due to \(M(u,v)>0\). Therefore
\[
g'(t)<0\qquad\text{for all }t>0.
\]
Thus \(g(t)\) is strictly decreasing.
Combining the fact from Lemma~\ref{lem4.6} that \(F(t)\) has a global maximum point, we deduce there exists a unique $t_\varepsilon>0$ such that
$F'(t_\varepsilon)=0$, i.e., $g(t_\varepsilon)=0$, and \(t_\varepsilon\) is the unique global maximum point.
In particular, if \((u,v)\in\mathcal N_\varepsilon\), then \(F'(1)=0\), and by uniqueness \(t_\varepsilon=1\). Thus
\[
J_\varepsilon(u,v)=\max_{t>0}J_\varepsilon(tu,tv).
\]

Define
\[
\Gamma_\varepsilon:=\Bigl\{\gamma\in C([0,1],\mathcal G):\gamma(0)=0,\ J_\varepsilon(\gamma(1))<0\Bigr\},
\]
and the mountain pass level
\[
c_\varepsilon:=\inf_{\gamma\in\Gamma_\varepsilon}\max_{t\in[0,1]}J_\varepsilon(\gamma(t)).
\]

First, we prove \(c_\varepsilon\ge\alpha\). Take any \(\gamma\in\Gamma_\varepsilon\). Since the energy is nonnegative in the ball and the endpoint has negative energy, there exists \(s\in(0,1)\) such that \(\|\gamma(s)\|_{\mathcal G}=r\). By the previous estimate,
\[
\max_{t\in[0,1]}J_\varepsilon(\gamma(t))\ge J_\varepsilon(\gamma(s))\ge\alpha.
\]
Thus \(c_\varepsilon\ge\alpha\).

Next, we prove \(c_\varepsilon\le d_\varepsilon\). Take any \((u,v)\in\mathcal N_\varepsilon\). Choose \(t_0\) large enough that \(J_\varepsilon(t_0u,t_0v)<0\), and define \(\gamma(s)=s t_0(u,v)\), \(s\in[0,1]\). Then \(\gamma\in\Gamma_\varepsilon\), and the maximum is attained at \(s=1/t_0\), that is, at \(t=1\). Hence
\[
\max_{s\in[0,1]}J_\varepsilon(\gamma(s))=J_\varepsilon(u,v).
\]
Therefore \(c_\varepsilon\le J_\varepsilon(u,v)\). Taking the infimum over \((u,v)\in\mathcal N_\varepsilon\), we obtain \(c_\varepsilon\le d_\varepsilon\).

Combining the two inequalities gives
\[
d_\varepsilon\ge c_\varepsilon\ge\alpha=:c_0>0.
\]
This completes the proof.
\end{proof}

\begin{lemma}\label{7.2}
    Set \(F_{\varepsilon}(t)=J_{\varepsilon}(t(\phi,\psi))\), where \((\phi,\psi)\in C_c^{\infty}(\mathbb{R}^N)\times C_c^{\infty}(\mathbb{R}^N)\). If there exists a sequence \(\varepsilon_n\to 0\) with \(t_{\varepsilon_n}\to 0\), then there holds \(F_{\varepsilon_n}(t_{\varepsilon_n})\to 0\).
\end{lemma}
\begin{proof}
    We give the specific expression as below
    \begin{align}
F_{\varepsilon_n}(t_{\varepsilon_n}) 
&= \frac{t_{\varepsilon_n}^p}{p}\left(\|\phi\|_{\mathcal{G}_a}^p+\|\psi\|_{\mathcal{G}_b}^p\right) 
-\frac{1}{p}\int_{\mathbb{R}^N} \Bigl[ |t_{\varepsilon_n}\phi|^{p-2}|t_{\varepsilon_n}\psi|^2\log(|t_{\varepsilon_n}\psi|^2+\varepsilon_n) \nonumber \\
&\qquad + |t_{\varepsilon_n}\psi|^{p-2}|t_{\varepsilon_n}\phi|^2\log(|t_{\varepsilon_n}\phi|^2+\varepsilon_n) \Bigr]\,dx
\end{align}
Firstly, it is obvious that \[\lim\limits_{n\to\infty} \frac{t_{\varepsilon_n}^p}{p}\left(\|\phi\|_{\mathcal{G}_a}^p+\|\psi\|_{\mathcal{G}_b}^p\right)=0.\]
On the other hand, if \(|t_{\varepsilon_n}\psi|=0\), we are done. It remains to consider \(|t_{\varepsilon_n}\psi|\neq 0\),
\begin{eqnarray*}
 |t_{\varepsilon_n}\phi|^{p-2}|t_{\varepsilon_n}\psi|^2\left|\log(|t_{\varepsilon_n}\psi|^2+\varepsilon_n)\right|&\le& |t_{\varepsilon_n}\phi|^{p-2}|t_{\varepsilon_n}\psi|^2 \left[\left|\log(|t_{\varepsilon_n}\psi|^2)\right|+\log\left(1+\frac{\varepsilon_n}{|t_{\varepsilon_n}\psi|^2}\right)\right]\\
 &\le& |t_{\varepsilon_n}\phi|^{p-2}|t_{\varepsilon_n}\psi|^2 \left|\log(|t_{\varepsilon_n}\psi|^2)\right|+ \varepsilon_n |t_{\varepsilon_n}\phi|^{p-2},
\end{eqnarray*}
which implies that \[\lim\limits_{n\to\infty}|t_{\varepsilon_n}\phi|^{p-2}|t_{\varepsilon_n}\psi|^2\left|\log(|t_{\varepsilon_n}\psi|^2+\varepsilon_n)\right|=0.\]
Analogously, \[\lim\limits_{n\to\infty}|t_{\varepsilon_n}\psi|^{p-2}|t_{\varepsilon_n}\phi|^2\left|\log(|t_{\varepsilon_n}\phi|^2+\varepsilon_n)\right|=0.\]
By the fact that \((\phi,\psi)\in C_c(\mathbb{R}^N)\times C_c(\mathbb{R}^N)\), the conclusion follows and complete the proof.
\end{proof}

\begin{lemma}\label{lem:regularised-boundedness}
Under the hypotheses of Theorem~\ref{thm2}, let \((u_\varepsilon,v_\varepsilon)\) be a ground state of the regularised functional \(J_\varepsilon\) on the Nehari manifold \(\mathcal N_\varepsilon\) for each \(\varepsilon\in(0,1/2)\). Then there exists a constant \(C>0\) such that
\[
\|(u_\varepsilon,v_\varepsilon)\|_{\mathcal G}\le C
\qquad\text{for all }\varepsilon\in(0,1/2).
\]
\end{lemma}

\begin{proof}
Choose a nonnegative pair \((\phi,\psi)\in C_c^\infty(\mathbb R^N)\times C_c^\infty(\mathbb R^N)\) with \((\phi,\psi)\not\equiv(0,0)\). For instance, take \(\phi=\psi\) with \(\phi\not\equiv0\). Then
\[
M(\phi,\psi):=\int_{\mathbb R^N}\bigl(|\phi|^{p-2}\psi^2+|\psi|^{p-2}\phi^2\bigr)\,dx>0.
\]
For each \(\varepsilon\in(0,1/2)\), the regularised fibre map
\[
F_\varepsilon(t):=J_\varepsilon(t(\phi,\psi))
\]
has a unique global maximum point \(t_\varepsilon>0\), and \(t_\varepsilon(\phi,\psi)\in\mathcal N_\varepsilon\).

We first show that \(\{t_\varepsilon\}\) is contained in a fixed compact interval \([a,b]\subset(0,\infty)\). As \(t\to+\infty\), the logarithmic term \(-\log t^2\,M(\phi,\psi)\) dominates uniformly in \(\varepsilon\), because \(M(\phi,\psi)>0\) and the remaining terms have compact support. Hence \(F_\varepsilon(t)\to-\infty\) uniformly in \(\varepsilon\), which gives a uniform upper bound \(t_\varepsilon\le b\).

Suppose now that there is a sequence \(\varepsilon_n\to0\) with \(t_{\varepsilon_n}\to0\). Since \(\phi,\psi\) have compact support, it follows from Lemma~\ref{7.2} that
\[
F_{\varepsilon_n}(t_{\varepsilon_n})\to 0.
\]
But \(t_{\varepsilon_n}(\phi,\psi)\in\mathcal N_{\varepsilon_n}\), so
\[
d_{\varepsilon_n}\le F_{\varepsilon_n}(t_{\varepsilon_n}).
\]
On the other hand, by Lemma~\ref{lem7.1} there exists \(c_0>0\) such that
\[
d_{\varepsilon_n}\ge c_0
\]
for all \(n\), a contradiction. Therefore \(t_\varepsilon\ge a>0\) for some positive \(a\) independent of \(\varepsilon\).
Thus \(t_\varepsilon\in[a,b]\) for every \(\varepsilon\in(0,1/2)\). Since \((\phi,\psi)\) is compactly supported and \(\varepsilon\) ranges in a bounded interval, we have
\[
\sup_{\varepsilon\in(0,1/2)}F_\varepsilon(t_\varepsilon)\le C.
\]
Hence
\[
d_\varepsilon=J_\varepsilon(u_\varepsilon,v_\varepsilon)
\le J_\varepsilon(t_\varepsilon(\phi,\psi))
=F_\varepsilon(t_\varepsilon)
\le C
\]
uniformly in \(\varepsilon\).

Now we set
\[
X_\varepsilon:=\|u_\varepsilon\|_{\mathcal G_a}^p+\|v_\varepsilon\|_{\mathcal G_b}^p.
\]
Since \((u_\varepsilon,v_\varepsilon)\in\mathcal N_\varepsilon\), the regularised Nehari identity gives
\[
X_\varepsilon
= L_\varepsilon^+(u_\varepsilon,v_\varepsilon)-L_\varepsilon^-(u_\varepsilon,v_\varepsilon)
+\frac{2}{p}R_\varepsilon(u_\varepsilon,v_\varepsilon),
\]
where
\[
R_\varepsilon(u,v):=\int_{\mathbb R^N}
\Bigl(|u|^{p-2}\frac{v^4}{v^2+\varepsilon}
+|v|^{p-2}\frac{u^4}{u^2+\varepsilon}\Bigr)\,dx.
\]
The regularised energy identity gives
\[
d_\varepsilon=\frac{2}{p^2}R_\varepsilon(u_\varepsilon,v_\varepsilon),
\]
so the uniform upper bound on \(d_\varepsilon\) implies
\[
R_\varepsilon(u_\varepsilon,v_\varepsilon)\le C
\]
uniformly in \(\varepsilon\).
Since \(L_\varepsilon^-\ge0\), we obtain
\[
X_\varepsilon
\le L_\varepsilon^+(u_\varepsilon,v_\varepsilon)+C.
\]
By the exponent calibration argument in the proof of
Lemma~\ref{lem:coercivity} and the uniform bound on
$R_\varepsilon(u_\varepsilon,v_\varepsilon)$ established above,
there exist constants $C>0$ and $\sigma\in(0,1)$,
independent of $\varepsilon$, such that the regularized
ground states satisfy we get
\[
X_\varepsilon
\le C X_\varepsilon^{\sigma}+C,
\]
which yields \(X_\varepsilon\le C\). Thus
\[
\|(u_\varepsilon,v_\varepsilon)\|_{\mathcal G}\le C
\]
uniformly in \(\varepsilon\), completing the proof.
\end{proof}

The following proposition controls the nonlinear
contributions on $S\cap\{|\tilde v|<\delta\}$.
After taking the limit superior as $n\to\infty$,
these contributions are bounded by $C\delta^\gamma$.
This is a quantitative estimate in the amplitude
threshold $\delta$, rather than a convergence rate
in the regularization parameter $\varepsilon_n$.
We state and prove the estimate only for the region
where $\tilde v$ is small; the corresponding estimate
for $\tilde u$ follows by symmetry.

\begin{proposition}\label{prop7.1}
Assume that the potentials \(a,b\) satisfy the hypotheses of Theorem~\ref{thm2}.
Let \(\varepsilon_n\in(0,1/2)\) with \(\varepsilon_n\to0\), and let
\(\{(u_n,v_n)\}\) be a sequence of ground states of \(J_{\varepsilon_n}\).
Then, up to a subsequence,
\[
(u_n,v_n)\to(\tilde u,\tilde v)
\quad\text{a.e. in }\mathbb R^N.
\]
Fix \(\phi,\psi\in C_c(\mathbb R^N)\) with compact support
\(S:=\operatorname{supp}(\phi)\cup\operatorname{supp}(\psi)\).
Then, for every \(0<\gamma\le2/3\) and every \(\delta\in(0,1/2)\),
\[
\begin{aligned}
&\limsup_{n\to\infty}\left|\int_{S\cap\{|\tilde v|<\delta\}}
|u_n|^{p-4}u_n\phi\,v_n^2
\log(v_n^2+\varepsilon_n)\,dx\right|\\
&\quad+\limsup_{n\to\infty}\left|\int_{S\cap\{|\tilde v|<\delta\}}
|u_n|^{p-2}v_n\psi\,
\log(v_n^2+\varepsilon_n)\,dx\right|\\
&\quad+\limsup_{n\to\infty}\left|\int_{S\cap\{|\tilde v|<\delta\}}
|u_n|^{p-2}\frac{v_n^3\psi}{v_n^2+\varepsilon_n}\,dx\right|
\le C\delta^\gamma.
\end{aligned}
\]
The constant \(C\) depends on \(p,\gamma,S\),
\(\sup_n\|(u_n,v_n)\|_{\mathcal G}\),
the embedding constants, and
\(\|\phi\|_{L^\infty},\|\psi\|_{L^\infty}\),
but is independent of \(\delta\).
\end{proposition}

\begin{proof}
Passing to a subsequence, the compact embeddings
\(G_a,G_b\hookrightarrow L^q(\mathbb R^N)\), \(p\le q<\infty\),
and the boundedness of \(\{(u_n,v_n)\}\) in \(\mathcal G\)
give \((u_n,v_n)\to(\tilde u,\tilde v)\) a.e. in \(\mathbb R^N\).
Let \(S_\delta:=S\cap\{|\tilde v|<\delta\}\).
By Egorov's theorem, for every \(\eta>0\) there exists
\(E_\eta\subset S_\delta\) with
\(|S_\delta\setminus E_\eta|<\eta\) such that
\(v_n\to\tilde v\) uniformly on \(E_\eta\).
Hence for large \(n\),
\[
|v_n|\le2\delta
\quad\text{on }E_\eta.
\]
We first estimate the integrals on \(E_\eta\),
and then handle the complement.

Set
\[
\delta_n:=\varepsilon_n^{1/4}.
\]
Decompose
\[
E_\eta=A_n^+\cup A_n^-,
\quad
A_n^+:=E_\eta\cap\{|v_n|\ge\delta_n\},
\quad
A_n^-:=E_\eta\cap\{|v_n|<\delta_n\}.
\]

On \(A_n^+\), we have \(\delta_n\le|v_n|\le2\delta\le1\).
Using
\[
\begin{aligned}
|\log(v_n^2+\varepsilon_n)|
&\le|\log v_n^2|
+\log\left(1+\frac{\varepsilon_n}{v_n^2}\right)\\
&\le|\log v_n^2|+\frac{\varepsilon_n}{v_n^2}
\le|\log v_n^2|+\varepsilon_n^{1/2}
\end{aligned}
\]
and, for \(0<\gamma\le2/3\),
\[
|\log v_n^2|\le C|v_n|^{-\gamma/2},
\]
we obtain
\[
v_n^2|\log(v_n^2+\varepsilon_n)|
\le C\delta^\gamma+C\varepsilon_n^{1/2},
\]
\[
|v_n||\log(v_n^2+\varepsilon_n)|
\le C\delta^\gamma+C\varepsilon_n^{1/2}.
\]
Therefore,
\[
\int_{A_n^+}|u_n|^{p-3}|\phi|v_n^2
|\log(v_n^2+\varepsilon_n)|\,dx
\le C\delta^\gamma+C\varepsilon_n^{1/2},
\]
and
\[
\int_{A_n^+}|u_n|^{p-2}|v_n||\psi|
|\log(v_n^2+\varepsilon_n)|\,dx
\le C\delta^\gamma+C\varepsilon_n^{1/2},
\]
where we used the uniform boundedness of \(\{u_n\}\) in \(G_a\).

On \(A_n^-\), since \(|v_n|<\delta_n\),
\[
v_n^2|\log(v_n^2+\varepsilon_n)|
\le\varepsilon_n^{1/2}(|\log\varepsilon_n|+C),
\]
\[
|v_n||\log(v_n^2+\varepsilon_n)|
\le\varepsilon_n^{1/4}(|\log\varepsilon_n|+C).
\]
As \(\varepsilon_n^\alpha|\log\varepsilon_n|\to0\)
for all \(\alpha>0\), both quantities tend to zero uniformly.
Hence the corresponding integrals on \(A_n^-\) tend to zero.

For the rational term, using
\[
\left|\frac{v_n^3}{v_n^2+\varepsilon_n}\right|\le|v_n|,
\]
we get
\[
\left|\int_{E_\eta}|u_n|^{p-2}
\frac{v_n^3\psi}{v_n^2+\varepsilon_n}\,dx\right|
\le\int_{E_\eta}|u_n|^{p-2}|v_n||\psi|\,dx.
\]
Since \(|v_n|\le2\delta\) on \(E_\eta\), this is bounded
by \(C\delta\), hence by \(C\delta^\gamma\) because \(\gamma\le1\).

It remains to consider \(S_\delta\setminus E_\eta\).
Split this set into three parts according to
\[
|v_n|\ge1,\quad 2\delta<|v_n|<1,\quad |v_n|\le2\delta.
\]
On \(\{|v_n|\le2\delta\}\), we proceed as on \(A_n^+\)
and \(A_n^-\), obtaining
\[
\int_{\{|v_n|\le2\delta\}\cap(S_\delta\setminus E_\eta)}
|u_n|^{p-3}|\phi|v_n^2
|\log(v_n^2+\varepsilon_n)|\,dx
\le C\delta^\gamma+o(1).
\]
On \(\{2\delta<|v_n|<1\}\), the logarithmic term is
bounded by a constant depending only on \(\delta\),
so the integral is controlled by
\[
C_\delta\int_{S_\delta\setminus E_\eta}
|u_n|^{p-3}|\phi|\,dx.
\]
On \(\{|v_n|\ge1\}\), fix \(\tau\in(0,1)\). Then
\[
|\log(v_n^2+\varepsilon_n)|
\le\log(2v_n^2)\le C_\tau|v_n|^\tau,
\]
and
\[
\begin{aligned}
&\int_{\{|v_n|\ge1\}\cap(S_\delta\setminus E_\eta)}
|u_n|^{p-3}|\phi|v_n^2
|\log(v_n^2+\varepsilon_n)|\,dx\\
&\quad\le C_\tau\|\phi\|_{L^\infty}
\|u_n\|_{L^p}^{p-3}\|v_n\|_{L^p}^{2+\tau}
|S_\delta\setminus E_\eta|^{(1-\tau)/p}
\le C_\tau\eta^{(1-\tau)/p}.
\end{aligned}
\]
Moreover, on \(\{|v_n|\le1\}\), the uniform bounds
\[
v_n^2|\log(v_n^2+\varepsilon_n)|
+|v_n||\log(v_n^2+\varepsilon_n)|\le C
\]
and H\"older's inequality bound the two logarithmic
integrals over \(S_\delta\setminus E_\eta\)
by \(C\eta^{3/p}\) and \(C\eta^{2/p}\), respectively.
On \(\{|v_n|\ge1\}\), the second logarithmic integral
is estimated in the same way, using
\[
\int_{S_\delta\setminus E_\eta}
|u_n|^{p-2}|v_n|^{1+\tau}|\psi|\,dx
\le C_\tau\eta^{(1-\tau)/p}.
\]
The rational integral over the complement is bounded by
\[
\int_{S_\delta\setminus E_\eta}
|u_n|^{p-2}|v_n||\psi|\,dx
\le C\eta^{1/p}.
\]
Thus all three contributions from the complement
vanish uniformly in \(n\) as \(\eta\to0\).
Taking the limit superior as \(n\to\infty\) and then
letting \(\eta\to0\) proves the assertion.
\end{proof}

\subsection*{Proof of Theorem~\ref{thm8}}
\begin{proof}
We proceed in several steps.

\medskip
\noindent\textbf{Step 1.}
By Lemma~\ref{lem:regularised-boundedness}, there exists a constant \(C>0\) such that
\[
\|(u_n,v_n)\|_{\mathcal G}\le C\qquad \forall n.
\]
Hence, extracting a subsequence if necessary, there exists \((u_0,v_0)\in\mathcal G\) such that
\[
(u_n,v_n)\rightharpoonup (u_0,v_0)\quad\text{weakly in }\mathcal G.
\]
By the compact embedding of Lemma~\ref{lem2.3}, for every \(q\ge N\),
\[
(u_n,v_n)\to(u_0,v_0)\quad\text{strongly in }L^q(\mathbb R^N)\times L^q(\mathbb R^N).
\]
Passing to a further subsequence, we may also assume
\[
(u_n,v_n)\to(u_0,v_0)\quad\text{a.e. in }\mathbb R^N.
\]
Since by Lemma~\ref{lem7.1} the regularised ground-state energy has a uniform positive lower bound, it follows that \((u_0,v_0)\not\equiv(0,0)\). Otherwise, \((u_n,v_n)\to(0,0)\quad\text{strongly in }L^q(\mathbb R^N)\times L^q(\mathbb R^N)\), which implies that \[R_{\varepsilon_n}(u_n,v_n)\to 0,\] as \(n\to \infty\). Then \[J_{\varepsilon_n}(u_n,v_n)=\frac{2}{p^2}R_{\varepsilon_n}(u_n,v_n)\to 0,\] as \(n\to \infty\), which is a contradiction to Lemma~\ref{lem7.1}.

\medskip
\noindent\textbf{Step 2.}
Fix a compact set $K\subset\mathbb R^N$. We first prove
the local $L^1$ convergence of the logarithmic energy densities.
Consider
\[
h_n:=|u_n|^{p-2}v_n^2\log(v_n^2+\varepsilon_n).
\]
By almost everywhere convergence and $\varepsilon_n\to0$,
\[
h_n\to h_0:=|u_0|^{p-2}v_0^2\log v_0^2
\quad\text{a.e.},
\]
where $h_0=0$ on $\{v_0=0\}$.

On $\{v_n^2+\varepsilon_n\ge1\}$, for any fixed
$\sigma\in(0,(p-4)/2)$,
\[
|h_n|
\le C_\sigma |u_n|^{p-2}|v_n|^{2+2\sigma}
\le C_\sigma\bigl(|u_n|^{p+2\sigma}
                 +|v_n|^{p+2\sigma}\bigr).
\]
The right-hand side is uniformly integrable on $K$
by strong convergence in $L^{p+2\sigma}(K)$.
On $\{v_n^2+\varepsilon_n<1\}$, the uniform bound
\[
\sup_{\substack{0<t<1\\0<\varepsilon<1/2}}
t|\log(t^2+\varepsilon)|<\infty
\]
gives
\[
|h_n|\le C|u_n|^{p-2}|v_n|.
\]
By H\"older's inequality,
\[
\bigl\||u_n|^{p-2}|v_n|\bigr\|_{L^{p/(p-1)}(K)}
\le \|u_n\|_{L^p(K)}^{p-2}\|v_n\|_{L^p(K)}
\le C.
\]
Thus $(h_n)$ is uniformly integrable on $K$, and Vitali's
theorem yields
\[
h_n\to h_0\quad\text{strongly in }L^1(K).
\]
The symmetric energy density is treated in the same way.

We also record the convergence needed in the
Euler--Lagrange equations. For any fixed $\eta\in(0,1)$,
uniformly in $0<\varepsilon<1/2$,
\[
\begin{aligned}
|s\log(s^2+\varepsilon)|
&\le C_\eta\bigl(|s|^{1-\eta}+|s|^{1+\eta}\bigr),\\
|s^2\log(s^2+\varepsilon)|
&\le C_\eta\bigl(|s|^{2-\eta}+|s|^{2+\eta}\bigr),
\qquad
\frac{|s|^3}{s^2+\varepsilon}\le |s|.
\end{aligned}
\]
Consequently, each logarithmic or rational term in the
nonlinear right-hand sides is bounded in absolute value by
\[
C_\eta\bigl(
|u_n|^{p-1-\eta}+|v_n|^{p-1-\eta}
+|u_n|^{p-1+\eta}+|v_n|^{p-1+\eta}
\bigr).
\]
The local bounds in every finite $L^q$ space therefore
give a uniform $L^r(K)$ bound for some $r>p'=p/(p-1)$.
Together with almost everywhere convergence, this implies,
by Vitali's theorem, that these terms converge strongly
in $L^{p'}(K)$ to their unregularized counterparts,
with the continuous extensions at zero understood.

\medskip
\noindent\textbf{Step 3.}
Define
\[
L_n^\pm
:=\int_{\mathbb R^N}\Bigl(
|u_n|^{p-2}v_n^2(\log(v_n^2+\varepsilon_n))_\pm
+|v_n|^{p-2}u_n^2(\log(u_n^2+\varepsilon_n))_\pm
\Bigr)\,dx,
\]
and let $L_0^\pm$ denote the corresponding unregularized
quantities at $(u_0,v_0)$. Recall that
\[
R_{\varepsilon_n}(u_n,v_n)
=\int_{\mathbb R^N}\Bigl(
|u_n|^{p-2}\frac{v_n^4}{v_n^2+\varepsilon_n}
+|v_n|^{p-2}\frac{u_n^4}{u_n^2+\varepsilon_n}
\Bigr)\,dx,
\]
and set
\[
M_0:=\int_{\mathbb R^N}
\bigl(|u_0|^{p-2}v_0^2+|v_0|^{p-2}u_0^2\bigr)\,dx.
\]
We prove that
\[
L_n^+\to L_0^+,
\qquad
R_{\varepsilon_n}(u_n,v_n)\to M_0.
\]

For any fixed $\sigma\in(0,(p-4)/2)$, the estimate
from Step~2 gives
\[
0\le
|u_n|^{p-2}v_n^2(\log(v_n^2+\varepsilon_n))_+
\le C_\sigma\bigl(
|u_n|^{p+2\sigma}+|v_n|^{p+2\sigma}
\bigr),
\]
and similarly for the symmetric term.
Since $u_n\to u_0$ and $v_n\to v_0$ strongly in
$L^{p+2\sigma}(\mathbb R^N)$, the majorants converge
in $L^1(\mathbb R^N)$. They are therefore uniformly
integrable and satisfy
\[
\lim_{R\to\infty}\sup_n
\int_{\mathbb R^N\setminus B_R}
\bigl(|u_n|^{p+2\sigma}+|v_n|^{p+2\sigma}\bigr)\,dx=0.
\]
Together with almost everywhere convergence, Vitali's
theorem yields global $L^1$ convergence of the positive
logarithmic densities, and hence $L_n^+\to L_0^+$.

For the rational term, Young's inequality gives
\[
0\le
|u_n|^{p-2}\frac{v_n^4}{v_n^2+\varepsilon_n}
\le |u_n|^{p-2}v_n^2
\le \frac{p-2}{p}|u_n|^p+\frac{2}{p}|v_n|^p.
\]
The same estimate holds with $u_n$ and $v_n$ interchanged.
Strong convergence in $L^p(\mathbb R^N)$ ensures that
these majorants converge in $L^1(\mathbb R^N)$, providing
both uniform integrability and uniform control of the
tails. Since
\[
|u_n|^{p-2}\frac{v_n^4}{v_n^2+\varepsilon_n}
\longrightarrow |u_0|^{p-2}v_0^2
\quad\text{a.e.},
\]
Vitali's theorem again gives global $L^1$ convergence
of the rational densities. Thus
\[
R_{\varepsilon_n}(u_n,v_n)\to M_0.
\]

\medskip
\noindent\textbf{Step 4.}
We now proceed with the identification of the weak limit. Firstly, since \(|\nabla u_n|^{p-2}\nabla u_n\) and \(|\nabla v_n|^{p-2}\nabla v_n\) have weak limit in \(L^{p'}(\mathbb{R}^N)\) respectively because both of them bounded in \(L^{p'}(\mathbb{R}^N)\).
Passing to the limit in the regularised Euler--Lagrange equation, for every \((\phi,\psi)\in C_c^\infty(\mathbb R^N)\times C_c^\infty(\mathbb R^N)\), we obtain
\[
\begin{aligned}
0={}&\int_{\mathbb R^N}\chi_u\cdot\nabla\phi\,dx+\int_{\mathbb R^N}a|u_0|^{p-2}u_0\phi\,dx\\
&+\int_{\mathbb R^N}\chi_v\cdot\nabla\psi\,dx+\int_{\mathbb R^N}b|v_0|^{p-2}v_0\psi\,dx\\
&-\int_{\mathbb R^N}R_u^{(0)}\phi\,dx-\int_{\mathbb R^N}R_v^{(0)}\psi\,dx,
\end{aligned}
\]
where \(\chi_u,\chi_v\in L^{p'}\) with \(p'=p/(p-1)\) are the weak limits of \(|\nabla u_n|^{p-2}\nabla u_n\) and \(|\nabla v_n|^{p-2}\nabla v_n\), respectively, and
\[
\begin{aligned}
R_u^{(0)}&=\frac{p-2}{p}|u_0|^{p-4}u_0v_0^2\log v_0^2
+\frac{2}{p}|v_0|^{p-2}u_0\log u_0^2
+\frac{2}{p}|v_0|^{p-2}u_0,\\
R_v^{(0)}&=\frac{p-2}{p}|v_0|^{p-4}v_0u_0^2\log u_0^2
+\frac{2}{p}|u_0|^{p-2}v_0\log v_0^2
+\frac{2}{p}|u_0|^{p-2}v_0.
\end{aligned}
\]
Here \(\langle R_u^{(0)},\phi\rangle\) denotes \(\int_{\mathbb R^N}R_u^{(0)}\phi\,dx\). This limiting process is similar as Step 2, and we omit the details for brevity.

We first verify that \(R_u^{(0)}(u_0)=\langle R_u^{(0)},u_0\rangle\) is finite. Since \(S_n=\|u_n\|_{\mathcal {G}_a}^p+\|v_n\|_{\mathcal {G}_b}^p\) is bounded and we have already proved \(L_n^+\to L_0^+<\infty\) and \(M_{\varepsilon_n}\to M_0<\infty\), the regularised Nehari identity
\[
S_n=L_n^+-L_n^-+\frac{2}{p}R_{\varepsilon_n}
\]
immediately gives \(\sup\limits_n L_n^-<\infty\). By Fatou's lemma,
\[
L_0^-\le\liminf_{n\to\infty}L_n^-<\infty.
\]
Thus the integrals defining \(R_u^{(0)}(u_0)\) and \(R_v^{(0)}(v_0)\) are finite.

We now identify \(\chi_u\); the \(v\)-component is completely analogous. Testing the first regularised equation with \(u_n\), we obtain
\[
\int_{\mathbb R^N}|\nabla u_n|^p\,dx+\int_{\mathbb R^N}a|u_n|^p\,dx
=\int_{\mathbb R^N}R_u^{(n)}u_n\,dx,
\]
where
\[
R_u^{(n)}
=\frac{p-2}{p}|u_n|^{p-4}u_n v_n^2\log(v_n^2+\varepsilon_n)
+\frac{2}{p}|v_n|^{p-2}u_n\log(u_n^2+\varepsilon_n)
+\frac{2}{p}|v_n|^{p-2}\frac{u_n^3}{u_n^2+\varepsilon_n}.
\]
Splitting the right-hand side into positive logarithmic part, negative logarithmic part, and rational part, and using the global convergence established above together with Fatou's lemma for the negative part, we obtain
\[
\limsup_{n\to\infty}\int_{\mathbb R^N}R_u^{(n)}u_n\,dx
\le R_u^{(0)}(u_0).
\]
Moreover, by weak lower semicontinuity,
\[
\liminf_{n\to\infty}\int_{\mathbb R^N}a|u_n|^p\,dx
\ge\int_{\mathbb R^N}a|u_0|^p\,dx.
\]
Therefore
\[
\limsup_{n\to\infty}\int_{\mathbb R^N}|\nabla u_n|^p\,dx
\le R_u^{(0)}(u_0)-\int_{\mathbb R^N}a|u_0|^p\,dx.
\]
We now show that the right-hand side equals \(\int_{\mathbb R^N}\chi_u\cdot\nabla u_0\,dx\). By Step~2, $R_u^{(0)}\in L_{\mathrm{loc}}^{p'}(\mathbb R^N)$.
Since $a$ is locally bounded, every term in the limiting
$u$-equation is continuous with respect to
$W^{1,p}$ convergence on a fixed bounded domain.
Thus, by approximation with smooth compactly supported
functions, we may test this equation with $\eta_Ru_0$,
where $\eta_R\in C_c^\infty(\mathbb R^N)$ satisfies \[ 0\le\eta_R\le1,\qquad \eta_R\equiv1\text{ on }B_R,\qquad \eta_R\equiv0\text{ outside }B_{2R}, \]\[|\nabla\eta_R|\le C/R.\] 
The boundary term satisfies
\[
\left|\int_{\mathbb R^N}
u_0\chi_u\cdot\nabla\eta_R\,dx\right|
\le \frac{C}{R}
\|\chi_u\|_{L^{p'}(B_{2R}\setminus B_R)}
\|u_0\|_{L^p(B_{2R}\setminus B_R)}
\longrightarrow0,
\] the
dominated convergence theorem yields
\[
\lim_{R\to\infty}\int_{\mathbb{R}^N}\eta_R R_u^{(0)}u_0\,dx
=\int_{\mathbb{R}^N}R_u^{(0)}u_0\,dx.
\]
And is thus we obtain
\[
\int_{\mathbb R^N}\chi_u\cdot\nabla u_0\,dx+\int_{\mathbb R^N}a|u_0|^p\,dx
=R_u^{(0)}(u_0).
\]
Thus
\[
\limsup_{n\to\infty}\int_{\mathbb R^N}|\nabla u_n|^p\,dx
\le \int_{\mathbb R^N}\chi_u\cdot\nabla u_0\,dx.
\]

We claim that if \(u_n\rightharpoonup u_0\) in \(W^{1,N}(\mathbb R^N)\), \(|\nabla u_n|^{p-2}\nabla u_n\rightharpoonup\chi_u\) in \(L^{N'}(\mathbb R^N)\), and
\[
\limsup_{n\to\infty}\int_{\mathbb R^N}|\nabla u_n|^p\,dx
\le\int_{\mathbb R^N}\chi_u\cdot\nabla u_0\,dx,
\]
then
\[
\chi_u=|\nabla u_0|^{p-2}\nabla u_0.
\]
Indeed, set $A(\xi):=|\xi|^{p-2}\xi$.
By the weak convergences and the assumed upper bound,
\[
\begin{aligned}
&\limsup_{n\to\infty}
\int_{\mathbb R^N}
\bigl(A(\nabla u_n)-A(\nabla u_0)\bigr)
\cdot(\nabla u_n-\nabla u_0)\,dx\\
&\qquad=
\limsup_{n\to\infty}
\int_{\mathbb R^N}|\nabla u_n|^p\,dx
-\int_{\mathbb R^N}\chi_u\cdot\nabla u_0\,dx
\le0.
\end{aligned}
\]
Since $p>2$, uniform monotonicity gives
\[
2^{2-p}\|\nabla u_n-\nabla u_0\|_{L^p}^p
\le
\int_{\mathbb R^N}
\bigl(A(\nabla u_n)-A(\nabla u_0)\bigr)
\cdot(\nabla u_n-\nabla u_0)\,dx.
\]
Hence $\nabla u_n\to\nabla u_0$ strongly in $L^p$.
By continuity of $A:L^p\to L^{p'}$ and uniqueness
of the weak limit, we conclude that
\[
\chi_u=|\nabla u_0|^{p-2}\nabla u_0.
\]
Similarly,
\[
\chi_v=|\nabla v_0|^{p-2}\nabla v_0.
\]
Therefore the limit equation is exactly the weak form of the original system, so \((u_0,v_0)\) is a weak solution of the original system.

\medskip
\noindent\textbf{Step 5.}
From Step 4, we can deduce the original Nehari identity. Indeed, 
Let \(\eta_R\in C_c^\infty(B_{2R})\) satisfy \(0\le\eta_R\le1\), \(\eta_R\equiv1\) on \(B_R\), \(\eta_R\equiv0\) outside \(B_{2R}\), and \(|\nabla\eta_R|\le C/R\). Using \((\eta_Ru_0,\eta_Rv_0)\) as test function in the weak formulation of the original system, adding the two equations, and letting \(R\to\infty\), the boundary terms vanish and the remaining terms converge by the dominated convergence theorem. We obtain
\[
\|u_0\|_{\mathcal {G}_a}^p+\|v_0\|_{\mathcal {G}_b}^p
=L_0^+-L_0^-+\frac{2}{p}M_0,
\]
where \(L_0^\pm\) denote the positive and negative parts of the logarithmic integral at \((u_0,v_0)\). Since the limit process is analogous as Step 4, we therefore omit the details of these adaptations. From now on, we write \((u_0,v_0)\in \mathcal{N}_0\) if \((u_0,v_0)\) satisfies the original Nehari identity and \((u_0,v_0)\in \mathcal{D}(J)\setminus\{(0,0)\}\).

\medskip
\noindent\textbf{Step 6.}
The regularised Nehari identity is
\[
S_n:=\|u_n\|_{\mathcal {G}_a}^p + \|v_n\|_{\mathcal {G}_b}^p
=L_n^+-L_n^-+\frac{2}{p}R_{\varepsilon_n}.
\]
Since \(L_n^+\to L_0^+\) and \(M_{\varepsilon_n}\to M_0\), we have
\[
S_n=L_0^+-L_n^-+\frac{2}{p}M_0+o(1).
\]
Taking the limit inferior and using the weak lower semicontinuity of the norm,
\[
S_0:=\|u_0\|_{\mathcal {G}_a}^p + \|v_0\|_{\mathcal {G}_b}^p
\le\liminf_{n\to\infty}S_n.
\]
Combining this with the original Nehari identity \(S_0=L_0^+-L_0^-+\frac{2}{p}M_0\), we obtain
\[
\limsup_{n\to\infty}L_n^-\le L_0^-.
\]
On the other hand, by Fatou's lemma,
\[
L_0^-\le\liminf_{n\to\infty}L_n^-.
\]
Thus
\[
\lim_{n\to\infty}L_n^-=L_0^-.
\]
Consequently,
\[
\lim_{n\to\infty}S_n
=L_0^+-L_0^-+\frac{2}{p}M_0
=S_0.
\]
Hence
$$
\|u_n\|_{\mathcal{G}_a}^p+\|v_n\|_{\mathcal{G}_b}^p\to\|u_0\|_{\mathcal{G}_a}^p+\|v_0\|_{\mathcal{G}_b}^p
$$
as \(n\to\infty\). On the other hand, \((u_n,v_n)\rightharpoonup(u_0,v_0)\), \[\|u_0\|_{\mathcal{G}_a}^p\le \liminf\limits_{n\to\infty}\|u_n\|_{\mathcal{G}_a}^p, \quad\|v_0\|_{\mathcal{G}_b}^p\le \liminf\limits_{n\to\infty}\|v_n\|_{\mathcal{G}_b}^p,\]
which yields \[\lim\limits_{n\to\infty}\|u_n\|_{\mathcal{G}_a}=\|u_0\|_{\mathcal{G}_a},\quad\lim\limits_{n\to\infty}\|v_n\|_{\mathcal{G}_b}=\|v_0\|_{\mathcal{G}_b}.\]
Since \(\mathcal{G}_a\) and \(\mathcal{G}_b\) are uniformly convex, they satisfy the Radon-Riesz property; therefore, the weak convergence together with the convergence of the norms implies \(u_n\) strongly converges to \(u_0\) in \(\mathcal{G}_a\) and \(v_n\) strongly converges to \(v_0\) in \(\mathcal{G}_b\), that is,
\[
(u_n,v_n)\to(u_0,v_0)\quad\text{strongly in }\mathcal G.
\]

\medskip
\noindent\textbf{Step 7.}
The regularised energy identity is
\[
J_{\varepsilon_n}(u_n,v_n)
=\frac{2}{p^2}R_{\varepsilon_n}(u_n,v_n).
\]
Since \(R_{\varepsilon_n}\to M_0\), and by the original Nehari identity \(J(u_0,v_0)=\frac{2}{p^2}M_0\), we obtain
\[
J_{\varepsilon_n}(u_n,v_n)\to J(u_0,v_0).
\]

\medskip
\noindent\textbf{Step 8. }
Let \((w,z)\in\mathcal{N}_0\), by the original Nehari identity,
\[
S(w,z)
=L(w,z)+\frac{2}{p}M(w,z),
\qquad
J(w,z)=\frac{2}{p^2}M(w,z)>0.
\]
For each \(\varepsilon>0\), the fibre map \(t\mapsto J_\varepsilon(t(w,z))\) has a unique maximum point \(t_\varepsilon>0\), and \(t_\varepsilon(w,z)\in\mathcal N_\varepsilon\). Define \(G_\varepsilon(t)\) as the right-hand side of the fibre derivative, that is,
\begin{eqnarray*}
G_{\varepsilon}(t)
&=
\int_{\mathbb R^N}|w|^{p-2}z^2\log(t^2z^2+\varepsilon)\,dx
+
\int_{\mathbb R^N}|z|^{p-2}w^2\log(t^2w^2+\varepsilon)\,dx
\\
&\quad
+\frac{2}{p}
\left(
\int_{\mathbb R^N}|w|^{p-2}\frac{t^2z^4}{t^2z^2+\varepsilon}\,dx
+
\int_{\mathbb R^N}|z|^{p-2}\frac{t^2w^4}{t^2w^2+\varepsilon}\,dx
\right),
\end{eqnarray*}
It is easy to check that the function \(G_{\varepsilon}(t)\) and \(G_0(t)\) are strictly increasing on \((0,\infty)\) and \(G_0(1)=\|w\|_{\mathcal {G}_a}^p+\|z\|_{\mathcal {G}_b}^p\). Thus for any \(\delta>0\),
\[
G_0(1+\delta)>\|w\|_{\mathcal {G}_a}^p+\|z\|_{\mathcal {G}_b}^p.
\]
Moreover, since \((w,z)\in\mathcal D(J)\) and the logarithmic and rational terms are continuous with respect to \(\varepsilon\) on compact intervals, we have
\[
G_\varepsilon(t)\to G_0(t)
\]
uniformly on any compact interval bounded away from \(0\), which follows by finding the dominated function and dominated convergence theorem. Since the argument is analogous as below, we omit the details for brevity. Therefore, for all sufficiently small \(\varepsilon\),
\[
G_\varepsilon(1+\delta)>\|w\|_{\mathcal {G}_a}^p+\|z\|_{\mathcal {G}_b}^p.
\]
Since \(G_\varepsilon\) is strictly increasing and \(G_\varepsilon(t_\varepsilon)=\|w\|_{\mathcal {G}_a}^p+\|z\|_{\mathcal {G}_b}^p\), we conclude \(t_\varepsilon<1+\delta\). Similarly, \(t_\varepsilon>1-\delta\). Hence \(t_\varepsilon\to1\) as \(\varepsilon\to0\).
Because \(t_\varepsilon\to1\), there exists \(c>0\) such that \(t_\varepsilon\ge c\) for all sufficiently small \(\varepsilon\). In the subsequent control function estimates, we may therefore assume \(t_n\ge c>0\).

Now define
\[
F_n(x):=t_n^p|w(x)|^{p-2}z(x)^2\log(t_n^2z(x)^2+\varepsilon_n).
\]
Choose \(\delta_0>0\) sufficiently small so that
\[
C^2\delta_0^2+\varepsilon_0<1,
\]
where \(\varepsilon_0\) is an upper bound for all \(\varepsilon_n\); such \(\delta_0\) exists because \(\varepsilon_n\to0\). Decompose \(\mathbb R^N\) into
\[
\Omega^-=\{|z|<\delta_0\},\qquad \Omega^+=\{|z|\ge\delta_0\}.
\]
On $\Omega^-$,
\[
|\log(t_n^2z^2+\varepsilon_n)|
\le |\log z^2|+C,
\]
and hence
\[
|F_n(x)|
\le C|w|^{p-2}z^2\bigl(|\log z^2|+1\bigr).
\]
The original Nehari identity
\[
L^-(w,z)=L^+(w,z)+\frac{2}{p}M(w,z)-S(w,z)
\]
and the finiteness of the terms on the right imply $L^-(w,z)<\infty$.
Thus
\[
|w|^{p-2}z^2|\log z^2|\in L^1(\mathbb R^N).
\]

On $\Omega^+$, we have $|z|\ge\delta_0>0$ and
$t_n^2z^2+\varepsilon_n\ge c^2\delta_0^2=:m>0$.
Using $|\log x|\le C_{m,\sigma}(1+x^\sigma)$ for $x\ge m$, we obtain
\[
|\log(t_n^2z^2+\varepsilon_n)|
\le C(1+|z|^{2\sigma}).
\]
Combining both regions gives
\[
|F_n(x)|
\le C|w|^{p-2}z^2
\bigl(|\log z^2|+1+|z|^{2\sigma}\bigr)
\in L^1(\mathbb R^N),
\]
where integrability follows from the preceding observation,
H\"older's inequality, and the Sobolev embeddings.

Similarly, for
\[
G_n(x):=t_n^p|z|^{p-2}w^2\log(t_n^2w^2+\varepsilon_n),
\]
we have the integrable bound
\[
|G_n(x)|
\le C|z|^{p-2}w^2
\bigl(|\log w^2|+1+|w|^{2\sigma}\bigr).
\]
Thus \(t_\varepsilon(w,z)\in\mathcal N_\varepsilon\). Since \((u_\varepsilon,v_\varepsilon)\) is a regularised ground state,
\[
J_{\varepsilon_n}(u_n,v_n)\le J_{\varepsilon_n}(t_{\varepsilon_n}(w,z)).
\]
Letting \(n\to\infty\) and using the energy convergence established in Step 7, we get
\[
J(u_0,v_0)\le J(w,z).
\]
As \((w,z)\) is arbitrary, \((u_0,v_0)\) minimises \(J\) among all nontrivial ground state solutions and thus \((u_0,v_0)\) is a ground state of the original equation. This completes the proof.
\end{proof}

\section{Appendix} 
\label{sec:appendix}

\qquad The following example shows that under {\rm(A1)}--{\rm(A2)} the logarithmic coupling integrals may diverge; hence {\rm(A2)} alone does not guarantee the finiteness of the energy functional. This justifies the use of the stronger condition {\rm(A2$'$)} as the natural framework for our variational analysis.

\begin{example}\label{exam8.1}
Consider a locally finite graph $G=(V,E)$ with $V=\mathbb{N}$ and
$x_{n}=n$ for $n\ge0$. Define
\[
u(x_{n})=\begin{cases}
\dfrac{1}{n^{\frac{2}{p}}(\log n)^{\theta}}, & n\ge3,\\[12pt]
0, & 0\le n\le2,
\end{cases}\qquad
v(x_{n})=\begin{cases}
\dfrac{1}{n^{\frac{2}{p}}(\log n)^{\phi}}, & n\ge3,\\[12pt]
0, & 0\le n\le2.
\end{cases}
\]
Take the measure and potentials as
\[
\mu(x_{n})=\begin{cases}
n, & n\ge3,\\[4pt]
1, & 0\le n\le2,
\end{cases}\qquad
a(x_{n})=\begin{cases}
(\log n)^{\delta}, & n\ge3,\\[4pt]
1, & 0\le n\le2,
\end{cases}\qquad
b(x_{n})=\begin{cases}
(\log n)^{\delta}, & n\ge3,\\[4pt]
1, & 0\le n\le2,
\end{cases}
\]
with $\delta>0$ sufficiently small. Direct calculations show that $u,v\in W^{1,p}(V)$ and
\[
\int_{V}a(x)|u(x)|^{p}\,d\mu<+\infty,\qquad
\int_{V}b(x)|v(x)|^{p}\,d\mu<+\infty,
\]
provided $p\theta>1+\delta$ and $p\phi>1+\delta$. 
On the other hand,
\begingroup
\allowdisplaybreaks
\begin{align*}
    \int_V |\nabla u|^p d\mu &= \sum_{n\ge 2} \mu(x_n)|\nabla u(x_n)|^p
            = \sum_{x\ge 2} \frac{1}{2^{\frac{p}{2}}\mu(x)^{\frac{p}{2}-1}}\left(\sum_{y\sim x}(u(y)-u(x))^2\right)^{\frac{p}{2}}\\
            &\le C\sum_{x\ge 2} \left(\sum_{y\sim x}(u(y)-u(x))^2\right)^{\frac{p}{2}}
               \le C\sum_{x\ge 2}\left[\left(u(x+1)-u(x)\right)^2+\left(u(x-1)-u(x)\right)^2\right]^{\frac{p}{2}}\\
            &\le C\sum_{x\ge 2}\left[\left|u(x+1)-u(x)\right|^p+\left|u(x-1)-u(x)\right|^p\right]\\
            &= C|u(3)|^p+C\sum_{x\ge 3}\left[\left|u(x+1)-u(x)\right|^p+\left|u(x-1)-u(x)\right|^p\right]\\
            &\le C|u(3)|^p+C\sum_{x\ge 3} |u(x)|^p<+\infty.
\end{align*}
\endgroup
        Similarly, we obtain $\int_V |\nabla v|^p d\mu<+\infty$.

However, the
logarithmic coupling terms diverge:
\[
\int_{V}|v(x)|^{p-2}|u(x)|^{2}\log|u(x)|^{2}\,d\mu
   \le -\frac{4}{p}\sum_{n\ge3}n^{-1}(\log n)^{-(p-2)\phi-2\theta+1}
   =-\infty,
\]
when $(p-2)\phi+2\theta-1\le1$, and similarly

\[
\int_{V}|u(x)|^{p-2}|v(x)|^{2}\log|v(x)|^{2}\,d\mu=-\infty,
\]
\vspace{4pt}
if $(p-2)\theta+2\phi-1\le1$. The system
\[
p\theta>1+\delta,\quad p\phi>1+\delta,\quad
(p-2)\phi+2\theta-1\le1,\quad (p-2)\theta+2\phi-1\le1.
\]
admits a solution $(\theta,\phi)$ for any fixed $p>4$ with $\delta$
small enough. Hence $J(u,v)$ fails to be well-defined under
{\rm(A1)}--{\rm(A2)} alone.
\end{example}

\begin{example}\label{ex:ill-posedness}
This example serves to identify the ill-posedness of the original functional under the hypotheses of Theorem~\ref{thm2}. Define
\[
a(x)=b(x)=2+\log\log(e+|x|).
\]
Then
\[
\inf a=\inf b=2>0,\qquad
\lim_{|x|\to\infty}a(x)=\lim_{|x|\to\infty}b(x)=+\infty,
\]
so the hypotheses are satisfied. Choose
\[
u(x)=v(x)=\frac{1}{(1+|x|)\left[\log(e+|x|)\right]^{\theta}},
\]
where
\[
\frac{1}{N}<\theta<\frac{2}{N},\qquad N> 4.
\]
This interval is nonempty, so such a choice is admissible.

Since \(N\theta>1\), we have
\[
\int_{\mathbb R^N}|u|^N\,dx
\sim \int^\infty r^{-1}(\log r)^{-N\theta}\,dr<\infty.
\]
Moreover,
\[
|\nabla u|^N\sim r^{-2N}(\log r)^{-N\theta},
\]
and hence
\[
\int_{\mathbb R^N}|\nabla u|^N\,dx
\sim \int^\infty r^{-N-1}(\log r)^{-N\theta}\,dr<\infty.
\]
Finally,
\[
a|u|^N\sim (\log\log r)\,r^{-N}(\log r)^{-N\theta}
\le C\,r^{-N}(\log r)^{-N\theta+\delta}
\]
for any \(\delta>0\). Choosing \(\delta\) sufficiently small so that \(N\theta-\delta>1\), we obtain
\[
\int_{\mathbb R^N}a|u|^N\,dx
\le C\int^\infty r^{-1}(\log r)^{-N\theta+\delta}\,dr<\infty.
\]
Thus \(u=v\in\mathcal G_a=\mathcal G_b\).

Since \(u=v\), we have
\[
|u|^{N-2}v^2=|u|^N.
\]
Therefore
\[
|u|^{N-2}v^2\log v^2
\sim |u|^N\log |u|^2
\sim -2r^{-N}(\log r)^{1-N\theta}.
\]
Consequently,
\[
\int_{\mathbb R^N}|u|^{N-2}v^2\log v^2\,dx
\sim -2\int^\infty r^{-1}(\log r)^{1-N\theta}\,dr.
\]
Because \(N\theta<2\), the exponent \(1-N\theta>-1\), and the integral diverges to \(-\infty\). The symmetric term diverges similarly.

The principal part is finite, while the negative part of the logarithmic term diverges to \(-\infty\). Hence
\[
J(u,v)=+\infty.
\]
Thus the original functional \(J\) is not finite everywhere under the hypotheses of Theorem~\ref{thm2}; in other words, \(J\) is ill-posed.
\end{example}
\subsection*{Open problems}
The methods and ideas developed here are applicable to a broader class of slowly growing, critically coupled systems. Besides, several questions remain open, and we briefly outline them below.

\begin{enumerate}
\item \textbf{The range \(2<p\le4\).}
The analysis of this paper relies on the condition \(p>4\).
Extending it to \(2<p\le4\) requires re-examining the regularity
of the mixed nonlinearities and the estimates used in the proofs.
Whether the full variational programme can be extended to this
lower regime remains an open challenge.

\item \textbf{Extension to Riemannian manifolds and bounded Euclidean domains with boundary conditions.}
Extending the present variational and rigidity framework to Riemannian manifolds or bounded Euclidean domains with boundary conditions would require substantial new ideas, since the current estimates rely heavily on the Euclidean and discrete-graph structure; curvature, boundary effects, and the loss of Fourier-analytic tools call for new compactness and refined calibration arguments.

\item \textbf{Hidden structures beyond the rigidity theorem.}
The Hessian factorisation in Theorem~\ref{thm5} cleanly separates the tangent-space dynamics from a rank-one correction forced by the Nehari constraint. This algebraic decomposition raises at least two natural directions: whether it reflects deeper structural properties of the system, and whether a similar rigidity framework can be developed for discrete analogues of the present PDE system.

\item \textbf{Other nonlinearities.}
The logarithmic coupling studied here is a prototype of
nonlinearities that grow slower than any positive power yet are
not uniformly bounded.  It would be interesting to extend the
techniques developed in this paper to related problems involving,
for example, double‑logarithmic terms, fractional powers, or
nonlocal couplings on discrete structures.

\item \textbf{Stability under potential perturbations.}
Theorem~\ref{thm5} identifies degeneracy at \(\mathcal R(u_0)=1\).
Open questions concern its relation to bifurcation or families of
minimisers, branch stability under arbitrary potential perturbations,
and the predictive role of the Hessian factorisation in symmetry
breaking and phase transitions.
\end{enumerate}
\noindent Last but not least, a further research direction is to investigate whether the methods developed here can be extended to fractional \(p\)-Laplacian systems and free-boundary problems with logarithmic coupling. These settings present additional challenges associated with nonlocal interactions and free-boundary regularity.

\section*{\fontsize{11}{12}\selectfont\bfseries Acknowledgements}
\vspace{-0.1cm}
I would like to thank my supervisor Professor Genggeng Huang for helpful discussions. Moreover, I would like to take this opportunity to thank Professor Xuexiu Zhong for his careful reading and valuable comments on this paper.
\vspace{-0.3cm}
\section*{\fontsize{11}{12}\selectfont\bfseries Author Contributions}
\vspace{-0.1cm}
As the sole author, I was responsible for the research design, the mathematical analysis, the proofs, the numerical verification, and the preparation of the manuscript. 
\vspace{-0.3cm}
\section*{\fontsize{11}{12}\selectfont\bfseries Funding}
\vspace{-0.1cm}
This research did not receive any specific grant from funding agencies in the public, commercial, or not-for-profit sectors.
\vspace{-0.3cm}
\section*{\fontsize{11}{12}\selectfont\bfseries Data Availability}
\vspace{-0.1cm}
Data sharing not applicable to this article as no datasets were generated or analysed during the current study.
\vspace{-0.7cm}
\section*{\fontsize{11}{12}\selectfont\bfseries Declarations}
\vspace{-0.1cm}
The author confirms that all ideas and proofs in this paper are the sole work of the author. AI (DeepSeek) was used exclusively for language polishing to improve the clarity of the presentation.
\vspace{-0.3cm}
\section*{\fontsize{11}{12}\selectfont\bfseries Competing Interests}
\vspace{-0.1cm}
The author declares no competing interests.

\end{document}